\documentclass[10pt]{article} 
\usepackage{amsthm}
\theoremstyle{definition}
\newtheorem{defi}{Definition}[section]
\newtheorem{remark}[defi]{Remark}

\newtheorem{example}[defi]{Example}

\usepackage[utf8x]{inputenc}
\theoremstyle{plain}
\newtheorem{theorem}[defi]{Theorem}
\newtheorem{cor}[defi]{Corollary}
\newtheorem{lemma}[defi]{Lemma}
\newtheorem{prop}[defi]{Proposition}

\newtheorem{asum}[defi]{Assumption}
\usepackage{cite}
\usepackage{amsmath, amsfonts}
\usepackage[english]{babel}
\usepackage{enumerate}
\usepackage [T1]{fontenc}
\usepackage{endnotes}
\usepackage{xcolor}
\usepackage{abstract}
\usepackage{enumitem,blindtext}
\usepackage[caption=false]{subfig} % <- Preamble
\usepackage[paper=a4paper,left=30mm,right=30mm,top=30mm,bottom=30mm]{geometry}
\usepackage{graphicx,epstopdf} 
\usepackage{algorithmic}

\usepackage{hyperref}
\ifpdf
  \DeclareGraphicsExtensions{.eps,.pdf,.png,.jpg}
\else
  \DeclareGraphicsExtensions{.eps}
\fi

\numberwithin{equation}{section}

\newcommand{\red}[1]{\textcolor{black}{#1}}
\newcommand{\black}[1]{\textcolor{black}{#1}}

\newcommand{\blue}[1]{\textcolor{black}{#1}}
\newcommand{\violet}[1]{\textcolor{black}{#1}}

\newcommand*\samethanks[1][\value{footnote}]{\footnotemark[#1]}

\title{Generalized Feynman-Kac Formula under volatility uncertainty}
\author{Bahar Akhtari\thanks{Department of Mathematical Sciences, Shahid Beheshti University, Tehran, Iran.  E.mail: b$\_$akhtari@sbu.ac.ir.} \and Francesca Biagini\thanks{Workgroup Financial and Insurance Mathematics, Department of Mathematics, Ludwig-Maximilians Universit{\"a}t, Theresienstrasse 39, 80333 Munich, Germany. Emails: mazzon@math.lmu.de, biagini@math.lmu.de.} \and Andrea Mazzon\samethanks[2] \and Katharina Oberpriller \thanks{University of Freiburg, Ernst-Zermelo-Strasse 1, 79104 Freiburg, Germany. Email: katharina.oberpriller@stochastik.uni-freiburg.de}}
\begin{document}
\maketitle
\begin{abstract}
In this paper we provide a generalization of \blue{a Feynmac-Kac formula} under volatility uncertainty in presence \blue{of a linear term in the PDE due to}  discounting. We state our result under different hypothesis with respect to the derivation given by Hu, Ji, Peng and Song (Comparison theorem, {F}eynman-{K}ac formula and {G}irsanov transformation for {B}{S}{D}{E}s driven by {G}-{B}rownian motion, \emph{Stochastic Processes and their Application}, 124 (2)), where the Lipschitz continuity of some functionals is assumed which is not necessarily satisfied in our setting. In particular, we \blue{show that} the $G$-conditional expectation of a discounted payoff \blue{is a viscosity solution of a nonlinear PDE}. In applications, this permits to \blue{calculate} such a sublinear expectation in a computationally efficient way. 
\end{abstract}


\section{Introduction}

In this paper we provide a \blue{Feynman-Kac} formula under volatility uncertainty to include an additional linear term \blue{in the associated PDE} due to discounting. The presence of the linear term in the PDE prevents the conditions under which a \blue{Feynman-Kac} formula in the $G$-setting is \blue{proved} in \cite{hu_ji_peng_song_2014}. Up to our knowledge, there are no contributions in the literature where such a relation between $G$-conditional expectation and solutions of (nonlinear) PDEs is given in a setting not satisfying the hypothesis of \cite{hu_ji_peng_song_2014}.
Here, we establish for the first time a relation between nonlinear PDEs and $G$-conditional expectation of a \emph{discounted} payoff.
 We introduce a family of fully nonlinear PDEs identified by a regularizing parameter with terminal condition $\varphi$ at time $T>0$, and obtain the $G$-conditional expectation of a discounted payoff as the limit of the solutions of such a family of PDEs when the regularity parameter goes to zero. \blue{Using a stability result, we can prove that such a limit is a viscosity solution of the limit PDE. Therefore, we are able to show that the $G$-conditional expectation of the discounted payoff is a solution of  the PDE. }

More specifically we consider the $G$-setting introduced in \violet{\cite{peng_nonlinearExpectation_book}}, a $G$-It\^{o} process $X=(X_t)_{t \in [0,T]}$ and the $G$-conditional expectation
\begin{equation}\label{eq:introduction}
\hat{\mathbb{E}}\left[\varphi(T,X_T)e^{-\int_r^T X_sds}\big|\blue{X_r=x}\right], 
\end{equation}
\blue{for any $(r,x) \in [0,T] \times (0, \infty)$.} Here $\varphi$ is a real-valued function, \blue{supposed to  be continuous with polynomial growth}, which may represent the payoff of an interest rate \blue{derivative or of a life insurance liability} in applications.

In the classical framework, the well known \blue{Feynman-Kac} representation theorem establishes -under some integrability conditions on their coefficients- a connection between SDEs and parabolic PDEs by the formula
\begin{equation}
u(t,y)=\mathbb{E}\left[f(\red{T},Y_T)e^{-\int_t^TY_sds}|Y_t=y\right], \quad 0 \le t \le T, \label{eq:Introduction3}
\end{equation}
where $f$ is a given payoff function, $Y$ is an It\^o process solution to a given SDE and $u$ is the unique classical solution to a parabolic PDE associated to the SDE solved by $Y$, with final condition $f$.

An analogous of  \blue{a Feynman-Kac} formula within the framework of $G$-expectation is given in Theorem 4.5 of \cite{hu_ji_peng_song_2014}. This result connects the solution of a $G$-BSDE to the unique \emph{viscosity} solution of a fully nonlinear PDE, whose nonlinearity is given by a term that depends on the function $G$ representing the uncertainty about the volatility. The authors are able to deal with viscosity (i.e., non necessarily classical) solutions to the associated PDEs assuming the Lipschitz continuity of the coefficients of the $G$-BSDE. However, due to the presence of the discounting term $e^{-\int_0^T X_sds}$ in \eqref{eq:introduction}, these conditions are not satisfied in our framework, so that we cannot apply this formula directly.

In the setting of short rate models under a fixed probability measure, a \blue{Feynman-Kac} formula is proposed \blue{in \cite{f_2009}} which derives \eqref{eq:Introduction3} starting from the  PDE associated to the SDE solved by $Y$, and proving that its solution can be seen as the conditional expectation of the discounted payoff. On the contrary in \cite{hu_ji_peng_song_2014}, the authors start from a $G$-conditional expectation and prove that this solves a fully nonlinear PDE. In the present work, we aim to extend the approach in \cite{f_2009} to the $G$-setting. 
 
In particular, we assume that $X$ \blue{is quasi-surely strictly positive and }solves a given $G$-SDE, with Lipschitz coefficients. We then consider a fully nonlinear, non-degenerate parabolic PDE associated to this \blue{$G$}-SDE. \blue{This PDE contains a linear term due to discounting, which requires to consider a quasi-surely strictly positive process $X$ in order to apply the associated PDE theory from \cite{krylov}.} In general such a PDE does not admit a classical $C^{1,2}$ solution. For this reason, we substitute the coefficients  of the original PDE with some bounded and $C^{1,2}$ cut-off functions depending on a parameter $\epsilon \in (0,1)$. In this way, we are able to show that the approximated PDE admits a \blue{viscosity} solution $u^{\epsilon}$ \blue{defined on an unbounded domain, which is additionally $C^{1,2}$ inside a bounded domain $D_{\epsilon}$ depending on $\epsilon$. We can then apply} the $G$-version of the It\^o's formula to prove the martingale property of the stopped process $M^{\epsilon}=(M_t^{\epsilon})_{t \in [0,T]}$ defined by $M_t^{\epsilon}:=u^{\epsilon}(t \wedge \tau_{\epsilon}, X_{t \wedge \tau_{\epsilon}})e^{-\int_0^{t \wedge \tau_{\epsilon}}(X_s + \epsilon) ds}$ for every $\epsilon \in(0,1)$, where the stopping time $\tau_{\epsilon}$ is \blue{the exit time from $D_{\epsilon}$}. We then \blue{show} that the sublinear conditional expectation in $(\ref{eq:introduction})$ can be obtained as the limit \blue{$u(r,x)$} of \blue{$u^{\epsilon}(r,x)$ for $\epsilon$ going to zero for any $(r,x) \in [0,T] \times (0, \infty)$}, see Theorem \ref{theorem:PointwiseConvergence}. This is \blue{obtained} by using a probabilistic approach exploiting the martingale property of $M^{\epsilon}$, $\blue{\epsilon \in (0,1)}$, together with some properties of the family of stopping times $(\tau_{\epsilon})_{\epsilon \blue{\in ( 0,1)}}$. \blue{Using a stability result, we are then able to prove that $u$ is a viscosity solution of the original PDE, getting the representation
\begin{equation}\label{eq:representationforphibounded}
u(r,x)=\hat{\mathbb{E}}\left[\varphi(T,X_T)e^{-\int_r^T X_sds}\big|X_r=x\right], 
\end{equation}
for any $(r,x) \in [0,T] \times (0, \infty)$. For this first analysis we assume $\varphi$ to be bounded, then we extend these results to the case when $\varphi$ has polynomial growth. In particular, we \blue{consider} a family of bounded functions $(\varphi_{\epsilon})_{\epsilon \in (0,1)}$ approximating $\varphi$, use a version of the dominate convergence theorem for the $G$-setting and employ a stability result in order to get Theorems \ref{thm:firstresultforunboundedpayoff} and \ref{thm:secondresultforunboundedpayoff}, which are the main outcomes of the paper.} \red{Note that in general it is not possible to prove uniqueness of the viscosity solution. However, we provide an error estimate in approximating the $G$-expectation by the unique viscosity solution $u^{\epsilon}$ for any $\epsilon>0$ and any bounded $\varphi$. }

These findings are also relevant for financial applications, in particular for the pricing of contingent claims in fixed income markets or in insurance modeling under volatility uncertainty, a topic which has been investigated  in \cite{hoelzermann_2019}, \cite{hoelzermann_2020} and \cite{hoelzermann_quian_2020} in the $G$-setting and in \cite{bo_2020} by considering nonlinear affine processes under model uncertainty which were introduced in \cite{fns_2019}.

In order to perform the analysis illustrated above, we have to use both stochastic calculus in the $G$-setting and the theory of nonlinear PDEs. In particular, one of the main technical difficulties is to prove the existence of a solution to the regularized PDEs \blue{which is also $C^{1,2}$ on a bounded domain depending on the regularizing parameter,} by following \cite{krylov}. For the reader's convenience, we recall some results of \cite{krylov} in Section \ref{sec:krylov} on regularity of solutions of nonlinear PDEs, whereas we devote Section \ref{sec:Gsetting} to an introduction to the $G$-setting and to the theory of stopping times in this framework, see \cite{denis_hu_peng_2010}, \cite{hu_ji_peng_song_2014}, \violet{\cite{peng_nonlinearExpectation_book}}, \cite{soner_touzi_zhang_MRP}, \cite{liu}. Section \ref{sec:results} contains the main contributions of the paper about \blue{a Feynmac-Kac} formula in the $G$-setting, \blue{see Theorems \ref{theorem:StabilityBounded} and \ref{thm:secondresultforunboundedpayoff}.} Here, a consistent effort is devoted to verify that the stopped process $M^{\epsilon}$ is well defined and a $G$-martingale for any $\epsilon \in (0,1)$, and to show that \eqref{eq:introduction} is equal to the limit of solutions of the regularized PDEs when $\epsilon$ goes to zero. 
\blue{It is also crucial to obtain a stability result in order to prove that such a limit is itself a solution of the limit PDE.}
\\
{In Section \ref{sec:numerical} we solve numerically the PDE and plot the approximated value of the $G$-expectation in \eqref{eq:representationforphibounded}, for some choices of the payoff functions and of the coefficients of the $G$-SDE identifying the underlying.} This is also a further contribution since it is well known that the strong law of large numbers in the $G$-setting differs from the one in the classical framework, see Theorem \violet{2.4.1 in \cite{peng_nonlinearExpectation_book}}, and this may cause some issues in the application of Monte-Carlo algorithms \blue{to approximate $G$-expectations}, see for example \cite{Teran_2018}.

\section{$G$-Setting}\label{sec:Gsetting}
We now recall the basic concepts for the $G$-setting based on \violet{\cite{peng_nonlinearExpectation_book}}.
\begin{defi}\label{def:initial}
Let $\Omega$ be a given set and let $\mathcal{H}$ be a vector lattice of $\mathbb{R}$-valued functions defined on $\Omega$ such that $c \in\mathcal{H}$ for all constants $c$, and $\vert X \vert \in \mathcal{H},$ if $X \in \mathcal{H}$. $\mathcal{H}$ is considered as the space of random variables. A sublinear expectation ${\mathbb{E}}$ on $\mathcal{H}$ is a functional ${\mathbb{E}}: \mathcal{H} \to \mathbb{R}$ satisfying the following properties: for all $X,Y \in \mathcal{H}$, we have
\begin{enumerate}
	\item Monotonicity: If $X \geq Y$ then $\mathbb{E}[X] \geq \mathbb{E}[Y]$. 
\item Constant preserving: $\mathbb{E}[c]=c$.
\item Subadditivity: $\mathbb{E}[X+Y] \leq \mathbb{E}[X]+\mathbb{E}[Y]$.
\item Positive homogenity: $\mathbb{E}[\lambda X]= \lambda \mathbb{E}[X]$, $\forall \lambda >0$. 
\end{enumerate}
The triple $(\Omega, \mathcal{H}, \mathbb{E})$ is called a sublinear expectation space. 
\end{defi}
Let $\mathcal{H}$ be a space of random variables such that if $X_i \in \mathcal{H},i=1,...,d$, then 
\begin{equation*}
	\varphi(X_1,...,X_d) \in \mathcal{H}, \quad \text{ for all } \varphi \in C_{\blue{l},lip}(\mathbb{R}^d),
\end{equation*}
where \blue{$C_{l,lip}(\mathbb{R}^d)$ denotes the linear space of functions satisfying for each $x,y \in \mathbb{R}^d$,
\begin{equation}
	\vert \varphi(x)-\varphi(x) \vert \leq C_{\varphi}(1+ \vert x \vert^m + \vert y \vert^m) \vert x-y \vert, 
\end{equation}
for some $C_{\varphi}>0, m \in \mathbb{N}$ depending on $\varphi$.}  
In the following we only consider random-variables with values in $\mathbb{R},$  i.e., we fix $d=1$. 

We now introduce the $G$-Brownian motion by giving the following definitions.
\begin{defi}
 A random variable $X$ is said to be independent from a random variable $Y$ if $\mathbb{E}[\varphi(X,Y)]=\mathbb{E}[\mathbb{E}[\varphi(X,y)]_{y=Y}]$, for $\varphi \in C_{\blue{l},lip}(\mathbb{R} \times \mathbb{R})$. 
 \end{defi}
\begin{defi}
Two random variables $X$ and $Y$ are identical distributed if $\mathbb{E}[\varphi(X)]=\mathbb{E}[\varphi(Y)]$ for all $\varphi \in C_{\blue{l},lip}(\mathbb{R})$. In this case, we write $X \sim Y$.
\end{defi}

\begin{defi}
A random variable $X$ in a sublinear expectation space $(\Omega, \mathcal{H}, \mathbb{E})$ is called $G$-normal distributed if for each $a,b \geq 0$ it holds
\begin{equation*}
	a X + b \overline{X} \sim \sqrt{a^2 + b^2}X, 
\end{equation*}	
where $\overline{X}$ is an independent copy of $X$. The letter $G$ denotes the function $G: \mathbb{R} \to \mathbb{R}$ defined by
\begin{equation}
	G(x):=  \frac{1}{2} \sup_{\sigma \in [\underline{\sigma}^2, \overline{\sigma}^2]} \sigma x=\frac{1}{2}(\overline{\sigma}^2x^+-\underline{\sigma}^2x^-),  	\quad 0 \leq \underline{\sigma}^2 \leq \overline{\sigma}^2. \label{G-equationOneDimension}
\end{equation}
\end{defi}
From now on we assume that there exists $\beta>0$ such that for all $y,\overline{y} \in \mathbb{R}$ with $y \geq \overline{y}$ we have
\begin{equation}
G(y)-G(\overline{y}) \geq \beta [y-\overline{y}]. \label{uniformlyElliptic}
\end{equation}
 Condition \eqref{uniformlyElliptic} is called the uniformly elliptic condition.
	
\begin{defi}
	Let $G: \mathbb{R} \to \mathbb{R}$ be a given monotonic and sublinear function. A stochastic process $B=(B_t)_{t \geq 0}$ on a sublinear expectation space $(\Omega, \mathcal{H}, \mathbb{E})$ is called a $G$-Brownian motion if it satisfies the following conditions:
	\begin{enumerate}
		\item $B_0=0$;
		\item $B_t \in \mathcal{H}$ for each $t \geq 0$;
		\item For each $t, s \geq 0$ the increment $B_{t+s}-B_t$ is independent of $(B_{t_1},...,B_{t_n})$ for each $n \in \mathbb{N}$ and $0 \leq t_1 <...<t_n \leq t$. Moreover, $(B_{t+s}-B_t)s^{-1/2}$ is $G$-normally distributed. 
	\end{enumerate}
\end{defi}
Fix now a time horizon $T>0$ and introduce the space $\Omega_T:=C_0([0,T],\mathbb{R})$ of all $\mathbb{R}$-valued continuous paths $\omega=(\omega_t)_{t \in [0,T]}$ with $\omega_0=0$. We equip this space with the uniform convergence on compact intervals topology and denote by $\mathcal{F}=\mathcal{B}(\Omega_T)$ the Borel $\sigma$-algebra. Moreover, the canonical process $B=(B_t)_{t\in [0,T]}$ on $\Omega_T$ is given by $B_t(\omega):=\omega_t$ for $\omega \in \Omega_T, t \in [0,T]$. For $t \in [0,T]$ we define $\Omega_t:=\lbrace \omega_{\cdot \wedge t}: \omega \in \Omega_T \rbrace$ and $\mathcal{F}_t:=\mathcal{B}(\Omega_t)$. 
\begin{defi}
Introduce $Lip(\Omega_t)$ as
\begin{equation*}
	Lip(\Omega_t):=\lbrace \varphi (B_{t_1},...,B_{t_n}) \vert n \in \mathbb{N}, t_1,...,t_n \in [0,t], \varphi \in C_{l,Lip}(\mathbb{R}^n) \rbrace,
\end{equation*}
for $t \in [0,T]$. Then the space $L_G^p(\Omega_t)$ is the completion of $Lip(\Omega_t)$, $t \ge 0$, under the norm $\| \xi \|_p=(\hat{\mathbb{E}}[\vert \xi \vert^p])^{1/p}$ for $p \geq 1$.
 \end{defi}
Next we introduce the definition of the $G$-expectation and the $G$-conditional expectation. 
\begin{defi}
A $G$-expectation $\hat{\mathbb{E}}$ is a sublinear expectation on $(\Omega,Lip(\Omega_T))$  defined as follows: For $X \in Lip(\Omega_T)$ of the form 
\begin{equation*}
	X=\varphi(B_{t_1}-B_{t_{\violet{0}}},...,B_{t_n}-B_{t_{n-1}}), \quad 0 \leq t_0 < t_1<...<t_n \leq T, \blue{ \varphi \in C_{l,Lip}(\mathbb{R}^n)}\blue{,}
\end{equation*}
we set 
\begin{equation*}
	\hat{\mathbb{E}}[X]:=\mathbb{E}[\varphi(\xi_1 \sqrt{t_1-t_0},...,\xi_n \sqrt{t_n - t_{n-1}})],
\end{equation*}
where $\xi_1,...,\xi_n$ are $1$-dimensional random variables on a sublinear expectation space $(\tilde{\Omega},\tilde{\mathcal{H}}, \mathbb{E})$ such that for each $i=1,...,n$, $\xi_i$ is $G$-normally distributed and independent of $(\xi_1,...,\xi_{i-1})$.
\end{defi}
The corresponding canonical process $B=(B_t)_{t \in [0,T]}$ is a $G$-Brownian motion on the sublinear expectation space $(\Omega_T, Lip(\Omega_T), \hat{\mathbb{E}})$. 
\begin{defi}
	Let $X \in Lip(\Omega_T)$ have the representation
\begin{equation*}
	X=\varphi(B_{t_1},B_{t_2}-B_{t_1},...,B_{t_n}-B_{t_{n-1}}), \quad \varphi \in  C_{l,lip}(\mathbb{R}^{ n}), \quad 0 \leq t_1 < ... < t_n \leq T.
\end{equation*}
Then the $G$-conditional expectation under $\mathcal{F}_{t_j}$ is defined as
\begin{equation*}
	\hat{\mathbb{E}}_{t_j}[X]=\hat{\mathbb{E}}[X \vert \mathcal{F}_{t_j}]= \psi (B_{t_1},B_{t_2}-B_{t_1},...,B_{t_j}-B_{t_j-1}), \quad j=1,\dots,n-1,
\end{equation*}
where
\begin{equation*}
	\psi(x):=\hat{\mathbb{E}}[\varphi(x,B_{t_j+1}-B_{t_j},...,B_{t_n}-B_{t_n-1})].
\end{equation*}
\end{defi}
The $G$-expectation and the $G$-conditional expectation can be extended to sublinear operators $\hat{\mathbb{E}}[\cdot]:L_G^p(\Omega) \to \mathbb{R}$ and $\hat{\mathbb{E}}_t[\cdot]: L_G^p(\Omega_T) \to L_G^p(\Omega_t)$, $t \in [0,T]$, see \violet{Section 3.2 in \cite{peng_nonlinearExpectation_book}.}\\
It is shown in \violet{Theorem 6.2.5 in \cite{peng_nonlinearExpectation_book}} that the $G$-expectation is an upper expectation, i.e., there exists a weakly compact set of probability measures $\mathcal{P}$ such that
\begin{equation} \label{eq:GExpectationUpperExpectation}
	\hat{\mathbb{E}}[X]= \sup_{P \in \mathcal{P}} {E}^P[X] \quad \text{ for each } X \in L_G^1(\Omega_T).
\end{equation}
Related to this set $\mathcal{P}$ we introduce the Choquet capacity defined by
\begin{equation}
	c(A):=\sup_{P \in \mathcal{P}} P(A), \quad A \in \mathcal{B}(\Omega_T). \label{eq:capacity}
\end{equation}
In the following definitions we introduce the most important spaces we will work with.
\begin{defi} \label{DefinitionSimpleIntegrands}
For $p \geq 1$ we denote by $M^{p,0}_G(0,T)$ the space of simple integrands. Specifically, for a given partition $\lbrace t_0,...,t_N \rbrace$ of $[0,T]$, $N\in\mathbb{N}$, we define an element $\eta \in M^{p,0}_G(0,T)$ by
\begin{equation*}
	\eta_t(\omega)=\sum_{j=0}^{N-1} \xi_j(\omega) \textbf{1}_{[t_j, t_{j+1})}(t),
\end{equation*}
where $\xi_i \in L_G^p(\Omega_{t_i}), i=0,...,N-1$. 
\end{defi}
\begin{defi}
For $p \geq 1$ we let $M_G^p(0,T)$ be the completion of $M^{p,0}_G(0,T)$ under the norm $(\hat{\mathbb{E}}[\int_0^T \vert \eta_t \vert^p dt])^{1/p}$,
and  $\overline{M}^p_G(0,T)$ be the completion of $M^{p,0}_G(0,T)$ under the norm $(\int_0^T \hat{\mathbb{E}}[\vert \eta_t\vert^p]dt)^{1/p}$. 
\end{defi}
Note that $\overline{M}_G^p(0,T) \subseteq {M}_G^p(0,T)$.
Similar as in the classical It\^{o} case the integral with respect to $B$ is first defined for the simple integrands in $M_G^{2,0}$ by the mapping $I(\eta)=\int_0^T \eta_s dB_s: M_G^{2,0} \to L_G^2(\Omega_T)$. Then this mapping can be continuously extended to $I:  M_G^{2} \to L_G^2(\Omega_T)$, see \violet{Lemma 3.3.4 in \cite{peng_nonlinearExpectation_book}}. As in the classical case the quadratic variation of the $G$-Brownian motion is defined as the process $\langle B \rangle=(\langle B \rangle_t)_{t \in [0,T]}$ with $\langle B \rangle_t=B_t^2- 2 \int_0^t B_s dB_s$ for $t \in [0,T]$. The integral with respect to the quadratic variation $\int_0^{\cdot} \eta_s d \langle B \rangle_s$ is introduced first for $\eta \in M_G^{1,0}(0,T)$ and then extended to $M_G^1(0,T)$, see \violet{Section 3.4 in \cite{peng_nonlinearExpectation_book}}.
\begin{defi}
A process $(M_t)_{t \in [0,T]}$ is called a \emph{$G$-martingale} (respectively, $G$-supermartingale, $G$-submartingale) if for each $t \in [0,T]$, $M_t \in L_G^1(\Omega_t)$ and for each $0 \leq s \leq t \leq T$, we have
\begin{equation*}
	\hat{\mathbb{E}}_s[M_t]=M_s, \quad \text{(respectively, } \leq M_s, \quad \geq M_s \text{).}
\end{equation*}
If in addition it holds also 
\begin{equation*}
	\hat{\mathbb{E}}_s[-M_t]=-M_s \quad\text{ for }0\leq s \leq t \leq T,
	\end{equation*}
then $(M_t)_{t \in [0,T]}$ is called a \emph{symmetric} $G$-martingale.
\end{defi}

\subsection{Exit times in the $G$-setting}
In the following we summarize some results given in \cite{liu} about exit times for semimartingales in the $G$-setting, which we use in Section \ref{sec:results}. In \cite{liu}, a general nonlinear expectation is considered, however for simplicity we now state the results for the $1$-dimensional $G$-expectation. Moreover, instead of dealing with open sets in $\mathbb{R}^d$ we only work with open intervals in $\mathbb{R}$.

Let $\Omega$ be introduced as in Definition \ref{def:initial} and $\mathcal{P}$ \blue{be} the weakly compact set of probability measures generating the $G$-expectation as in $(\ref{eq:GExpectationUpperExpectation})$. \blue{Consider} a one-dimensional $\mathcal{P}$-semimartingale $Y=(Y_t)_{t \ge 0}$, i.e., suppose that under each $P \in \mathcal{P}$, $Y$ has the decomposition 
\begin{equation} \label{eq:Semimartingale}
 	Y_t=M_t^P+A_t^P, \quad t \geq 0,
 \end{equation}
where $M^P$ is a continuous local martingale and $A^P$ a finite variation process. \blue{Given an open interval $D$ in $\mathbb{R}$,} define the following exit times
\begin{align*}
	\tau_{D}&:=\inf \lbrace t \geq 0: Y_t \in D^c \rbrace \quad \text{ and } \quad \tau_{\overline{D}}:=\inf \lbrace t \geq 0: Y_t \in \overline{D}^c \rbrace,
\end{align*}
where \blue{$D^c$ and }$\overline{D}^c$ denote the complement of \blue{$D$} and of the closure $\overline{D}$ of $D$ in $\mathbb{R}$\blue{, respectively.}
\blue{I}ntroduce the space
\begin{equation*}
	\Omega^{\omega}_{\blue{D}}=\lbrace \omega' \in \Omega: \omega_t'=\omega_t \text{ \blue{for} } \blue{t \in }[0,\tau_D(\omega)] \rbrace, \quad \text{ for each } \omega \in \Omega. 
\end{equation*}
\begin{asum}{(Local growth condition (H))}\label{LocalGrowthCondition} \blue{Let $D$ be an open interval in $\mathbb{R}$. We say that the stochastic process $Y$ in \eqref{eq:Semimartingale} satisfies the local growth condition $(H)$ if the following holds. }
For each $P \in \mathcal{P}$ there exists a $P$-null set $N$ such that, if $\omega \in N^c$ satisfies $\tau_D(\omega) < \infty$, then there exist some stopping time $\sigma^{\omega}$ and constants $\lambda^{\omega},\epsilon^{\omega}>0$ so that
\begin{enumerate}
	\item $\sigma^{\omega}(\omega')>0$ for $\omega' \in N^c \cap \Omega^{\omega}_{\blue{D}}$;
	\item For $\omega' \in N^c \cap \Omega^{\omega}_{\blue{D}}$, on the interval $[0,\sigma^{\omega}(\omega') \wedge (\tau_{\overline{D}}(\omega') - \tau_D(\omega'))]$, it holds that
	\begin{align*}
				d\langle M^P \rangle_{\tau_D(\omega)+t}(\omega') \geq \epsilon^{\omega}\vert dA^P_{\tau_D(\omega)+t}(\omega') \vert \quad \text{ and } \quad d\langle M^P \rangle_{\tau_D(\omega)+t}(\omega')>0.
	\end{align*}
	Moreover, these three quantities $\sigma^{\omega}, \lambda^{\omega}$ and $\epsilon^{\omega}$ can depend on $P,\omega$ and are supposed to be uniform for all $\omega' \in N^c \cap \Omega^{\omega}_{\blue{D}}$. 
\end{enumerate}
\end{asum}
We now recall the definition of a quasi-continuous function on a \violet{measurable} space $(\Omega,\mathcal{B}(\Omega))$, see Definition \violet{6.1.26 in \cite{peng_nonlinearExpectation_book}}, and \blue{of} a quasi-continuous process, see p. 13 in \cite{liu}. The capacity $c$ is introduced in $(\ref{eq:capacity})$. 
\begin{defi}
\begin{enumerate}
	\item A mapping $\Psi$ on $\Omega$ with values in a topological space is said to be quasi-continuous (q.c.) if for all $\epsilon>0$ there exists an open set $O \subseteq \Omega$ with $c(O)<\epsilon$ such that $\Psi \vert_{O^c}$ is continuous. 	
	\item A process $Y=(Y_t)_{t \in [0,T]}$ is quasi-continuous on $\Omega \times [0,T]$ if for each $\epsilon>0$, there exists an open set $O \subseteq \Omega$ with $c(O)<\epsilon$ such that $Y_{\cdot}(\cdot)$ is continuous on $O^c \times [0,T]$.
\end{enumerate}
\end{defi}
\begin{prop}{(Corollary 3.7, 1. in \cite{liu})} \label{cor:StoppingTimeMaxQuasicontitnuous}
	\blue{Let $D$ be an open interval in $\mathbb{R}$.} Assume that $\blue{Z}$ is quasi-continuous and Assumption \ref{LocalGrowthCondition} \blue{holds}. Then $\tau_{D} \wedge {Z}$ and $\tau_{\overline{D}} \wedge {Z}$ are both quasi-continuous. 
\end{prop}

\begin{cor}{(Corollary 4.8 in \cite{liu})} \label{StoppingLiu}
	Let $\tau$ be a quasi-continuous stopping time. If $(M_t)_{t \geq 0}$ is a $G$-martingale, then $(M_{t \wedge \tau})_{t \geq 0}$ is still a $G$-martingale.
\end{cor}
Proposition 4.10 in \cite{liu} and Remark 4.12 in \cite{liu} give the following result.
\begin{prop}\label{PropositionLiu}
	Let $\tau \leq T$ be a quasi-continuous stopping time. Then for each $p \geq 1$, we have
	\begin{equation*}
		\textbf{1}_{[0,\tau]} \in M_G^p(0,T).
	\end{equation*}
	Moreover, for any $\eta \in M_G^p(0,T)$ it holds $\eta\textbf{1}_{[0,\tau]} \in  M_G^p(0,T)$ and
	\begin{equation*}
	\int_0^{\tau} \eta_s dB_s = \int_0^T \eta_s \textbf{1}_{[0,\tau]}(s) dB_s.
\end{equation*}
\end{prop}
By combining Example 5.1(i) and Remark 5.2 from \cite{liu}, we can characterize under which conditions a solution of a $G$-SDE satisfies Assumption \ref{LocalGrowthCondition}.
\begin{lemma} \label{ExampleLiu}
	Let $Y=(Y_t)_{t \ge 0}$ be the unique solution to the $1$-dimensional $G$-SDE
\begin{equation*}
	dY_t=b(t,Y_t)dt + h(t,Y_t)d \langle B \rangle_t +\sigma(t,Y_t)dB_{t} \quad Y_0=y, \quad t\ge 0,
\end{equation*}
where $y \in \mathbb{R}$ and $b,h, \sigma: [0,T]  \times \mathbb{R} \to \mathbb{R}$ satisfy the following conditions:
\begin{enumerate}
	\item the functions $b(t,\cdot), h(t, \cdot), \sigma(t,\cdot)$ are Lipschitz continuous for all $t \in [0,T]$ and the functions $b(\cdot, x), h(\cdot, x), \sigma (\cdot, x)$ are continuous for all $x \in \mathbb{R}$. 
	\item $\sigma$ is non-degenerate, i.e., there exists $\lambda >0$ such that
	\begin{equation*}
		\lambda \leq \sigma(t,y)^2 \quad \text{ for all } y \in \blue{\overline{D}, t \in [0,T]}.
	\end{equation*}
	Then the process $Y$ is quasi-continuous and Assumption \ref{LocalGrowthCondition} is satisfied.
\end{enumerate}
\end{lemma}
\begin{proof}
	It follows by Example 5.1 (i) and Remark 5.2 of \cite{liu}.
\end{proof}
\blue{
\begin{remark} \label{remark:StoppingTimeGreater}
	\blue{Given an open interval $D$ in $\mathbb{R}$,} note that all results stated in this subsection are also valid for the exit times 
\begin{align*}
	\tau^{D,r}&:=\inf \lbrace t \geq r: Y^r_t \in D^c \rbrace \quad \text{ and } \quad \tau^{\overline{D},r}:=\inf \lbrace t \geq r: Y^r_t \in \overline{D}^c \rbrace,
\end{align*}
where $r \in [0,T]$ is fixed \blue{and $Y^r=(Y_t^r)_{t \in [r,T]}$ is a process starting at time $r$, with $Y^r_r \in D$.}
\end{remark}}

\section{Nonlinear PDEs}\label{sec:krylov}
We here recall some results of \cite{krylov} which are fundamental for our analysis in Section \ref{sec:results} and add some new properties.
Introduce an $\mathbb{R}$-valued Borel function $F: (\mathbb{R}^5, \mathcal{B}(\mathbb{R}^5)) \to (\mathbb{R},\mathcal{B}(\mathbb{R}))$. Moreover, let $D$ be a (possibly unbounded) open interval of $\mathbb{R}$ and $Q = (0,T) \times D$ be a \blue{cylindrical} domain in $\mathbb{R}^{2}$, with $T>0$.
\blue{We introduce the following notation
\begin{align*}
	\partial_t Q:=\lbrace (t,x): t=T, \  x \in \overline{D} \rbrace, \quad \partial_x Q:= (0,T) \times \partial D.
\end{align*}
Then the parabolic boundary of $Q$ {is defined as} $\partial'Q:=\partial_t Q \cup \partial_x Q$. For $D=(0,\infty)$, the parabolic boundary of $Q=(0,T) \times (0,\infty)$ is given by
\begin{align*}
	\partial'Q = \lbrace T \rbrace \times \mathbb{R}^+ \cup (0,T) \times \lbrace 0 \rbrace.
\end{align*}
Note that in this paper $\mathbb{R}^+:=[0,\infty)$. Moreover, for all $\mathbb{R}$-valued functions $f: B \to \mathbb{R}$ with $B \subseteq \mathbb{R}^d, d \geq 1$, we use the notation $\| f \|_{\infty} :=\sup_{y \in B} \vert f(y) \vert$, where $\vert \cdot \vert$ denotes the absolute value.}
This section is devoted to the analysis of conditions that provide the existence of a classic solution $u:\overline Q \to \mathbb{R} $ to the PDE 
\begin{align}
	u_t+F(\black{t,x},u,\black{u_{x},u_{xx}})=0  \quad &\black{\text{ for } (t,x) \in Q} \label{ProblemI}\\
	u=\varphi \quad& \black{\text{ for } (t,x) \in \blue{\partial'Q},}  \label{ProblemII}
\end{align}
where we \blue{set}
\begin{equation*}
	u:=u(t,x), \quad u_x:=u_x(t,x), \quad u_t:=u_t(t,x), \quad u_{xx}:=u_{xx}(t,x),
\end{equation*}
for the sake of simplicity. In the following we use the notation of \cite{krylov}, where $u$, $u_x$, $u_{xx}$ have not to be understood as dependent on $(t,x)$ but only as arguments of $F$.\\
First, we introduce the sets of functions $F(\eta,K,Q)$ and $F_1(\eta,K,Q)$ for constants $K \geq  \eta >0$.
\begin{defi}{(Definition 1, Section 5.5 in \cite{krylov})}\label{DefiFunctions}
Introduce the constants $K \geq  \eta >0$.
\begin{enumerate}
	
\item We say that $ F \in F_1(\eta,K,Q)$ if the following properties are satisfied:
\begin{itemize}
\item for every $t \in [0,T]$ the function $F$ is twice continuously differentiable with respect to $(x,u, \black{u_{x},u_{xx}}) \in D \times \mathbb{R}^3$;
\item for all $(t,x) \in Q$ and $u_{xx}, \tilde{u}_{xx}, u_x,\tilde{u}_x, u,\tilde{u},  \lambda, \tilde{x} \in \mathbb{R}$ the following inequalities hold:
\begin{align*}
	\eta \vert \lambda \vert^2 \leq F_{u_{\black{xx}}} \black{\lambda^2 } \leq K \vert \lambda \vert^2 \\
	\vert F-F_{u_{\black{xx}}}u_{\black{xx}}\vert \leq M_1^F(u) \big( 1+  \vert u_{\black{x}} \vert^2\big) \\
	\vert F_{u_{\black{x}}}\vert \big( 1+  \vert u_{\black{x}} \vert ) + \vert F_u\vert +\vert F_{\black{x}}\vert \big( 1+ \vert u_{\black{x}} \vert \big)^{-1}  \leq M_1^F(u) \big (1+\vert u_{\black{x}} \vert^2+ \vert u_{\black{xx}} \vert \big)  \\
	[M_2^F(u,u_{\black{x}})]^{-1} F_{(\eta), (\eta)} \leq \vert \tilde{u}_{\black{xx}} \vert \bigg [\vert \tilde{u}_{\black{x}} \vert + \big ( 1+  \vert u_{\black{xx}} \vert \big) (\vert \tilde{u} \vert + \vert \tilde{x} \vert )\bigg]+ \\ 
	 \vert \tilde{u}_{\black{x}} \vert^2 \big( 1+  \vert u_{\black{xx}} \vert  \big) + \big (1+ \vert u_{\black{xx}} \vert ^3 \big) (\vert \tilde{u} \vert^2 + \vert \tilde{x}\vert^2),
\end{align*}
where \blue{we omit} the arguments $(t,x,u,\black{u_x,u_{xx}})$ of $F$ and its derivatives, $\eta:=(\black{\tilde{x},\tilde{u},\tilde{u}_x,\tilde{u}_{xx}})$ and
\begin{align*}
	F_{(\eta), (\eta)}&:=F_{u_{\black{xx}}u_{\black{xx}}} \tilde{u}_{\black{xx}} \tilde{u}_{\black{xx}}+ 2 F_{u_{\black{xx}}u_r}\tilde{u}_{\black{xx}}\tilde{u}_r + 2 F_{u_{\black{xx}}u}\tilde{u}_{\black{xx}}\tilde{u} + 2 F_{u_{ij}x}\tilde{u}_{\black{xx}}\tilde{x} \\
	&\quad + F_{u_{\black{x}}u_{\black{x}}}\tilde{u}_{\black{x}} \tilde{u}_{\black{x}} + 2 F_{u_{\black{x}} u} \tilde{u}_{\black{x}} \tilde{u}+ 2 F_{u_{\black{x}}x}\tilde{u}_{\black{x}}\tilde{x}+ F_{uu}(\tilde{u})^2+ 2F_{ux}\tilde{u} \tilde{x}+ F_{x x} \tilde{x} \tilde{x}.
\end{align*}
Moreover, $M_1^F(u)$ and $M_2^F(u,u_x)$ are here some continuous functions which grow with $\vert u\vert$ and $\blue{(u_x)^2}$, and such that $M_2^F \geq 1$. 
\end{itemize}
\item We say that $F \in F(\eta, K, Q)$ if the following properties are satisfied:
\begin{itemize}
\item $F$ is continuously differentiable with respect to all arguments;
\item $F \in F_1(\eta, K, Q)$;
\item for each $(t,x) \in Q$ and $u_{\black{x}}, u \in \mathbb{R}$ we have
\begin{equation*}
	\vert F_t \vert \leq M_3^F(u,u_{\black{x}}) \big (1+ \vert u_{\black{u_{xx}}} \vert^2 \big),
\end{equation*}
where $M^F_3$ is a continuous function growing with $\vert u \vert $ and $\blue{(u_x)^2}$. 
\end{itemize}
\end{enumerate}
\end{defi}

\begin{defi}{(Definition 1, Section 6.1 in \cite{krylov})} \label{DefiFunctions2}
Introduce the constants $K \geq  \eta >0$ and define the set $P(M,Q)$ as 
	\begin{equation*}
		P(M,Q):=\lbrace (\black{t,x},u,u_{\black{x}},u_{\black{xx}}):  \vert u_{\black{xx}} \vert + \vert u_{\black{x}} \vert + \vert u \vert \leq M, \black{(t,x)} \in Q \rbrace.
	\end{equation*} 
We say that $F \in \overline{F}(\eta,K,Q)$ if there exists a sequence of functions $F_n \in F(\eta, K,Q)$ converging to $F$ as $n \to \infty$ at every point of the set $P(\infty, Q)$ such that
\begin{enumerate}
	\item $ M_i^{F_1}=M_i^{F_2}=....:=\tilde{M}_i^{F} \  i=1,2,3;$
	\item for any $n=1,2,...,$ the function $F_n$ is infinitely differentiable with respect to $(\black{t,x},u,u_{\black{x}},u_{\black{xx}})$ and the derivative of any order of the function $F_n$ with respect to $(\black{t,x},u,u_{\black{x}},u_{\black{xx}})$ is bounded on $P(M,Q)$ for any $M < \infty$;
	\item there exist constants $\delta_0=:\delta_0^F >0,  M_0=:M_0^F>0$ such that 
	\begin{equation*}
		F_n(\black{t,x},-M_0,0,u_{\black{xx}}) \geq \delta_0, \quad F_n(\black{t,x},M_0,0,-u_{\black{xx}}) \leq -\delta_0
	\end{equation*}
	for all $n \geq 1, \black{(t,x)} \in Q$, and $\black{u_{xx}} \in \mathbb{R}$. 
	\end{enumerate}

\end{defi}

\blue{Note that the following results also hold for $D=\mathbb{R}$. }
\begin{theorem}{(Theorem 3 \blue{and 4} in Section 6.4 in \cite{krylov})} \label{Theorem3}
Fix $K \geq \eta > 0$, $T>0$, $\alpha \in (0,1)$, and introduce an interval $D \subset \black{\mathbb{R}}$ and a domain $Q \subseteq (0,T) \times D$. Moreover, let $F \in \overline{F}(\eta, K,Q)$. Suppose that $\varphi \in C (\overline{Q})$ and $\blue{\| \varphi \|_{\infty}} \leq M_0^F$ on $Q$. Then Problem \eqref{ProblemI}-\eqref{ProblemII} has a solution $u$ with the following properties:
\begin{enumerate}
	\item $u \in C(\overline{Q})$, $\blue{\| u \|_{\infty}}\leq M_0^F$ on $Q$;
	\item \blue{for every $\kappa \in (0,1)$, the function $u$ belongs to the space $C^{1,2+\alpha_0} ([0,T-\kappa^2] \times \mathbb{R})$ where $\alpha_0=\alpha_0(d,K,\eta) \in (0,1)$, and the norm of $u$ in this space is bounded by a constant depending only on $d,K,\eta,M_0^F,\kappa,$ and the function $\tilde{M}_i^F, i=1,2,3$. }
	\item Let $D^0,D^1$, and $D$ be domains such that $D^0 \subset D^1 \subset D$ and let $\rho:=\text{dist}(\partial D^0, \partial D^1)>0$. If $\varphi(T, \cdot) \in C^{2+\alpha}(D^1)$, then there exists a solution of Problem \eqref{ProblemI}-\eqref{ProblemII}  possessing the properties stated in 1. and such that $u \in C^{\black{1},2+\beta}([0,T] \times \overline{D}^0)$ where $\beta=\alpha_0 \wedge \alpha, \alpha_0 = \alpha_0(d,K, \eta) \in (0,1)$. The norm of $u$ in this space is bounded by a constant depending only on $d,K,\eta, \alpha, M_0^F, \rho$, in the functions $\tilde{M}_i^F, i=1,2,3,$ and the norm of $\varphi(T, \cdot) \in C^{2+\alpha}(D^1)$.

\end{enumerate}
\end{theorem}

We now state a theorem given in \cite{krylov} which we will use \black{in Example \ref{ExampleKrylov}} to give an example for $F \in \overline{F}(\eta,K,Q)$. 

\begin{theorem}{(Theorem 5 (d), Section 6.1 in \cite{krylov})} \label{theoremInf}
Fix $K \geq \eta > 0$, $T>0$, $\alpha \in (0,1)$, and introduce an interval $D \subset \black{\mathbb{R}}$ and a domain $Q \subseteq (0,T) \times D$. Fix a set of indices $\black{N}$ and introduce a function $F_{{n}} \in \overline{F}(\eta,K,Q)$ for every $\black{n \in N}$, such that $\blue{\tilde{M}_1^{F_{n}},\tilde{M}_2^{F_{n}}, \tilde{M}_3^{F_{n}}}, \delta_0^{F_{\black{n}}}, M_0^{F_{\black{n}}}$ are independent of $\black{n}$. Then $F:=\inf \lbrace F_{\black{n}}:\black{ n \in N} \rbrace \in \overline{F}(\eta,K,Q)$. Moreover, $\tilde{M}_i^{F}, \blue{i=1,2,3,} \  \delta_0^F, M_0^F$ can be replaced by $1 + \tilde{M}_i^{F_{\black{n}}},\blue{i=1,2,3}, \ \frac{1}{2} \delta_0^{F_{n}}, M_0^{F_{\black{n}}}$ (for any $\black{n \in N }$).
\end{theorem}

Using Theorem \ref{theoremInf} several examples of functions in $\overline{F}(\eta, K, Q)$ can be constructed.
\begin{example}{(Example 8, Section 6.1 in \cite{krylov}) } \label{ExampleKrylov}
Let $N$ be a set of indices. \blue{Consider two functions $a, b : N \times (0,T) \times D \times \mathbb{R} \times \mathbb{R} \to \mathbb{R}$.} Suppose that, for every fixed $\black{n \in N}$, the functions $a$ and $b$ are continuously differentiable with respect to \blue{$(t,x,u,u_x) \in (0,T) \times D \times \mathbb{R} \times \mathbb{R}$}, and twice continuously differentiable with respect to \blue{$(x,u,u_x) \in D \times \mathbb{R} \times \mathbb{R}$}, for every fixed $t \in \blue{(0,T)}$.

In addition, let the second derivatives of $a$ and $b$ with respect to $(\black{x,u,u_x})$ and the first derivatives with respect to $t$ be bounded on every set $\tilde{P}(M):=\lbrace (\black{n,t,x,u,u_x):n \in N}, \vert u_{x} \vert + \vert u \vert \leq M, \black{(t,x)} \in Q \rbrace$. Moreover, suppose that for all $\black{n \in N}, \black{(t,x)} \in Q, \lambda \in \black{\mathbb{R}}, u, u_{\black{x}}$ the following inequalities hold:
\begin{align}
	\eta \vert \lambda \vert^2 \leq a \lambda^{\black{2}}\leq K \vert \lambda \vert^2, \quad \vert b \vert \leq M_1(u) \big ( 1+ \vert u_{\black{x}} \vert^2 \big), \label{InequalityLinear1} \\
	\vert a_{u_{\black{x}}} \vert \big ( 1+ \vert u_{\black{x}} \vert\big)+ \vert a_u\vert + \vert a_{x} \vert \big(1+ \vert u_{\black{x}} \vert \big)^{-1} \leq M_1(u), \label{InequalityLinear2}\\
	\vert b_{u} \vert \big ( 1+ \vert u_{\black{x}} \vert\big) + \vert b_u \vert + \vert b_{x}\vert  \big(1+ \vert u_{\black{x}} \vert \big)^{-1} \leq M_1(u) \big ( 1+  \vert u_{\black{x}} \vert^2 \big), \label{InequalityLinear3}
\end{align}
 where $M_1(u)$ is a continuous function. Finally assume that the inequalities
\begin{equation}
	b(n,\black{t,x},-M_0, 0) \geq \delta_0,\quad  b(n,\black{t,x},M_0,0) \leq - \delta_0 \label{InequalityLinear4}
\end{equation}
hold for some constants $\delta_0 >0, M_0 >0$ with $\black{n \in N}, \black{(t,x)} \in Q$.
Then it can easily be seen that by Theorem \ref{theoremInf}, we have
\begin{equation*}
	F:=\inf_{n \in N} [a(n,\black{t,x},u,u_{\black{x}})u_{\black{xx}} + b(n,\black{t,x},u,u_{\black{x}})] \in \overline{F}(\eta,K,Q).
\end{equation*}
\blue{The constants $M_0^F$ and $\delta_0^F$ introduced in Theorem \ref{theoremInf} are equal to $M_0$  and $\frac{1}{2} \delta_0$, respectively}. Moreover, $\tilde{M}_i^F$ can be chosen \blue{to be dependent} on $d,K,M_1$, \blue{on} the functions of $M$ which dominate the second derivatives of $a,b$ with respect to $(x,u,u_{\black{x}})\blue{,}$ and on their first derivatives with respect to $t$ on $\tilde{P}(M)$. 
\end{example}

We now prove a new result that we use in Section \ref{sec:results}.

\begin{lemma}\label{lem:supinf}
\black{Let $(F_i)_{i \in I}$ be a family of Borel measurable functions $F_i: (\mathbb{R}^5,\mathcal{B}(\mathbb{R}^5)) \to (\mathbb{R},\mathcal{B}(\mathbb{R}))$ such that}
	\begin{equation}
		-F_{i}(t,x,u,u_{x},u_{xx})=F_{i}(t,x,-u,-u_{x},-u_{xx}),  \label{symmetryProperty}
	\end{equation}
	for $(t,x) \in Q$, $i \in I$, where $I$ is an index set. Then the PDE
	\begin{equation} \label{PDEsup}
	u_t + \sup_{i \in I} F_{i}(t,x,u,u_{x},u_{xx}) = 0
	\end{equation}
	has a solution $u \in C^{1,2}(Q)$ if and only if the PDE 
		\begin{equation}
		u_t + \inf_{i \in I}  F_{i}(t,x,u,u_{x},u_{xx})=0. \label{PDEinf}
	\end{equation}
	has a solution $u \in C^{1,2}(Q)$.
\end{lemma}
\begin{proof}	Suppose that $u \in C^{1,2}(Q)$ is a solution to the PDE   \eqref{PDEsup}. Then it holds
	\begin{align}
		 0&=u_t + \sup_{i \in I} F_{i}(t,x,u,u_{x},u_{xx}) \notag \\
		 &=u_t - \inf_{i \in I} \blue{\lbrace -F_{i}(t,x,u,u_{x},u_{xx}) \rbrace }\label{relationInfSup} \\
		 &= u_t - \inf_{i \in I}  F_{i}(t,x,-u,-u_{x},-u_{xx})=0 \label{ApplySymmetryProperty},
	\end{align}
	where we used the property $(\ref{symmetryProperty})$ of $\black{F_i}$ \blue{in $(\ref{relationInfSup})$}. 
	We now define $\overline{u}(t,x):=-u(t,x)$ for $(t,x) \in Q$. Then $(\ref{ApplySymmetryProperty})$ implies
	\begin{equation*}
		-\overline{u}_t - \inf_{i \in I}  F_{i}(t,x,\overline{u},\overline{u}_{x},\overline{u}_{xx})=0, 
	\end{equation*}
	that is,
	\begin{equation}
		\overline{u}_t + \inf_{i \in I}  F_{i}(t,x,\overline{u},\overline{u}_{x},\overline{u}_{xx})=0. \label{PDEinf}
	\end{equation}
	Thus $\overline{u}$ is a $C^{1,2}(Q)$-solution to the PDE \eqref{PDEinf}. The opposite direction is analogous. 
\end{proof}

\section{Feynman-Kac formula in the $G$-setting}\label{sec:results}
\blue{In this section we prove a generalized Feynman-Kac formula in the $G$-setting, in presence of a linear term in the associated PDE. In this way we complement the results of \cite{hu_ji_peng_song_2014} which hold without the linear term in the PDE and under different conditions. More precisely, we do not consider a $G$-BSDE as in \cite{hu_ji_peng_song_2014}, but solely a $G$-SDE. Moreover, we do not assume that the payoff function satisfies a Lipschitz condition, and the presence of a discounting term in the $G$-expectation prevents a further Lipschitz assumption involving the coefficients of an associated $G$-BSDE to hold, which is crucial in \cite{hu_ji_peng_song_2014}. On the other hand, we consider some other assumptions on the coefficients of the PDE which are stated in Assumption \ref{asumC1,2General}, and which are not necessary in \cite{hu_ji_peng_song_2014}.\\}
\blue{For fixed $r \in [0,T]$ and $x>0$} consider the $G$-It\^o process $X^{\blue{r,x}}=(X_t^{\blue{r,x}})_{t \in [r,T]}$ given by
	\begin{equation}
		X_t^{\blue{r,x}}=\blue{x}+\int_{\blue{r}}^t f(s,X_s^{\blue{r,x}})d\blue{s} + \int_{\blue{r}}^t g(s,X_s^{\blue{r,x}}) d\langle B \rangle_s + \int_{\blue{r}}^t h(s,X_s^{\blue{r,x}})dB_s, \quad \blue{r} \leq t \leq T, \label{G-SDE_New}
	\end{equation}
where \blue{$f,g,h:[0,T] \times \mathbb{R} \to \mathbb{R}$ \blue{are} deterministic functions such that} $f(t,\cdot), g(t, \cdot), h(t,\cdot)$ are Lipschitz-continuous functions for every $t \in [0,T]$ \blue{and $f\blue{(\cdot, x)},g\blue{(\cdot, x)},h\blue{(\cdot, x)}$ are continuous in $t$ \blue{for every $x \in \mathbb{R}$}.} \blue{Our aim is to show that 
\begin{equation}
u(r,x):= \hat{\mathbb{E}}_r \left[ \varphi(T,X_T^{r,x})e^{-\int_r^T X_s^{r,s}ds} \right], \quad r \in [0,T], x>0
\end{equation}
satisfies the PDE
\begin{align}
	u_t + 2G \left(u_x g(t,x) + \frac{1}{2} u_{xx} \left(h(t,x)\right)^2\right)	 + f(t,x)u_x - xu =0, \quad &\text{ for } (t,x) \in Q  \label{eq:PDENoTimeDependenceCovergenceIntro1}\\
		u=\varphi \quad &\text{ for } (t,x) \in \partial'Q,\label{eq:PDENoTimeDependenceConvergenceIntro2}
\end{align}
for $Q=(0,T) \times (0,\infty)$ and for suitable terminal payoffs $\varphi$ satisfying suitable conditions which we discuss in the sequel.}\\%\blue{CHANGE THIS: Then there exists a unique solution of the SDE in $\overline{M}_G^2(0,T)$ by Theorem 1.2, Chapter V in \cite{shige_script}.} 
From now on, we work under the following assumption.

\begin{asum} \label{asumC1,2General}
	\blue{The functions $f,g,h$ belong to the space $C^{1,2}([0,T] \times \black{(0,\infty)})$. Moreover, $h$ is bounded away from zero on every subset $\{(t,y) \in [0,T] \times \mathbb R: y \ge a\}$, $a>0$, and $h(t,x)>0$ for every $t \in [0,T]$, $x>0$. Finally,  
	$h_{xx}(t,x) \geq 0$ for all $(t,x) \in [0,T] \times \mathbb{R}^+$.} 
\end{asum}

{As stated in  \cite{hu_ji_peng_song_2014} and in Chapter \violet{5 in \cite{peng_nonlinearExpectation_book}}, there exists a unique solution $X^{r,x} \in \overline{M}^p_G(r,T), p \geq 2$ of the $G$-SDE \eqref{G-SDE_New}. Moreover, the following estimates hold for any $p \geq 2$ and $(r,x) \in [0,T] \times (0,\infty)$:
	\begin{align}
		\hat{\mathbb{E}}_r\left[\vert X_t^{r,x}\vert^p\right] &\leq C(1+\vert x \vert^p), \quad t \in [r,T] \label{eq:Estimate1}\\
		\hat{\mathbb{E}}_r \left[ \sup_{s \in [r,r+\delta]}\vert X_s^{r,x}-x \vert^p \right] &\leq C(1+\vert x \vert^p) \delta^{p/2}, \quad \delta \in [0,T-r], \label{eq:Estimate2}
	\end{align}
	where the constant $C$ depends on $G(\cdot), p, T$ and the Lipschitz constant of $f,g,h$. 
}
\blue{Note that the conditions in Assumption \ref{asumC1,2General} on the function $h$ are necessary to apply the PDE theory introduced in Section \ref{sec:krylov}. Moreover, the assumption on $h_{xx}$ is necessary for technical reasons in Definition \ref{Defi:SmoothFunctions} and Lemma \ref{lemma:RegularityCutoffFunctions}.}

\begin{defi}{\label{Defi:SmoothFunctions}}
For fixed $\epsilon \in (0,1)$ define $\tilde{\phi}^{\epsilon}: [0,T] \times \mathbb{R}^+ \to \mathbb{R}$ by
	\begin{equation*}
		\tilde{\phi}^{\epsilon}(t,x)=\phi (t,\epsilon^{-1})+\phi_x(\blue{t},\epsilon^{-1})\arctan(x-\epsilon^{-1})+\phi_{xx}(\blue{t},\epsilon^{-1})(1-e^{-\frac{1}{2}(x-\epsilon^{-1})^2}),
	\end{equation*}
 for $\phi=f,g,h$ introduced in $(\ref{G-SDE_New})$. Define then $\phi^{\epsilon}: [0,T] \times \mathbb{R}^+ \to \mathbb{R}$ for $\phi=f,g$, as given by
	\begin{equation} \label{eq:DefinititonCutoff1}	\phi^{\epsilon}(t,x) = \begin{cases}
		\phi(t,x) & \text{ for }x \leq \epsilon^{-1} \\
		\tilde{\phi}^{\epsilon}(t,x) & \, \text{ for } x >\epsilon^{-1},
			\end{cases}
	\end{equation}
	for $(t,x) \in [0,T] \times \mathbb{R}^+$. Moreover, define $h^{\epsilon}: [0,T] \times \mathbb{R}^+ \to \mathbb{R}$ as 
	\begin{equation} \label{eq:DefinititonCutoff2}
	h^{\epsilon}(t,x) = \begin{cases}
		\bar h^{\epsilon}(t,x) & \text{ for } x < \epsilon \\
		h(t,x) & \text{ for } \epsilon \leq x \leq \epsilon^{-1} \\
		\tilde{h}^{\epsilon}(t,x) & \, \text{ for } x >\epsilon^{-1},
			\end{cases}
	\end{equation}
	where $\bar h^{\epsilon}: [0,T] \times \mathbb{R}^+ \to \mathbb{R}$ is given by 
	\begin{equation} \label{eq:DefinitionCutoff3}
		\bar h^{\epsilon}(t,x)=h(t,\epsilon)+h_x(t,\epsilon)a\cdot\arctan(a^{-1}(x-\epsilon)) + h_{xx}(t,\epsilon) (1-e^{-\frac{1}{2}(x-\blue{\epsilon})^2}),
	\end{equation}
	setting $a=1$ if $h_x(\blue{t},\epsilon) \leq 0$ and taking $a$ such that $0 < a < \frac{2h(t,\epsilon)}{\pi h_x(\blue{t},\epsilon)}$ if $h_x(\blue{t},\epsilon)>0$. 	Note that for $h_x(\blue{t},\epsilon)>0$ such a constant $a$ exists, as for fixed $\epsilon \in (0,1)$ the function $h(\cdot,\epsilon)$ is \blue{$C^{1}([0,T])$}. Finally, \blue{we introduce} the function $\vartheta^{\epsilon}: \mathbb{R}^+ \to \mathbb{R}$ \blue{given} by	
\begin{equation} \label{eq:DefinititonCutoff3}
	\vartheta^{\epsilon}(x) = \begin{cases}
		x & x \leq \epsilon^{-1} \\
		\epsilon^{-1}+\arctan(x-\epsilon^{-1}) & \, x > \epsilon^{-1}.
			\end{cases}
	\end{equation}
 \end{defi}
 
 \begin{lemma} \label{lemma:RegularityCutoffFunctions}
	 Under Assumption \ref{asumC1,2General} the following holds \blue{for any $\epsilon \in (0,1)$}:
	\begin{enumerate}
		\item $\phi^{\epsilon}, \vartheta^{\epsilon}$ and their first and second derivatives are bounded, for every $\phi=f,g,h$.
		\item $\phi^{\epsilon} \in C^{1,2}([0,T] \times \mathbb{R}^+)$, for every $\phi=f,g,h$, and  $\vartheta^{\epsilon} \in C^2(\mathbb{R}^+)$.
		\item $h^{\epsilon}$ is bounded away from zero.
	\end{enumerate} 
\end{lemma}

\begin{proof}
	\blue{Fix $\epsilon \in (0,1)$ and let $\phi =f,g,h$.} By Assumption \ref{asumC1,2General}, $\phi$ is bounded on $[0,T] \times [0,\epsilon^{-1}]$ and $\tilde{\phi}^{\epsilon}(t,x) \leq \phi(t,\epsilon^{-1})+\blue{\frac{\pi}{2}}\phi_x(\blue{t},\epsilon^{-1})+\phi_{xx}(\blue{t},\epsilon^{-1})$ is also bounded for $(t,x) \in [0,T]\times (\epsilon^{-1},\infty)$. Therefore, $\phi^{\epsilon}$ in $(\ref{eq:DefinititonCutoff1})$ is bounded. Moreover, $\tilde{\phi}^{\epsilon}(t,\epsilon^{-1})=\phi(t,\epsilon^{-1})$ and $\tilde{\phi}_t^{\epsilon}(t,\epsilon^{-1})=\phi_t(t,\epsilon^{-1})$ for $t \in [0,T]$.
	Furthermore, for $x \in (\epsilon^{-1},\infty)$ it holds 
	\begin{align*}
		 \tilde{\phi}^{\epsilon}_x(t,x)&=\phi_x(t,\epsilon^{-1})\frac{1}{(\epsilon^{-1}-x)^2+1}+\phi_{xx}(t,\epsilon^{-1})(x-\epsilon^{-1})e^{-\frac{1}{2}(x-\epsilon^{-1})^2} \\
		\blue{\tilde{\phi}^{\epsilon}_{xt}(x,t)}&\blue{= \phi_{xt}(t,\epsilon^{-1})\frac{1}{(\epsilon^{-1}-x)^2+1} + \phi_{xxt}(t,\epsilon^{-1})(x-\epsilon^{-1})e^{-\frac{1}{2}(x-\epsilon^{-1})^2}}\\
		\tilde{\phi}^{\epsilon}_{xx}(t,x)&=\phi_x(t,\epsilon^{-1}) \frac{2(\epsilon^{-1}-x)}{((\epsilon^{-1}-x)^2+1)^2}-\phi_{xx}(t,\epsilon^{-1})(x-\epsilon^{-1})^2 e^{-\frac{1}{2}(x-\epsilon^{-1})^2}+ \phi_{xx}(t,\epsilon^{-1})e^{-\frac{1}{2}(x-\epsilon^{-1})^2}. 
	\end{align*} 
	In particular, $\tilde{\phi}^{\epsilon}_x(t,\epsilon^{-1})=\phi_x(t,\epsilon^{-1})$, \blue{$\tilde{\phi}_{xt}^{\epsilon}(t,\epsilon^{-1}) = \phi_{xt}(t,\epsilon^{-1})$}, $\tilde{\phi}_{xx}^{\epsilon}(t,\epsilon^{-1})=\phi_{xx}(t,\epsilon^{-1})$. Hence, $\phi^{\epsilon} \in C^{1,2}([0,T] \times \mathbb{R}^+)$ with bounded first and second derivatives for $\phi=f,g$. With similar calculations it can be shown that  $h^{\epsilon} \in C^{1,2}([0,T] \times \mathbb{R}^+)$, with bounded first and second derivatives. Note that by the definition of $\bar h^{\epsilon}$ in $(\ref{eq:DefinitionCutoff3})$ and as $\vert \arctan(x) \vert \blue{\leq} \pi/2$ for all $x \in \mathbb{R}$, it follows $\bar h^{\epsilon}(t,x) >\delta_{\epsilon}$ for all $(t,x) \in [0,T] \times [0,\epsilon]$, where $\delta_{\epsilon}$ is a constant which only depends on $\epsilon$. Thus $h^{\epsilon}$ is bounded away from zero.
	By the definition of $\vartheta^{\epsilon}$ in $(\ref{eq:DefinititonCutoff3})$ we have $\vartheta^{\epsilon} \in \blue{C^{2}(\mathbb{R}^+)}$ with bounded first and second derivatives. 
\end{proof}

\subsection{Feynman-Kac formula for a bounded payoff}\label{subsec:boundedpayoff}
\blue{In this section we prove a Feynman-Kac formula for a continuous and bounded payoff $\varphi$ in \eqref{eq:PDENoTimeDependenceConvergenceIntro2}. However, based on the results we find under this assumption and in particular on Theorem \ref{theorem:StabilityBounded}, we are able to extend this result to the case of an unbounded payoff with polynomial growth in Section \ref{sectionUnbounded}.\\}
\blue{
\begin{asum}\label{asum:phi}
The function $\varphi$ in \eqref{eq:PDENoTimeDependenceConvergenceIntro2} belongs to $C([0,T] \times \mathbb R^+)$ and is bounded by a constant $M_0 >0$.
\end{asum}}
\blue{
In order to apply Theorem \ref{Theorem3}, we need the terminal condition of the PDE \eqref{ProblemI}-\eqref{ProblemII} to be bounded and more than twice differentiable on a suitable interval, see the final part of the proof of Proposition \ref{prop:ExistenceExtendedSolution}. For this reason, we first approximate $\varphi$ with a {family} of functions satisfying this property.}
\blue{
\begin{defi}\label{def:ApproximationPayoff}
	Let $\varphi \in C([0,T] \times \mathbb R^+)$. By the Weierstrass approximation theorem, for any \blue{$\epsilon \in (0,1)$} \blue{there exists a} polynomial \blue{function} $\tilde{\varphi}^{\epsilon}: [0,T] \times [0,2\epsilon^{-1}] \to \mathbb{R}^+$ on $[0,T] \times [0,2 \epsilon^{-1}]$ such that 
	\begin{equation}\label{eq:approxphiepsbelow}
	\vert \varphi(t,x)-\tilde\varphi^{\epsilon} (t,x)\vert < \epsilon \text{ for every $(t,x) \in [0,T] \times [0,2\epsilon^{-1}]$}.
	\end{equation}
	 We then define the function $\varphi^{\epsilon}: [0,T] \times \mathbb{R}^+ \to \mathbb{R}$ as 
	\begin{equation} \label{eq:DefinititonCutOffPayoffFunction}
	\varphi^{\epsilon}(t,x) = \begin{cases}
		\tilde{\varphi}^{\epsilon}(t,x) & \text{for } x \leq  2\epsilon^{-1} \\
		\tilde{\varphi}^{\epsilon}(t,2\epsilon^{-1}) & \text{for } x > 2\epsilon^{-1}.
			\end{cases}
	\end{equation}
\end{defi}}

\blue{
\begin{lemma} \label{lemma:RegularityConvergencePayoff}
 Let $\varphi \in C([0,T] \times \mathbb R^+)$ satisfy Assumption \ref{asum:phi} and $(\varphi^{\epsilon})_{\blue{\epsilon \in (0,1)}}$ be the \blue{family} of functions constructed in Definition \ref{def:ApproximationPayoff}.
Then the following holds:
\begin{enumerate}
	\item For fixed $T>0$, $\varphi^{\epsilon}(T, \cdot)$ converges to $\varphi(T,\cdot)$ pointwise for $\epsilon \to 0$.
	\item For fixed $\blue{\epsilon \in (0,1)}$, for every $(t,x) \in [0,T] \times \mathbb{R}^+$ it holds 
	\begin{equation} \label{eq:ApproximationIsBounded}
		\vert \varphi^{\epsilon}(t,x) \vert \leq M_0+1,
	\end{equation}
	\item For fixed $\blue{\epsilon \in (0,1)}$, $\varphi^{\epsilon} \in C([0,T] \times \mathbb{R}^+)$.
	\item For fixed $\blue{\epsilon \in (0,1)}$, $\varphi^{\epsilon}(T, \cdot) \in C^{\infty}\left([0,2\epsilon)\right)$.
\end{enumerate} 
\end{lemma}
\begin{proof}
	Points 3. and 4. immediately follow by the definition of $\varphi^{\epsilon}$ in \eqref{eq:DefinititonCutOffPayoffFunction} and since $\tilde{\varphi}^{\epsilon}$ is a polynomial. Moreover, inequality \eqref{eq:ApproximationIsBounded} is a consequence of \eqref{eq:approxphiepsbelow} and the fact that $\varphi$ is bounded by the constant $M_0$. \\
 To prove Point 1., fix $T >0$ \blue{and $\delta > 0$}. We want to show that for all $\blue{x \ge 0}$ there exists a constant $\bar{\epsilon}(x,\delta)>0$ such that for all $\tilde{\epsilon} <\bar{\epsilon}(x, \delta)$ it holds $\vert \varphi^{\tilde{\epsilon}}(T,x)-\varphi(T,x) \vert < \delta$. \\Choose $\bar{\epsilon}(x, \delta)=\min (2x^{-1}, \delta)$ and let $\tilde{\epsilon}<\bar{\epsilon}(x, \delta)$. Then, since $x \le 2\left(\bar{\epsilon}(x, \delta)\right)^{-1} < 2\tilde{\epsilon}^{-1}$, we have 
	\begin{equation*}
		\vert \varphi^{\tilde{\epsilon}}(T,x)-\varphi(T,x) \vert = \vert \tilde{\varphi}^{\tilde{\epsilon}}(T,x)-\varphi(T,x) \vert < \tilde{\epsilon} <\blue{\delta}.
	\end{equation*}
\end{proof}}

\blue{The following lemma will be useful to derive our convergence results, see the proof of Theorem \ref{theorem:PointwiseConvergence}.
\begin{lemma}\label{lem:convergenceapproxvarphi}
Let $\varphi \in C([0,T] \times \mathbb R^+)$ satisfy Assumption \ref{asum:phi} and $(\varphi^{\epsilon})_{\blue{\epsilon \in (0,1)}}$ be the \blue{family} of functions constructed in Definition \ref{def:ApproximationPayoff}. \blue{Given $(r,x) \in [0,T] \times (0,\infty)$,} let $X^{\blue{r,x}}$ be the solution to the $G$-SDE \eqref{G-SDE_New} and assume that $X^{\blue{r,x}}$ is quasi-surely strictly positive. Then it holds
$$
\lim_{\epsilon \to \blue{0}} \blue{\mathbb{\hat{E}}}_{\blue{r}}\left[|\varphi(T,X_T^{\blue{r,x}})-\varphi^{\epsilon}(T,X_T^{\blue{r,x}})|\right] = 0.
$$
\end{lemma}
\begin{proof}
The result easily follows by the dominated convergence theorem stated in Theorem 3.2 in \cite{hu_zhou_2018}, which can be applied by Points 1. and 2. of Lemma \ref{lemma:RegularityConvergencePayoff}.
\end{proof}
}

\blue{Set $Q:=(0,T) \times (0,\infty)$ from now on}. \blue{For $\epsilon \in (0,1)$ we} then consider the following PDE 
\begin{align}
		u_t + 2G \left(u_x g^{\epsilon}(t,x) + \frac{1}{2} u_{xx} \left(h^{\epsilon}(t,x)\right)^2\right) + f^{\epsilon}(t,x)u_x - (\vartheta^{\epsilon}(x)+\blue{\epsilon})u=0, \quad &\text{ for } (t,x) \in Q  \label{PDEExtendedDomain1}\\
		u=\blue{\varphi^{\epsilon}} \quad &\text{ for } (t,x) \in \blue{\partial' Q}, \label{PDEExtendedDomain2}
	\end{align}
where $G(\cdot)$ is defined as in \eqref{G-equationOneDimension}.\\
\black{C}ondition \black{$(\ref{uniformlyElliptic})$} guarantees that the PDE in $(\ref{PDEExtendedDomain1})$-$(\ref{PDEExtendedDomain2})$ admits a unique viscosity solution \blue{$u^{\epsilon}$}, see also Appendix C in \violet{\cite{peng_nonlinearExpectation_book}}. Together with Assumption \ref{asumC1,2General}, it \black{also} guarantees that the PDE in $(\ref{PDEExtendedDomain1})$ is non-degenerate. Define the function $F^{\blue{\epsilon}}:[0,T]\times \mathbb{R}^4 \to \mathbb{R}$ by $F^{\blue{\epsilon}}(t,x,r,p,y):=2G\left(pg^{\blue{\epsilon}}(t,x)+ \frac{1}{2} y \left(h^{\blue{\epsilon}}(t,x)\right)^2 \right) + f^{\blue{\epsilon}}(t,x)p-\blue{(\vartheta^{\epsilon}(x)+\epsilon)}r$. 
	Indeed for ${y} > \overline{y}$ it holds
		\begin{align}
			&F^{\blue{\epsilon}}(t,x,r,p,y) - F^{\blue{\epsilon}}(t,x,r,p,\overline{y}) \nonumber
 			\\&=2G \left ( pg^{\blue{\epsilon}}(t,x)+\frac{1}{2}y \left(h^{\blue{\epsilon}}(t,x)\right)^2\right)-2G \left ( pg^{\blue{\epsilon}}(t,x)+\frac{1}{2}\overline{y} \left(h^{\blue{\epsilon}}(t,x)\right)^2\right)  \nonumber \\
 			&\geq 2 \beta \left( pg^{\blue{\epsilon}}(t,x) + \frac{1}{2} y \left(h^{\blue{\epsilon}}(t,x)\right)^2-pg^{\blue{\epsilon}}(t,x) - \frac{1}{2} \overline{y}\left(h^{\blue{\epsilon}}(t,x)\right)^2 \right) \nonumber \\
 			&=\beta  \left(h^{\blue{\epsilon}}(t,x)\right)^2 (y-\overline{y})>0,\label{eq:Non-degenerate}
		\end{align}
	where we used $(\ref{uniformlyElliptic})$ for the first inequality and the fact that $h|_Q$ is bounded away from $0$ for the strict inequality \blue{in \eqref{eq:Non-degenerate}}. Note that in order to ensure non-degeneracy, it would be enough to assume $h(t,x) \neq 0$ for every $(t,x) \in Q$. However, we need the restriction $h|_Q$ to be bounded away from zero later on, see \eqref{eq:conditionAGeneral}.

\begin{prop} \label{prop:ExistenceExtendedSolution}
\blue{Let $\varphi \in C([0,T] \times \mathbb R^+)$ satisfy Assumption \ref{asum:phi} and $(\varphi^{\epsilon})_{\blue{\epsilon \in (0,1)}}$ be the {family} of functions constructed in Definition \ref{def:ApproximationPayoff}. Moreover, let Assumption \ref{asumC1,2General} hold.}
 \black{Let} $\epsilon \in (0,1)$ and $u^{\epsilon}$ be the unique viscosity solution of the PDE $(\ref{PDEExtendedDomain1})$-$(\ref{PDEExtendedDomain2})$. Then $u^{\epsilon} \in C([0,T] \times \blue{\mathbb{R}^+})$ and $ \blue{\| u^{\epsilon} \|_{\infty}} \leq {M_0+1}$ on $Q$, where {$M_0$ is the constant introduced in Assumption \ref{asum:phi}}. Moreover, there exists $\beta \in (0,1)$ such that $u^{\epsilon} \in C^{\black{1},2+\beta}([0,T] \times [\epsilon,\epsilon^{-1}])$.
	\end{prop}
	
\begin{proof}
\black{Let} $\epsilon \in (0,1)$. 
 Equation \eqref{PDEExtendedDomain1} can be written as
	$$
	u_t + F^{\epsilon}(t,x,u,u_x,u_{xx}) = 0, \quad \text{ for } (t,x) \in Q,
	$$
where
\begin{equation*}
	F^{\epsilon}(t,x,u,u_x,u_{xx})=\sup_{\sigma \in [\underline{\sigma}^2, \overline{\sigma}^2]} \bigg [  u_x g^{\epsilon}(t,x)\sigma + \frac{1}{2} u_{xx} \left(h^{\epsilon}(t,x)\right)^2 \sigma \bigg] +f^{\epsilon}(t,x)u_x -\left(\vartheta^{\epsilon}(x)+\blue{\epsilon}\right)u.
\end{equation*}
Note that it holds
\begin{align}
	F^{\epsilon}(t,x,u,u_x,u_{xx})&= \sup_{\sigma \in [\underline{\sigma}^2, \overline{\sigma}^2]} F_{\sigma}^{\epsilon}(t,x,u,u_x,u_{xx}) \nonumber \\
	&:=\sup_{\sigma \in [\underline{\sigma}^2, \overline{\sigma}^2]} \bigg [ \frac{1}{2} \left(h^{\epsilon}(t,x)\right)^2 \sigma u_{xx} +  \left ( g^{\epsilon}(t,x)\sigma +f^{\epsilon}(t,x) \right)u_x -\left(\vartheta^{\epsilon}(x)+\blue{\epsilon}\right)u \bigg]\black{.} \label{DefinitionABGeneral} 
\end{align}
Since $F^{\epsilon}_{\sigma}$ is linear in the arguments $u,u_{x},u_{xx}$, it satisfies condition $(\ref{symmetryProperty})$. Thus, by Lemma \ref{lem:supinf}, the PDE  $(\ref{PDEExtendedDomain1})$-$(\ref{PDEExtendedDomain2})$ has a $C^{1,2}$-solution if and only if the problem
\begin{align*}
	\overline{u}_t + \inf_{\sigma \in [\underline{\sigma}^2, \overline{\sigma}^2]} \bigg[ \frac{1}{2} \left(h^{\epsilon}(t,x)\right)^2 \sigma \overline{u}_{xx} +  \left ( g^{\epsilon}(t,x)\sigma +f^{\epsilon}(t,x) \right)\overline{u}_x -\left(\vartheta^{\epsilon}(x)+\blue{\epsilon}\right)\overline{u}\bigg]=0, \quad &\text{for } (t,x) \in Q\\
	\overline{u}=-\varphi^{\blue{\epsilon}} \quad &\text{for }(t,x)\in \blue{\partial' Q}  
\end{align*}
has a $C^{1,2}$-solution. \black{By applying Example \ref{ExampleKrylov}} we now show that for all $\sigma \in [\underline{\sigma}^2,\overline{\sigma}^2]$ the function $\tilde{F}^{\epsilon}_{\sigma}(t,x,{u},{u},{u}_x,{u}_{xx})$ given by
	\begin{align}
		\tilde{F}^{\epsilon}_{\sigma}(t,x,{u},{u}_x,{u}_{xx})= \frac{1}{2}\left(h^{\epsilon}(t,x)\right)^2 \sigma u_{xx}+ \left(g^{\epsilon}(t,x)\sigma+f^{\epsilon}(t,x)\right)u_x-\left(\vartheta^{\epsilon}(x)+\blue{\epsilon}\right)u \label{eq:FepsilonGeneral}
	\end{align}
	is an element in $\overline{F}(\eta,K,Q)$. \blue{We set}
	\begin{align}
		a^{\epsilon}(\sigma,t,x)&:=\frac{1}{2}\left(h^{\epsilon}(t,x)\right)^2 \sigma \label{eq:aEpsilon} \\
		b^{\epsilon}(\sigma,t,x,y,z)&:=\left[g^{\epsilon}(t,x)\sigma+f^{\epsilon}(t,x)\right]z-\left(\vartheta^{\epsilon}(x)+\blue{\epsilon}\right)y \label{eq:bEpsilon}.
	\end{align}	
	Note that by Lemma \ref{lemma:RegularityCutoffFunctions} the functions $a^{\epsilon}$ and $b^{\epsilon}$ are continuously differentiable\black{\footnote{See computations in Section A.2 in Appendix.}} with respect to $t,x,y,z$ and twice continuously differentiable with respect to $x,y,z$. \\
By Lemma \ref{lemma:RegularityCutoffFunctions} we obtain that the second derivatives of $a^{\epsilon} $ and $b^{\epsilon}$ with respect to $x,y,z$, \black{i.e., $a^{\epsilon}_{xx}$ and $b^{\epsilon}_{n,m}$ for $n,m \in \lbrace x,y,z \rbrace$,} and the first derivative with respect to $t$ are bounded on $\tilde{P}(M):= \lbrace (\sigma,t,x,y,z): \sigma \in [\underline{\sigma}^2,\overline{\sigma}^2], \vert z \vert + \vert y \vert \leq M, (t,x) \in Q \rbrace$. \\ 
We now show that for all $\sigma \in [\underline{\sigma}^2, \overline{\sigma}^2], (t,x) \in Q, \lambda \in \mathbb{R}$, inequalities $(\ref{InequalityLinear1})$-$(\ref{InequalityLinear3})$ are satisfied.
\blue{We first prove that}
\begin{equation}\label{eq:conditionAGeneral}
\eta \leq  a^{\epsilon}(\sigma,t,x) \leq K
\end{equation}
for some $\eta>0,K>0$.
By definition of $a^{\epsilon}(\sigma,t,x)$, Assumption \ref{asumC1,2General} and Lemma \ref{lemma:RegularityCutoffFunctions} this is equivalent to show
\begin{equation}
	2 \eta / C^2 \leq \sigma \leq 2 K/C^2. \label{BoundsSigmaGeneral}
\end{equation}
By condition $(\ref{uniformlyElliptic})$ applied to $-1=\overline{y}\leq y=0$ there exists $\beta>0$ such that $\beta \leq \frac{\underline{\sigma}^2}{2}$. As $\underline{\sigma}^2 \leq \sigma$, we can find $\eta>0$ such that the first inequality in $(\ref{BoundsSigmaGeneral})$ holds. Moreover, the second inequality follows by $\sigma \leq \overline{\sigma}^2$.

Next, we need to find an upper bound for $\vert b \vert$. We have that
\begin{align}
	\vert b^{\epsilon}(\sigma,t,x,y,z) \vert& = \vert z \big (g^{\epsilon}(t,x) \sigma + f^{\epsilon}(t,x) \big) - (\vartheta^{\epsilon}(x)+\blue{\epsilon})y \vert \leq  \vert z \vert \vert g^{\epsilon}(t,x) \sigma + f^{\epsilon}(t,x)\vert  + \vert \vartheta^{\epsilon}(x)+\blue{\epsilon} \vert \vert y \vert \nonumber \\
	&\leq C_1\vert z \vert + C_2 \vert y \vert  \leq (1+ \vert z \vert^2) C_1 + (1+ \vert z \vert^2) C_2 \vert y \vert \nonumber \\ 
	&=  \underbrace{(C_1 + C_2 \vert y \vert)}_{:=\tilde{M}_1(y)} \big (1+ \vert z \vert ^2 \big), \label{eq:EstitmateBGeneral}
\end{align}
where the constants $C_1>0, C_2>0$ exist by Lemma \ref{lemma:RegularityCutoffFunctions}. Moreover, we used that $\vert z \vert \leq 1+ \vert z \vert^2$. 
Inequality $(\ref{InequalityLinear2})$ holds, as 
\begin{align}
	\vert a^{\epsilon}_{\blue{x}}(\sigma,t,x) \vert = \vert 2 \sigma h^{\epsilon}(t,x) h^{\epsilon}_{\blue{x}}(t,x) \vert \leq C_3,
\end{align}
by Lemma \ref{lemma:RegularityCutoffFunctions}. We now prove $(\ref{InequalityLinear3})$. It holds
\begin{align}
	\vert b^{\epsilon}_{\blue{z}}(\sigma,t,x,y,z)\vert (1+ \vert z \vert ) +  \vert   b^{\epsilon}_{\blue{y}}(\sigma,t,x,y,z) \vert + \vert  b^{\epsilon}_{\blue{x}}(\sigma,t,x,y,z) \vert (1+\vert z \vert)^{-1} = \nonumber \\
	\vert g^{\epsilon}(t,x)\sigma + f^{\epsilon}(t,x) \vert (1+ \vert z \vert )+ \vert \vartheta^{\epsilon}(x) \blue{+ \epsilon} \vert + \vert z \big[  g^{\epsilon}_{\blue{x}}(t,x) \sigma +  f^{\epsilon}_{\blue{x}}(t,x) \big ] - y \vartheta^{\epsilon}_{\blue{x}}(x) \vert (1+\vert z \vert)^{-1} \leq \nonumber \\
	C_1 (1+ \vert z \vert )+ C_2 + ( \vert z \vert  C_4 + C_2\vert  y \vert) (1+\vert z \vert)^{-1} \leq   \vert z \vert (C_1 + C_4)+ C_2(1+\vert y \vert)+C_1 \leq \nonumber \\
	\underbrace{( 2C_1 + C_4+ (1+ C_2)\vert y \vert )}_{:=\tilde{M}_2(y)}(1+ \vert z \vert^2 ),
\end{align}
for a constant $C_4>0$ which we get from Lemma \ref{lemma:RegularityCutoffFunctions}. It is obvious that for all $y$ we have $\tilde{M}_1(y) \leq \tilde{M}_2(y)$. Thus, it follows by $(\ref{eq:EstitmateBGeneral})$ that $\vert b^{\epsilon}(\sigma,t,x,y,z) \vert \leq \tilde{M}_2(y) (1+\vert z \vert^2)$, which is a necessary condition in Example \ref{ExampleKrylov}. \\
For the inequalities in $(\ref{InequalityLinear4})$ fix {$M_0>0$} such that {$\| \varphi^{\epsilon} \|_{\infty} \leq M_0+1$} and choose $0<\delta_0 \leq {2M_0}$. As $\vartheta^{\epsilon}(x)>0$ it holds
\begin{align*}
	b^{\epsilon}(\sigma, t, x,-M_0,0)=(\vartheta^{\epsilon}(x)+\blue{\epsilon}){M_0} \geq \delta_0  \\
	b^{\epsilon}(\sigma, t, x,M_0,0)=-(\vartheta^{\epsilon}(x)+\blue{\epsilon}){M_0} \leq -\delta_0. 
\end{align*}
By Example \ref{ExampleKrylov} and Theorem \ref{theoremInf} we can conclude that $\tilde{F}^{\epsilon} \in \overline{F}(\eta,K,Q)$. \\
Then Theorem \ref{Theorem3} and \blue{Point 3. of Lemma \ref{lemma:RegularityConvergencePayoff} imply the existence of a solution \blue{$u^{\epsilon}$} for the PDE $(\ref{PDEExtendedDomain1})$-$(\ref{PDEExtendedDomain2})$} such that \blue{$u^{\epsilon} \in C(\overline{Q})$ and} $\| u^{\epsilon} \|_{\infty} \leq {M_0+1}$ on $Q$. \blue{Furthermore, by \blue{Point 2.} of  Theorem \ref{Theorem3} together with Remark 2.1.3 in \cite{peng_nonlinearExpectation_book} it follows that $u^{\epsilon}$ is a viscosity solution.}\\ Moreover, we can apply Point \blue{3}. of Theorem \ref{Theorem3}, \blue{by setting} $D^0:=\blue{(\epsilon,\epsilon^{-1})}$, $D^1:=(\epsilon/2, 2\epsilon^{-1})$ and $D=(0,\infty)$, to show that there exists $\beta \in (0,1)$ such that $u^{\epsilon} \in C^{\black{1},2+\beta}([0,T] \times [\epsilon, \epsilon^{-1}])$. 
Indeed, by definition of $D^0,D^1$ we have that 
 $\rho:=\text{dist}(\partial D^0, \partial D^1)=\text{dist}( \lbrace \epsilon, \epsilon^{-1} \rbrace, \lbrace \epsilon/2, 2\epsilon^{-1}  \rbrace)>0$ as $0 < \epsilon/2 <\epsilon< \epsilon^{-1}<2 \epsilon^{-1}$, and \blue{by} \blue{point 4. of Lemma \ref{lemma:RegularityConvergencePayoff} that} $\varphi^{\blue{\epsilon}}(T, \cdot) \in C^{\infty}(D^1)$. So we have verified that the conditions of \blue{P}oint \blue{3}. of Theorem \ref{Theorem3} are satisfied.\end{proof}
		
For fixed $\epsilon	\in (0,1)$, define the intervals $D_{\epsilon}:=(\epsilon, \epsilon^{-1}),\overline{D}_{\epsilon}:=[\epsilon, \epsilon^{-1}]$ and the random time
\begin{equation}
	\tau_{\epsilon}^{\blue{r,x}}:=\inf \lbrace s \in [\blue{r},T]: X_s^{\blue{r,x}} \notin [\epsilon, \epsilon^{-1} ] \rbrace \wedge T, \label{eq:TauEpsilon}
\end{equation}
where we use the convention $\inf \emptyset:=+\infty$. 
We now state the following result about stopping times.
{\begin{lemma}\label{QuasiContinuity}
\blue{Given $(r,x) \in [0,T] \times (0,\infty)$,} let $X^{\blue{r,x}}=(X_t^{\blue{r,x}})_{t \in [\blue{r},T]}$ be the $G$-It\^{o} process in \eqref{G-SDE_New}. Fix $\epsilon \in (0,1)$.
Assume that Assumption \ref{asumC1,2General} is satisfied. Then $\tau_{\epsilon}^{\blue{r,x}}$ is a quasi-continuous stopping time.
\end{lemma}
\begin{proof}
	Fix $\epsilon \in (0,1)$ \blue{ and define} the auxiliary stopping time $\tau_{\overline{D}_{\epsilon}}^{\blue{r,x}}:=\inf \lbrace s \in [\blue{r},T]: X_s^{\blue{r,x}} \notin {\overline{D}_{\epsilon}} \rbrace$, \blue{so} that $\tau_{\epsilon}^{\blue{r,x}}=\tau_{\overline{D}_{\epsilon}}^{\blue{r,x}} \wedge T$. \blue{By} Lemma \ref{ExampleLiu} \blue{and Remark \ref{remark:StoppingTimeGreater}} it follows that the process $X^{\blue{r,x}}$ is quasi-continuous, and as $X^{\blue{r,x}}$ is also $\mathbb{F}$-adapted, the exit time $\tau_{\overline{D}_{\epsilon}}^{\blue{r,x}}$ is a stopping time.
Moreover, Assumption \ref{LocalGrowthCondition} is satisfied, again by Lemma \ref{ExampleLiu} applied to $Y=X^{r,x}$. \blue{Thus, it follows directly by Proposition \ref{cor:StoppingTimeMaxQuasicontitnuous} that $\tau_{\epsilon}^{\blue{r,x}}$ is quasi-continuous as $T>0$ is \blue{constant}.}
\end{proof}}

\begin{prop} \label{prop:ExtendedMartingaleProperty}
\blue{Let $\varphi \in C([0,T] \times \mathbb R^+)$ satisfy Assumption \ref{asum:phi} and $(\varphi^{\epsilon})_{\blue{\epsilon \in (0,1)}}$ be the \blue{family} of functions constructed in Definition \ref{def:ApproximationPayoff}. \blue{Given $(r,x) \in [0,T] \times (0,\infty)$, consider the solution $X^{\blue{r,x}}$ of the $G$-SDE \eqref{G-SDE_New}. Moreover, let Assumption \ref{asumC1,2General} hold}} and $u^{\epsilon}$ be the viscosity solution of the PDE $(\ref{PDEExtendedDomain1})-(\ref{PDEExtendedDomain2})$, \blue{for given $\epsilon \in (0,1)$}. Then \blue{for every $(r,x) \in [0,T] \times (0,\infty)$, we have that}
	\begin{equation}
		M_t^{\epsilon,\blue{r,x}}:=u^{\epsilon}(t \wedge \tau_{\epsilon}^{\blue{r,x}}, X^{\blue{r,x}}_{t \wedge \tau_{\epsilon}^{\blue{r,x}}})e^{-\int_{\blue{r}}^{t \wedge \tau_{\epsilon}^{\blue{r,x}}}(X_s^{\blue{r,x}} + \blue{\epsilon}) ds}, \quad \blue{ r \leq t \leq T} \label{eq:MartingaleProcessEpsilon}
	\end{equation}
	is a \blue{$G$-martingale}, where the stopping time $\tau_{\epsilon}^{\blue{r,x}}$ is given in $(\ref{eq:TauEpsilon})$.
\end{prop}

\begin{proof}
Fix $\epsilon \in (0,1)$.  We need to verify that the stopped process $X^{\tau_{\epsilon}^{\blue{r,x}}}:=\left(X_t^{\tau_{\epsilon}^{\blue{r,x}}}\right)_{t \in [\blue{r},T]}$ with $X_{t}^{\tau_{\epsilon}^{\blue{r,x}}}:=X_{t \wedge \tau_{\epsilon}^{\blue{r,x}}}^{\blue{r,x}}$ for $t \in [\blue{r},T]$ is well-defined. By definition of $X^{\blue{r,x}}$ in $(\ref{G-SDE_New})$ we have to look at
\begin{equation}
	X_t^{\tau_{\epsilon}^{\blue{r,x}}}=\blue{x}+\int_{\blue{r}}^{t \wedge \tau_{\epsilon}^{\blue{r,x}}} f(s,X_s^{\blue{r,x}})d\blue{s} + \int_{\blue{r}}^{t \wedge \tau_{\epsilon}^{\blue{r,x}}} g(s,X_s^{{\blue{r,x}}}) d\langle B \rangle_s + \int_{\blue{r}}^{t \wedge \tau_{\epsilon}^{\blue{r,x}}} h(s,X_s^{\blue{r,x}})dB_s, \quad \blue{r} \leq t \leq T. \label{G-SDETauEpsilon}
\end{equation}
Since $\tau_{\epsilon}^{\blue{r,x}}$ {is quasi-continuous by Lemma \ref{QuasiContinuity}, by Proposition \ref{PropositionLiu}} we get \blue{for $r \leq t \leq T$}
\begin{equation}
	X_t^{\tau_{\epsilon}^{\blue{r,x}}}=\blue{x}+\int_{\blue{r}}^{t} f(s,X_s^{\blue{r,x}})\textbf{1}_{[0,\tau_{\epsilon}^{\blue{r,x}}]}(s) d\blue{s} + \int_{\blue{r}}^{t } g(s,X_s^{\blue{r,x}}) \textbf{1}_{[0,\tau_{\epsilon}^{\blue{r,x}}]}(s) d\langle B \rangle_s + \int_{\blue{r}}^{t } h(s,X_s^{\blue{r,x}})\textbf{1}_{[0,\tau_{\epsilon}^{\blue{r,x}}]}(s)dB_s. \label{G-SDETauRewrittenEpsilon}
\end{equation}
Moreover, by \black{Proposition \ref{PropositionLiu} it follows that $f(s,X_s^{r,x})\textbf{1}_{[r,\tau_{\epsilon}^{r,x}]}(s),  g(s,X_s^{r,x})\textbf{1}_{[r,\tau_{\epsilon^{r,x}}]}(s), \newline h(s,X_s^{r,x}) \textbf{1}_{[r,\tau_{\epsilon}^{r,x}]}(s) \in M_G^2(0,T)$. }\\
By the definition of $\tau_{\epsilon}^{\blue{r,x}}$ in $(\ref{eq:TauEpsilon})$ we know that $X_t^{\tau_{\epsilon}^{\blue{r,x}}} \in \black{[}\epsilon, \epsilon^{-1}\black{]}$ for $t \in [\blue{r}, T]$. Moreover, $u^{\epsilon} \in C^{1,2}([0,T] \times [\epsilon, \epsilon^{-1}])$ by Proposition \ref{prop:ExistenceExtendedSolution}. Thus, we can apply G-It\^{o}'s formula from Theorem \violet{3.6.3 in \cite{peng_nonlinearExpectation_book}} to the process $M^{{\epsilon},\blue{r,x}}=(M_t^{{\epsilon},\blue{r,x}})_{t \in [\blue{r},T] }$ given by 
$$
M^{\epsilon,\blue{r,x}}_t=u^{\epsilon}(t \wedge \tau_{\epsilon}^{\blue{r,x}}, X^{\blue{r,x}}_{t \wedge \tau_{\epsilon}^{\blue{r,x}}})e^{-\int_{\blue{r}}^{t \wedge \tau_{\epsilon}^{\blue{r,x}}}(X_s^{\blue{r,x}} + \blue{\epsilon}) ds}.
$$
 Note that the conditions of Theorem \violet{3.6.3 in \cite{peng_nonlinearExpectation_book}} are satisfied as $u^{\epsilon} \in C^{1,2}([0,T] \times [\epsilon, \epsilon^{-1}])$ implies that $u^{\epsilon}_{\blue{xx}}$ is bounded on $[0,T] \times [\epsilon, \epsilon^{-1}]$ and therefore satisfies a polynomial growth condition. Moreover, $f(s,X_s^{\blue{r,x}})\textbf{1}_{[\blue{r},\tau_{\epsilon}^{\blue{r,x}}]}(s),g(s,X_s^{\blue{r,x}})\textbf{1}_{[\blue{r},\tau_{\epsilon}^{\blue{r,x}}]}(s), h(s,X_s^{\blue{r,x}}) \textbf{1}_{[\blue{r},\tau_{\epsilon}^{\blue{r,x}}]}(s)$ are bounded, as $X_s^{\blue{r,x}} \in [\epsilon, \epsilon^{-1}]$ for any $s \in [\blue{r},\tau_{\epsilon}^{\blue{r,x}}]$ and then Assumption \ref{asumC1,2General} implies \black{that} $f \vert_{[0,T] \times [\epsilon, \epsilon^{-1}]}, g \vert_{[0,T] \times [\epsilon, \epsilon^{-1}]},\linebreak h \vert_{[0,T] \times [\epsilon, \epsilon^{-1}]} $ are bounded. So we get 
\allowdisplaybreaks{
\begin{align}
	dM_{t}^{\epsilon,\blue{r,x}}&=u^{\epsilon}(t\wedge \tau_{\epsilon}^{\blue{r,x}},X^{\blue{r,x}}_{t \wedge \tau_{\epsilon}^{\blue{r,x}}})de^{-\int_{\blue{r}}^{t \wedge \tau_{\epsilon}^{\blue{r,x}}} (X_{s}^{\blue{r,x}}+\blue{\epsilon}) ds} + e^{-\int_{\blue{r}}^{t \wedge \tau_{\epsilon}^{\blue{r,x}}} (X_{s}^{\blue{r,x}}+\blue{\epsilon}) ds} du^{\epsilon}(t \wedge \tau_{\epsilon}^{\blue{r,x}},X^{\blue{r,x}}_{t \wedge \tau_{\epsilon}^{\blue{r,x}}}) \nonumber\\
	&= e^{-\int_{\blue{r}}^{t \wedge \tau_{\epsilon}^{\blue{r,x}}} (X_{s}^{\blue{r,x}}+\blue{\epsilon}) ds}\big[u_t^{\epsilon}(t \wedge \tau_{\epsilon}^{\blue{r,x}},X^{\blue{r,x}}_{t \wedge \tau_{\epsilon}^{\blue{r,x}}})+u_x^{\epsilon}(t \wedge \tau_{\epsilon}^{\blue{r,x}},X_{t \wedge \tau_{\epsilon}^{\blue{r,x}}}^{\blue{r,x}})f(t \wedge \tau_{\epsilon}^{\blue{r,x}},X_{t \wedge \tau_{\epsilon}^{\blue{r,x}}}^{\blue{r,x}}) \nonumber \\
	&\quad -u^{\epsilon}(t \wedge \tau_{\epsilon}^{\blue{r,x}},X_{t \wedge \tau_{\epsilon}^{\blue{r,x}}}^{\blue{r,x}})(X_{t \wedge \tau_{\epsilon}^{\blue{r,x}}}^{\blue{r,x}}+\blue{\epsilon})\big]dt \nonumber \\
	&\quad+e^{-\int_{\blue{r}}^{t \wedge \tau_{\epsilon}^{\blue{r,x}}} (X_{s}^{\blue{r,x}}+\blue{\epsilon}) ds}h(t \wedge \tau_{\epsilon}^{\blue{r,x}},X^{\blue{r,x}}_{t \wedge \tau_{\epsilon}^{\blue{r,x}}})u_x^{\epsilon}(t \wedge \tau_{\epsilon}^{\blue{r,x}},X^{\blue{r,x}}_{t \wedge \tau_{\epsilon}^{\blue{r,x}}})dB_t \nonumber\\
	&\quad+ e^{-\int_{\blue{r}}^{t \wedge \tau_{\epsilon}^{\blue{r,x}}} (X_{s}^{\blue{r,x}}+\blue{\epsilon}) ds}\big[ u_x^{\epsilon}(t \wedge \tau_{\epsilon}^{\blue{r,x}},X^{\blue{r,x}}_{t \wedge \tau_{\epsilon}^{\blue{r,x}}}) g(t \wedge \tau_{\epsilon}^{\blue{r,x}},X^{\blue{r,x}}_{t \wedge \tau_{\epsilon}^{\blue{r,x}}}) \nonumber\\
	&\quad +\frac{1}{2}u_{xx}^{\epsilon}(t \wedge \tau_{\epsilon}^{\blue{r,x}},X^{\blue{r,x}}_{t \wedge \tau_{\epsilon}^{\blue{r,x}}})(h(t \wedge \tau_{\epsilon}^{\blue{r,x}},X^{\blue{r,x}}_{t \wedge \tau_{\epsilon}^{\blue{r,x}}}))^2\big]d\langle B \rangle_t \label{eq:SwitchToSmoothedFunctions} \\
	&= e^{-\int_{\blue{r}}^{t \wedge \tau_{\epsilon}^{\blue{r,x}}} (X_{s}^{\blue{r,x}}+\blue{\epsilon}) ds}\big[u_t^{\epsilon}(t \wedge \tau_{\epsilon}^{\blue{r,x}},X^{\blue{r,x}}_{t \wedge \tau_{\epsilon}^{\blue{r,x}}})+u_x^{\epsilon}(t \wedge \tau_{\epsilon}^{\blue{r,x}},X^{\blue{r,x}}_{t \wedge \tau_{\epsilon}^{\blue{r,x}}})f^{\epsilon}(t \wedge \tau_{\epsilon}^{\blue{r,x}},X^{\blue{r,x}}_{t \wedge \tau_{\epsilon}^{\blue{r,x}}}) \nonumber \\
	&\quad-u^{\epsilon}(t \wedge \tau_{\epsilon}^{\blue{r,x}},X^{\blue{r,x}}_{t \wedge \tau_{\epsilon}^{\blue{r,x}}})(\vartheta^{\epsilon}(X^{\blue{r,x}}_{t \wedge \tau_{\epsilon}^{\blue{r,x}}})+\blue{\epsilon})\big]dt \nonumber \\
	&\quad+ e^{-\int_{\blue{r}}^{t \wedge \tau_{\epsilon}^{\blue{r,x}}} (X_{s}^{\blue{r,x}}+\blue{\epsilon}) ds}h^{\epsilon}(t \wedge \tau_{\epsilon}^{\blue{r,x}},X^{\blue{r,x}}_{t \wedge \tau_{\epsilon}^{\blue{r,x}}})u_x^{\epsilon}(t \wedge \tau_{\epsilon}^{\blue{r,x}},X^{\blue{r,x}}_{t \wedge \tau_{\epsilon}^{\blue{r,x}}})dB_t \nonumber\\
	&\quad+ e^{-\int_{\blue{r}}^{t \wedge \tau_{\epsilon}^{\blue{r,x}}} (X_{s}^{\blue{r,x}}+\blue{\epsilon}) ds}\big[ u_x^{\epsilon}(t \wedge \tau^{\blue{r,x}},X^{\blue{r,x}}_{t \wedge \tau_{\epsilon}^{\blue{r,x}}}) g^{\epsilon}(t \wedge \tau_{\epsilon}^{\blue{r,x}},X^{\blue{r,x}}_{t \wedge \tau_{\epsilon}^{\blue{r,x}}}) \nonumber \\
	&\quad +\frac{1}{2}u_{xx}^{\epsilon}(t \wedge \tau_{\epsilon}^{\blue{r,x}},X^{\blue{r,x}}_{t \wedge \tau_{\epsilon}^{\blue{r,x}}})(h^{\epsilon}(t \wedge \tau_{\epsilon}^{\blue{r,x}},X^{\blue{r,x}}_{t \wedge \tau_{\epsilon}^{\blue{r,x}}}))^2\big]d\langle B \rangle_t \nonumber,
\end{align}}
where we used in $(\ref{eq:SwitchToSmoothedFunctions})$ that for $\phi \in \lbrace f,g,h \rbrace$ it holds $\phi^{\epsilon} \vert_{[0,T] \times [\epsilon, \epsilon^{-1}]}=\phi \vert_{[0,T] \times [\epsilon, \epsilon^{-1}]}$ and \newline $\vartheta^{\epsilon} \vert_{[0,T] \times [\epsilon, \epsilon^{-1}]}=id \vert_{[0,T] \times [\epsilon, \epsilon^{-1}]}$ for all $\epsilon \in (0,1)$ by \black{$(\ref{eq:DefinititonCutoff1}),(\ref{eq:DefinititonCutoff2})$ and $(\ref{eq:DefinititonCutoff3}).$}\\

For $\blue{r} \leq z \leq u \leq T$ we have
\allowdisplaybreaks{
\small{
\begin{align}
\hat{\mathbb{E}}_z[M_u^{\epsilon,\blue{r,x}}]&= \hat{\mathbb{E}}_z\bigg[u^{\epsilon}(\blue{r},X_{r}^{\blue{r,x}})+ \int_{\blue{r}}^{u \wedge \tau_{\epsilon}^{\blue{r,x}}} e^{-\int_{\blue{r}}^t (X_s^{\blue{r,x}}+\blue{\epsilon})  ds} h^{\epsilon}(t,X_t^{\blue{r,x}})u_x^{\epsilon}(t,X_t^{\blue{r,x}})dB_t \nonumber \\
&\quad+ \int_{\blue{r}}^{u \wedge \tau_{\epsilon}^{\blue{r,x}}} e^{-\int_{\blue{r}}^t (X_s^{\blue{r,x}}+\blue{\epsilon}) ds}\big[ u_t^{\epsilon}(t,X_t^{\blue{r,x}})+f^{\epsilon}(t,X_t^{\blue{r,x}})u_x^{\epsilon}(t,X_t^{\blue{r,x}})-(\vartheta(X_t^{\blue{r,x}})+\blue{\epsilon}) u^{\epsilon}(t,X_t^{\blue{r,x}})\big] dt  \nonumber \\
& \quad+ \int_{\blue{r}}^{u \wedge \tau_{\epsilon}^{\blue{r,x}}} e^{-\int_{\blue{r}}^t (X_s^{\blue{r,x}}+\blue{\epsilon})  ds}\big[u_x^{\epsilon}(t,X_t^{\blue{r,x}})g^{\epsilon}(t,X_t^{\blue{r,x}})+ \frac{1}{2}u_{xx}^{\epsilon}(t,X_t^{\blue{r,x}})(h^{\epsilon}(t,X_t^{\blue{r,x}}))^2\big]d\langle B \rangle_t \bigg]  \nonumber \\
&= \hat{\mathbb{E}}_z\bigg[u^{\epsilon}(\blue{r},\blue{X_r^{r,x}})+ \int_{\blue{r}}^{u}\textbf{1}_{[0,\tau_{\epsilon}^{\blue{r,x}}]}(t) e^{-\int_{\blue{r}}^t (X_s^{\blue{r,x}}+\blue{\epsilon})ds} h^{\epsilon}(t,X_t^{\blue{r,x}})u_x^{\epsilon}(t,X_t^{\blue{r,x}})dB_t \notag\\
&\quad+\int_{\blue{r}}^{u}\textbf{1}_{[0,\tau_{\epsilon}^{\blue{r,x}}]}(t) e^{-\int_{\blue{r}}^t(X_s^{\blue{r,x}}+\blue{\epsilon})ds}\big[ u_t^{\epsilon}(t,X_t^{\blue{r,x}})+f^{\epsilon}(t,X_t^{\blue{r,x}})u^{\epsilon}_x(t,X_t^{\blue{r,x}}) \nonumber \\
&\quad -(\vartheta(X_t^{\blue{r,x}})+\blue{\epsilon}) u^{\epsilon}(t,X_t^{\blue{r,x}})\big] dt  \nonumber \\
&\quad+ \int_{\blue{r}}^{u} \textbf{1}_{[0,\tau_{\epsilon}^{\blue{r,x}}]}(t)e^{-\int_{\blue{r}}^t(X_s^{\blue{r,x}}+\blue{\epsilon})ds}\big[u_x^{\epsilon}(t,X_t^{\blue{r,x}})g^{\epsilon}(t,X_t^{\blue{r,x}})+ \frac{1}{2}u_{xx}^{\epsilon}(t,X_t^{\blue{r,x}})(h^{\epsilon}(t,X_t^{\blue{r,x}}))^2\big]d\langle B \rangle_t   \bigg] \label{TechniquesStep2Epsilon}   \\
&= M_z^{\epsilon,\blue{r,x}} +  \hat{\mathbb{E}}_z \bigg[ \int_z^u \textbf{1}_{[0,\tau_{\epsilon}^{\blue{r,x}}]}(t) e^{-\int_{\blue{r}}^t (X_s^{\blue{r,x}}+\blue{\epsilon})ds} h^{\epsilon}(t,X_t^{\blue{r,x}})u_x^{\epsilon}(t,X_t^{\blue{r,x}})dB_t    \nonumber \\
&\quad +  \int_z^u \textbf{1}_{[0,\tau_{\epsilon}^{\blue{r,x}}]}(t) e^{-\int_{\blue{r}}^t (X_s^{\blue{r,x}}+\blue{\epsilon}) ds}\big[u_t^{\epsilon}(t,X_t^{\blue{r,x}})+f^{\epsilon}(t,X_t^{\blue{r,x}})u_x^{\epsilon}(t,X_t^{\blue{r,x}})\nonumber \\
&\quad -(\vartheta(X_t^{\blue{r,x}})+\blue{\epsilon}) u^{\epsilon}(t,X_t^{\blue{r,x}})\big] dt  \nonumber\\
&\quad + \int_z^u \textbf{1}_{[0,\tau_{\epsilon}^{\blue{r,x}}]}(t) e^{-\int_{\blue{r}}^t (X_s^{\blue{r,x}}+\blue{\epsilon}) ds}\big[u_x^{\epsilon}(t,X_t^{\blue{r,x}})g^{\epsilon}(t,X_t^{\blue{r,x}})+ \frac{1}{2}u_{xx}^{\epsilon}(t,X_t^{\blue{r,x}})(h^{\epsilon}(t,X_t^{\blue{r,x}}))^2\big]d\langle B \rangle_t\bigg]  \label{SubadditivityEpsilon}\\
&= M_z^{\epsilon,\blue{r,x}} +  \hat{\mathbb{E}}_z\bigg[ \int_z^u \textbf{1}_{[0,\tau_{\epsilon}^{\blue{r,x}}]}(t) e^{-\int_{\blue{r}}^t (X_s^{\blue{r,x}}+\blue{\epsilon}) ds}\big[u_t^{\epsilon}(t,X_t^{\blue{r,x}})+f^{\epsilon}(t,X_t^{\blue{r,x}})u_x^{\epsilon}(t,X_t^{\blue{r,x}}) \nonumber \\
&\quad -(\vartheta(X_t^{\blue{r,x}})+\epsilon) u^{\epsilon}(t,X_t^{\blue{r,x}})\big] dt  \nonumber \\
&\quad + \int_z^u \textbf{1}_{[0,\tau_{\epsilon}^{\blue{r,x}}]}(t) e^{-\int_{\blue{r}}^t (X_s^{\blue{r,x}}+\blue{\epsilon}) ds}\big[u_x^{\epsilon}(t,X_t^{\blue{r,x}})g^{\epsilon}(t,X_t^{\blue{r,x}})+ \frac{1}{2}u_{xx}^{\epsilon}(t,X_t^{\blue{r,x}})(h^{\epsilon}(t,X_t^{\blue{r,x}}))^2\big]d\langle B \rangle_t\bigg], \label{IntegralZero} \\
&= M_z^{\epsilon,\blue{r,x}} +  \hat{\mathbb{E}}_z\bigg[ \int_z^u \textbf{1}_{[0,\tau_{\epsilon}^{\blue{r,x}}]}(t) e^{-\int_{\blue{r}}^t (X_s^{\blue{r,x}}+\blue{\epsilon}) ds}\big[u_t^{\epsilon}(t,X_t^{\blue{r,x}})+f^{\epsilon}(t,X_t^{\blue{r,x}})u_x^{\epsilon}(t,X_t^{\blue{r,x}}) \nonumber \\
&\quad-(\vartheta(X_t^{\blue{r,x}})+\blue{\epsilon}) u^{\epsilon}(t,X_t^{\blue{r,x}})\big] dt  \nonumber \\
&\quad  \blue{+} \int_z^u \textbf{1}_{[0,\tau_{\epsilon}^{\blue{r,x}}]}(t) e^{-\int_{\blue{r}}^t (X_s^{\blue{r,x}}+\blue{\epsilon}) ds}\big[u_x^{\epsilon}(t,X_t^{\blue{r,x}})g^{\epsilon}(t,X_t^{\blue{r,x}})+ \frac{1}{2}u_{xx}^{\epsilon}(t,X_t^{\blue{r,x}})(h^{\epsilon}(t,X_t^{\blue{r,x}}))^2\big]d\langle B \rangle_t \nonumber \\
&\quad \pm \int_z^u \textbf{1}_{[0,\tau_{\epsilon}^{\blue{r,x}}]}(t) 2G(e^{-\int_{\blue{r}}^t (X_s^{\blue{r,x}}+\blue{\epsilon}) ds}[u_x^{\epsilon}(t,X_t^{\blue{r,x}})g^{\epsilon}(t,X_t^{\blue{r,x}})+ \frac{1}{2}u_{xx}^{\epsilon}(t,X_t^{\blue{r,x}})(h^{\epsilon}(t,X_t^{\blue{r,x}}))^2])dt \bigg], \label{align:SupermartingaleExtended}
\end{align}}
}
where we used Proposition \ref{PropositionLiu} in $(\ref{TechniquesStep2Epsilon})$ and Remark \violet{3.2.4 in \cite{peng_nonlinearExpectation_book}} together with the fact that $M_z^{\epsilon,\blue{r,x}}$ is $\mathcal{F}_z$-measurable in $(\ref{SubadditivityEpsilon})$ and Proposition \violet{3.3.6. in \cite{peng_nonlinearExpectation_book}} in $(\ref{IntegralZero})$. \\
By Proposition \violet{3.2.9 in \cite{peng_nonlinearExpectation_book}} for $\overline{Y},Y \in L_G^1(\Omega_T)$ such that $\hat{\mathbb{E}}_z[Y]=-\hat{\mathbb{E}}_z[-Y]$, we have
	\begin{equation}
		\hat{\mathbb{E}}_z[\overline{Y}+ Y]= \hat{\mathbb{E}}_z[\overline{Y}]+ \hat{\mathbb{E}}_z[Y]. \label{eq:AdditivityExtended}
	\end{equation}
We now apply this result to $(\ref{align:SupermartingaleExtended})$ for 
\begin{equation*}
	Y:= \int_z^u \textbf{1}_{[0,\tau_{\epsilon}^{\blue{r,x}}]}(t) e^{-\int_{\blue{r}}^t (X_s^{\blue{r,x}}+\blue{\epsilon}) ds}  Z(t,X_t^{\blue{r,x}})dt  
\end{equation*}
with
\begin{align*}
Z(t,X_t^{\blue{r,x}})&:=
u_t^{\epsilon}(t,X_t^{\blue{r,x}})+2G \big ( u_x^{\epsilon}(t,X_t^{\blue{r,x}})g^{\epsilon}(t,X_t^{\blue{r,x}})+
	\frac{1}{2}u_{xx}^{\epsilon}(t,X_t^{\blue{r,x}})(h^{\epsilon}(t,X_t^{\blue{r,x}}))^{2} \big) \\
	&\quad +f^{\epsilon}(t,X_t^{\blue{r,x}})u_x^{\epsilon}(t,X_t^{\blue{r,x}}) -(\vartheta(X_t^{\blue{r,x}})+\blue{\epsilon}) u^{\blue{\epsilon}}(t,X_t^{\blue{r,x}}).
\end{align*}
First, \black{note} that $Y \in L_G^1(\Omega_T)$, as $\int_z^{u} \textbf{1}_{[0,\tau_{\epsilon}^{\blue{r,x}}]}(t)e^{-\int_{\blue{r}}^t (X_s^{\blue{r,x}}+\blue{\epsilon}) ds}(u_t^{\epsilon}(t,X_t^{\blue{r,x}})+f^{\epsilon}(t,X_t^{\blue{r,x}})u_x^{\epsilon}(t,X_t^{\blue{r,x}})-X_t^{\blue{r,x}} u^{\epsilon}(t,X_t^{\blue{r,x}}))dt \in L_G^2(\Omega_T) \subset L_G^1(\Omega_T)$ by \violet{Theorem 3.6.3 in \cite{peng_nonlinearExpectation_book}} and Proposition \ref{PropositionLiu}. Moreover, by these results it also follows that for all $t \in [\blue{r},T]$
\begin{equation*}
	\textbf{1}_{[0,\tau_{\epsilon}^{\blue{r,x}}]}(t)e^{-\int_{\blue{r}}^t \blue{(X_s^{\blue{r,x}}+\epsilon)ds}}G(u_x^{\epsilon}(t,X_t^{\blue{r,x}})g^{\epsilon}(t,X_t^{\blue{r,x}})+1/2 u_{xx}^{\epsilon}(t,X_t^{\blue{r,x}})(h^{\epsilon}(t,X_t^{\blue{r,x}}))^2) \in M_G^1(0,T).
\end{equation*}
Thus, as in  Proposition \violet{4.1.4 in \cite{peng_nonlinearExpectation_book}} it holds $Y \in L_G^1(\Omega_T)$.
\black{As} $u^{\epsilon}$ is the solution of the PDE in $(\ref{PDEExtendedDomain1})$ and $X_t^{\blue{r,x}}(\omega) \in \black{[}\epsilon, \epsilon^{-1}\black{]}$ for $t \leq \tau_{\epsilon}^{\blue{r,x}}(\omega)$ for every $\omega \in \Omega$, it holds $Z(t,X_t^{\blue{r,x}})\equiv 0$ and $Y=0$.  
As a consequence we have
\begin{equation*}
	-\hat{\mathbb{E}}_z[ Y]=0=\hat{\mathbb{E}}_z[-Y].
\end{equation*}
Thus, it follows by $(\ref{align:SupermartingaleExtended})$ and  $(\ref{eq:AdditivityExtended})$
\begin{align}
	\hat{\mathbb{E}}_z[M_u^{\epsilon,\blue{r,x}}]&=M_z^{\epsilon,\blue{r,x}} + \hat{\mathbb{E}}_z\bigg[  \int_z^u \textbf{1}_{[0,\tau_{\epsilon}^{\blue{r,x}}]}(t) e^{-\int_{\blue{r}}^t \blue{(X_s^{\blue{r,x}}+\epsilon)} ds}\big[u_x^{\epsilon}(t,X_t^{\blue{r,x}})g^{\epsilon}(t,X_t^{\blue{r,x}}) \nonumber \\
	&\quad + \frac{1}{2}u_{xx}^{\epsilon}(t,X_t^{\blue{r,x}})(h^{\epsilon}(t,X_t^{\blue{r,x}}))^2\big]d\langle B \rangle_t \nonumber \\
& \quad - \int_z^u \textbf{1}_{[0,\tau_{\epsilon}^{\blue{r,x}}]}(t) 2G(e^{-\int_{\blue{r}}^t \blue{(X_s^{\blue{r,x}}+\epsilon)} ds}[u_x^{\epsilon}(t,X_t^{\blue{r,x}})g^{\epsilon}(t,X_t^{\blue{r,x}}) \nonumber \\
&\quad + \frac{1}{2}u_{xx}^{\epsilon}(t,X_t^{\blue{r,x}})(h^{\epsilon}(t,X_t^{\blue{r,x}}))^2])dt \bigg]. \label{eq:MartingalePropertyExtended}
\end{align}
	By Proposition \violet{4.1.4 in \cite{peng_nonlinearExpectation_book}} we know that for any $\psi \in M_G^1(0,T)$ the process $\overline{M}=(\overline{M}_t)_{t \in [0,T]}$ given by 
\begin{equation}
	\overline{M}_t=\int_0^t \psi_s d \langle B \rangle_s-\int_0^t 2 G(\psi_s)ds, \quad t \in [0,T], \label{nonsymmetricMartingaleExtended}
\end{equation}
is a $G$-martingale. Thus, for $\blue{r} \leq z \leq u \leq T$ it follows
\begin{align}
	0&=\hat{\mathbb{E}}_z[\overline{M}_u-\overline{M}_z]=\hat{\mathbb{E}}_z\bigg[ \int_z^u \psi_s d \langle B \rangle_s-\int_z^u 2 G(\psi_s)ds  \bigg], \label{MartingalePropertyExtended}
\end{align}
where we used for the first equality Remark \violet{3.2.4 in \cite{peng_nonlinearExpectation_book}}. Note that, as $\tau_{\epsilon}^{\blue{r,x}}$ is quasi-continuous, it follows by Corollary \ref{StoppingLiu} that $\overline{M}^{\tau_{\epsilon}^{\blue{r,x}}}:=(\overline{M}_{t \wedge \tau_{\epsilon}^{\blue{r,x}}})_{t \in [0,T]}$ is also a $G$-martingale. By applying $(\ref{MartingalePropertyExtended})$ to $\psi_s=e^{-\int_{\blue{r}}^s (X_{\tilde{s}}^{\blue{r,x}}+\blue{\epsilon}) d\tilde{s}}\big[u_x^{\epsilon}(s,X_s^{\blue{r,x}})g^{\epsilon}(s,X_s^{\blue{r,x}})+ \frac{1}{2}u_{xx}^{\epsilon}(s,X_s^{\blue{r,x}})(h^{\epsilon}(s,X_s^{\blue{r,x}}))^2\big]$ and by using the same arguments as in $(\ref{G-SDETauEpsilon})$-$(\ref{G-SDETauRewrittenEpsilon})$ we get by $(\ref{eq:MartingalePropertyExtended})$ \black{that}
\begin{equation*}
	\hat{\mathbb{E}}_z[M_u^{\epsilon,\blue{r,x}}]=M_z^{\epsilon,\blue{r,x}} \quad \text{ for all } \blue{r} \leq z \leq u \leq T.
\end{equation*}
\end{proof}
\blue{
\begin{lemma}
\blue{Let $\varphi \in C([0,T] \times \mathbb R^+)$ satisfy Assumption \ref{asum:phi} and $(\varphi^{\epsilon})_{{\epsilon \in (0,1)}}$ be the {family} of functions constructed in Definition \ref{def:ApproximationPayoff}. Given $(r,x) \in [0,T] \times (\epsilon,\epsilon^{-1})$, let $X^{{r,x}}$ be the solution to the $G$-SDE \eqref{G-SDE_New}, and Assumption \ref{asumC1,2General} hold.}
	For $\epsilon \in (0,1) $ consider the viscosity solution $u^{\epsilon}$ of the PDE $(\ref{PDEExtendedDomain1})-(\ref{PDEExtendedDomain2})$. Then we have that
	\begin{align} \label{eq:RepresentationNew}
		u^{\epsilon}(r,x)=\hat{\mathbb{E}}_r\left[ u^{\epsilon}(T \wedge \tau_{\epsilon}^{r,x}, X^{r,x}_{T \wedge \tau_{\epsilon}^{r,x}})e^{-\int_r^{T \wedge \tau_{\epsilon}^{r,x}}(X_s^{r,x}+\epsilon)ds}\right].
	\end{align}
\end{lemma}
\begin{proof}
	This follows directly by Proposition \ref{prop:ExtendedMartingaleProperty} as $M^{\epsilon,r,x}_r=u^{\epsilon}(r,x)$. 
\end{proof}
\begin{remark}
	Note that the $G$-conditional expectation on the right-hand side in \eqref{eq:RepresentationNew} is deterministic, as the process $X^{r,x}$ starts in $x>0$ at time $r>0$. 
\end{remark}}

\red{The following lemma is needed, together with Proposition \ref{prop:estimate}, to get Theorem \ref{theorem:PointwiseConvergence}.
\begin{lemma}\label{lem:convergencefortauepsilon}
Given $(r,x) \in [0,T] \times (0,\infty)$, assume that the solution $X^{{r,x}}$ to the $G$-SDE \eqref{G-SDE_New} is {quasi-surely} strictly positive. Moreover, let Assumption \ref{asumC1,2General} hold. Then we have that
$$
\lim_{\epsilon \to 0}\mathbb{E}_{r} \left[\textbf{1}_{\lbrace \tau_{\epsilon}^{r,x}<T\rbrace}\right]=0.
$$
\end{lemma}
\begin{proof}
By Lemma \ref{ExampleLiu} we have that $X^{r,x}$ is quasi-continuous and therefore quasi-surely finite on $[r,T]$. In particular, there exists a quasi-null set $B^*$ such that for all $\omega \in \Omega \setminus B^*$ the function $t \mapsto X_t^{r,x}(\omega)$, $t \in [r,T]$ is continuous, and bounded from below from $0$ and from above on $[r,T]$, as $X^{r,x}$ is also quasi-surely positive. Fix $\omega \in \Omega \backslash B^*$. Then there exists $\eta_{\omega}^*>0$ such that for every $\tilde{\epsilon}<\eta^*_{\omega}$ it holds $\tau_{\tilde{\epsilon}}^{r,x}(\omega) = T$. Therefore, the result follows by the dominated convergence theorem stated in Theorem 3.2 in \cite{hu_zhou_2018}. 
\end{proof}
}
	
\red{\begin{prop} \label{prop:estimate}
	Let $\varphi \in C([0,T] \times \mathbb R^+)$ satisfy Assumption \ref{asum:phi}. Given $(r,x) \in [0,T] \times (0,\infty)$, assume that the solution $X^{{r,x}}$ to the $G$-SDE \eqref{G-SDE_New} is {quasi-surely} positive. Moreover, let Assumption \ref{asumC1,2General} hold. For every $\epsilon \in (0,1)$, let $u^{\epsilon}$ be the viscosity solution of the PDE $(\ref{PDEExtendedDomain1})-(\ref{PDEExtendedDomain2})$. Then 
	\begin{equation} \label{eq:estimate}
	 \left| \hat{\mathbb{E}}_{{r}} \left[ \varphi (T,X_T^{{r,x}}) e^{-\int_{{r}}^T X_s^{{r,x}} ds} \right]- {u^{\epsilon}(r,x)}   \right| \leq \epsilon + 4(M_0+\epsilon) \hat{\mathbb{E}}_{r} \left[\textbf{1}_{\lbrace \tau_{\epsilon}^{r,x}<T\rbrace}\right]	 +(M_0+ \epsilon) \left(1-e^{-(T-r) \epsilon}\right).
	\end{equation}
\end{prop}}
\begin{proof}
Given $(r,x) \in [0,T] \times (0,\infty)$, let $X^{r,x}$ be the solution to the $G$-SDE \eqref{G-SDE_New}. For any $\epsilon \in (0,1)$ such that $x \in (\epsilon,\epsilon^{-1})$ it holds
	\begin{align}
	& \left| \hat{\mathbb{E}}_{\blue{r}} \left[ \varphi (T,X_T^{\blue{r,x}}) e^{-\int_{\blue{r}}^T X_s^{\blue{r,x}} ds} \right]- \blue{u^{\epsilon}(r,x)}   \right| \nonumber \\
		&=\left| \hat{\mathbb{E}}_{\blue{r}} \left[ \varphi (T,X_T^{\blue{r,x}}) e^{-\int_{\blue{r}}^T X_s^{\blue{r,x}} ds} \right]-\hat{\mathbb{E}}_{\blue{r}} \left[ u^{\epsilon}(T \wedge \tau_{\epsilon}^{\blue{r,x}}, X^{\blue{r,x}}_{T \wedge \tau_{\epsilon}^{\blue{r,x}}})e^{-\int_{\blue{r}}^{T \wedge \tau_{\epsilon}^{\blue{r,x}}}(X^{\blue{r,x}}_s + \blue{\epsilon})ds}  \right] \right| \label{eq:fromnewrepresentation} \\
		&\blue{\leq }\left| \hat{\mathbb{E}}_{\blue{r}} \left[ \varphi (T,X_T^{\blue{r,x}}) e^{-\int_{\blue{r}}^T X_s^{\blue{r,x}} ds}- u^{\epsilon}(T \wedge \tau_{\epsilon}^{\blue{r,x}}, X^{\blue{r,x}}_{T \wedge \tau_{\epsilon}})e^{-\int_{\blue{r}}^{T \wedge \tau_{\epsilon}^{\blue{r,x}}}(X_s^{\blue{r,x}} + \blue{\epsilon})ds}  \right] \right| \label{eq:Symmetry} \\
		&\blue{\le  \hat{\mathbb{E}}_{\blue{r}} \left[ \left|\varphi (T,X_T^{\blue{r,x}}) e^{-\int_{\blue{r}}^T X_s^{\blue{r,x}} ds}- u^{\epsilon}(T \wedge \tau_{\epsilon}^{\blue{r,x}}, X^{\blue{r,x}}_{T \wedge \tau_{\epsilon}^{\blue{r,x}}})e^{-\int_{\blue{r}}^{T \wedge \tau_{\epsilon}^{\blue{r,x}}}(X^{\blue{r,x}}_s + \blue{\epsilon})ds}  \right|\right]  \notag} \\
		&\blue{\le  \hat{\mathbb{E}}_{\blue{r}} \left[ \left|\varphi (T,X_T^{\blue{r,x}}) e^{-\int_{\blue{r}}^T X_s^{\blue{r,x}} ds}- \varphi^{\epsilon}(T, X_{T}^{\blue{r,x}})e^{-\int_{\blue{r}}^{T}(X_s^{\blue{r,x}} + \blue{\epsilon})ds}  \right|\right]  \notag} \\
		&\blue{\quad +  \hat{\mathbb{E}}_{\blue{r}} \left[ \left|\varphi^{\epsilon}(T, X_{T}^{\blue{r,x}})e^{-\int_{\blue{r}}^{T}(X_s^{\blue{r,x}} + \blue{\epsilon})ds}- u^{\epsilon}(T \wedge \tau_{\epsilon}^{\blue{r,x}}, X^{\blue{r,x}}_{T \wedge \tau_{\epsilon}^{\blue{r,x}}})e^{-\int_{\blue{r}}^{T \wedge \tau_{\epsilon}^{\blue{r,x}}}(X_s^{\blue{r,x}} + \blue{\epsilon})ds}  \right|\right]. \label{eq:firstsplittingofcondexp}} 
	\end{align}
Note that in $(\ref{eq:Symmetry})$ {we use Proposition \violet{3.2.3 in \cite{peng_nonlinearExpectation_book}}}, \blue{whereas \eqref{eq:fromnewrepresentation} immediately follows from \eqref{eq:RepresentationNew}}. 
\red{For the second term in \eqref{eq:firstsplittingofcondexp} we get
	\begin{align}
		 &\hat{\mathbb{E}}_{r} \left[ \left|\varphi^{\epsilon}(T, X_{T}^{r,x})e^{-\int_{r}^{T}(X_s^{r,x} + \epsilon)ds}- u^{\epsilon}(T \wedge \tau_{\epsilon}^{r,x}, X^{r,x}_{T \wedge \tau_{\epsilon}^{r,x}})e^{-\int_{r}^{T \wedge \tau_{\epsilon}^{r,x}}(X_s^{r,x} + \epsilon)ds}  \right|\right]\nonumber \\
		 & = \hat{\mathbb{E}}_{r} \left[ \left|\varphi^{\epsilon}(T, X_{T}^{r,x})e^{-\int_{r}^{T}(X_s^{r,x} + \epsilon)ds}- u^{\epsilon}(T \wedge \tau_{\epsilon}^{r,x}, X^{r,x}_{T \wedge \tau_{\epsilon}^{r,x}})e^{-\int_{r}^{T \wedge \tau_{\epsilon}^{r,x}}(X_s^{r,x} + \epsilon)ds}  \right| \textbf{1}_{\lbrace \tau_{\epsilon}^{r,x}<T\rbrace}\right] \label{eq:firstsplittingofcondexp0}\\
		 &\leq 2(M_0+\epsilon) \hat{\mathbb{E}}_{r} \left[\textbf{1}_{\lbrace \tau_{\epsilon}^{r,x}<T\rbrace}\right], 
		 \end{align}
		 where we use in \eqref{eq:firstsplittingofcondexp0} that for $ \tau_{\epsilon}^{r,x}\geq T$ we have $u^{\epsilon}(T \wedge \tau_{\epsilon}^{r,x}, X^{r,x}_{T \wedge \tau_{\epsilon}^{r,x}})=\varphi^{\epsilon}(T, X_{T}^{r,x})$.
		For the first term in \eqref{eq:firstsplittingofcondexp} we have
	\begin{align}
		&\hat{\mathbb{E}}_{r} \left[ \left|\varphi (T,X_T^{{r,x}}) e^{-\int_{r}^T X_s^{r,x} ds}- \varphi^{\epsilon}(T, X_{T}^{r,x})e^{-\int_{r}^{T}(X_s^{r,x} + {\epsilon})ds}  \right|\right] \nonumber  \\
		&\leq \hat{\mathbb{E}}_{r} \left[ \left|\varphi (T,X_T^{{r,x}}) - \varphi^{\epsilon}(T, X_{T}^{r,x})e^{-(T-r) \epsilon)}  \right|\right] \nonumber \\
		&= \hat{\mathbb{E}}_{r} \left[ \left|\varphi (T,X_T^{{r,x}}) - \varphi^{\epsilon}(T, X_{T}^{r,x}) + \varphi^{\epsilon}(T, X_{T}^{r,x})-\varphi^{\epsilon}(T, X_{T}^{r,x})e^{-(T-r) \epsilon)}  \right|\right] \nonumber \\
		& \leq \hat{\mathbb{E}}_{r} \left[ \left|\varphi (T,X_T^{{r,x}}) - \varphi^{\epsilon}(T, X_{T}^{r,x})\right \vert \right] +  \left(1-e^{-(T-r) \epsilon}\right) \hat{\mathbb{E}}_r \left[ \left \vert \varphi^{\epsilon}(T, X_{T}^{r,x}) \right \vert \right]. \label{eq:ErrorEstimateCalculations1}
	\end{align}
	Moreover, we get
	\begin{align}
		& \hat{\mathbb{E}}_{r} \left[ \left|\varphi (T,X_T^{{r,x}}) - \varphi^{\epsilon}(T, X_{T}^{r,x})\right \vert \right] \nonumber \\
		&=\hat{\mathbb{E}}_{r} \left[ \left|\varphi (T,X_T^{{r,x}}) - \tilde{\varphi}^{\epsilon}(T, X_{T}^{r,x})\right \vert \textbf{1}_{\lbrace X_T^{r,x} \leq 2 \epsilon^{-1} \rbrace} + \left|\varphi (T,X_T^{{r,x}}) - \tilde{\varphi}^{\epsilon}(T, 2 \epsilon^{-1})\right \vert \textbf{1}_{\lbrace X_T^{r,x} >2 \epsilon^{-1} \rbrace} \right] \nonumber\\ 
		& \leq  \epsilon \hat{\mathbb{E}}_{r} \left[ \textbf{1}_{\lbrace X_T^{r,x} \leq 2 \epsilon^{-1} \rbrace} \right] + 2(M_0+\epsilon) \hat{\mathbb{E}}_{r} \left[  \textbf{1}_{\lbrace X_T^{r,x} >2 \epsilon^{-1} \rbrace} \right] \label{eq:ErrorEstimateCalculations2}.
	\end{align}
	Thus, putting together \eqref{eq:firstsplittingofcondexp}- \eqref{eq:ErrorEstimateCalculations2} it follows
	\begin{align}
		& \left| \hat{\mathbb{E}}_{{r}} \left[ \varphi (T,X_T^{{r,x}}) e^{-\int_{{r}}^T X_s^{{r,x}} ds} \right]- {u^{\epsilon}(r,x)}   \right| \nonumber \\
		&\leq \epsilon \hat{\mathbb{E}}_{r} \left[ \textbf{1}_{\lbrace X_T^{r,x} \leq 2 \epsilon^{-1} \rbrace} \right] + 2(M_0+\epsilon) \hat{\mathbb{E}}_{r} \left[  \textbf{1}_{\lbrace X_T^{r,x} >2 \epsilon^{-1} \rbrace} \right]  +(M_0+ \epsilon) \left(1-e^{-(T-r) \epsilon}\right) \nonumber\\
		& \quad + 2(M_0+\epsilon) \hat{\mathbb{E}}_{r} \left[\textbf{1}_{\lbrace \tau_{\epsilon}^{r,x}<T\rbrace}\right]	\nonumber \\
		&\leq \epsilon + 4(M_0+\epsilon) \hat{\mathbb{E}}_{r} \left[\textbf{1}_{\lbrace \tau_{\epsilon}^{r,x}<T\rbrace}\right]	 +(M_0+ \epsilon) \left(1-e^{-(T-r) \epsilon}\right) \label{eq:ErrorEstimateCalculations3},
		\end{align}
		where \eqref{eq:ErrorEstimateCalculations3} follows because the process $X^{r,x}$ is continuous and thus $\lbrace X_T^{r,x}>2 \epsilon^{-1} \rbrace \subset \lbrace \tau_{\epsilon}^{r,x} <T \rbrace$.
}
\end{proof}
\red{The next theorem immediately follows from Lemma \ref{lem:convergencefortauepsilon} and Proposition \ref{prop:estimate}.}
\begin{theorem} \label{theorem:PointwiseConvergence}
	Let $\varphi \in C([0,T] \times \mathbb R^+)$ satisfy Assumption \ref{asum:phi}. Given $(r,x) \in [0,T] \times (0,\infty)$, assume that the solution $X^{{r,x}}$ to the $G$-SDE \eqref{G-SDE_New} is {quasi-surely} positive. Moreover, let Assumption \ref{asumC1,2General} hold. For every $\epsilon \in (0,1)$, let $u^{\epsilon}$ be the viscosity solution of the PDE $(\ref{PDEExtendedDomain1})-(\ref{PDEExtendedDomain2})$. Then 
	\begin{equation} \label{eq:thmConvergenceFinal}
	\lim_{\epsilon \to 0} 	\left \vert u^{\epsilon}(\blue{r,x})- \hat{\mathbb{E}}_{\blue{r}}[\varphi(T,X_T^{\blue{r,x}})e^{-\int_{\blue{r}}^{T}X_s^{\blue{r,x}}ds}] \right \vert \blue{=}0.
	\end{equation}
\end{theorem}

%\subsection{Feynman-Kac Formula}
\blue{{We now establish a} Feynman-Kac formula for the $G$-conditional expectations 
$$
\hat{\mathbb{E}}_{\blue{r}}\left[\varphi(T,X_T^{\blue{r,x}})e^{-\int_{\blue{r}}^{T}X_s^{\blue{r,x}}ds}\right], \quad r \in [0,T].
$$
{by providing a stability result on the limits of the modified PDE \eqref{PDEExtendedDomain1}-\eqref{PDEExtendedDomain2} and of Theorem \ref{theorem:PointwiseConvergence}. }}
\blue{To prove the main result of this section, we make use of the following definition of uniform convergence.
\begin{defi} \label{def:LocallyUniformly}
	For $\epsilon \in (0,1)$, {consider functions} $F^{\epsilon}:[0,T] \times \mathbb{R}^+ \times \mathbb{R}^3 \to \mathbb{R}$. {We say that $(F^{\epsilon})_{\epsilon >0}$} converges locally uniformly for $\epsilon \to 0$ {to a function $F$} if for every $z:=(t,x,u,u_x,u_{xx}) \in [0,T] \times \mathbb{R}^+ \times \mathbb{R}^3$ there exists a neighborhood $N_z$ of the point $z$ such that $F^{\epsilon}$ converges to $F$ uniformly on $N_z$ for $\epsilon \to 0$, i.e., $$\sup_{y \in N_{\violet{z}}} \vert F^{\epsilon}(y)-F(y) \vert \to 0$$ for $\epsilon \to 0$.
	\end{defi}
In the sequel we write $F^{\epsilon} \to F$ to indicate convergence of $F^{\epsilon}$ to $F$.
\begin{remark} \label{remark:LocallyUniformly}
As $\mathbb{R}^5$ is a locally compact space, $F^{\epsilon}$ converges to $F$ locally uniformly for $\epsilon \to 0$ is equivalent to $F^{\epsilon} \to F$ compactly, i.e., for every compact set $K \subset [0,T] \times \mathbb{R}^+ \times \mathbb{R}^3$, $F^{\epsilon} \vert_K \to F \vert_K$ uniformly for $\epsilon \to 0$. 
\end{remark}	
	We start with two lemmas.}
\blue{\begin{lemma} \label{lemma:ConvergencePDEEquation}
\begin{enumerate}
	\item For any $\epsilon \in (0,1)$, define $F^{\epsilon}:[0,T] \times \mathbb{R}^+ \times \mathbb{R}^3  \to \mathbb{R}$ as
	\begin{align*}
		F^{\epsilon}(t,x,u,u_x,u_{xx}):= 2 G\left( u_x g^{\epsilon}(t,x)+ \frac{1}{2} u_{xx}\left(h^{\epsilon}(t,x)\right)^2\right)	 + f^{\epsilon}(t,x)u_x - (\vartheta^{\epsilon}(t,x)+\epsilon)u,
	\end{align*}
	and define $F:[0,T] \times\mathbb{R}^+ \times \mathbb{R}^3 \to \mathbb{R}$ by
	\begin{align*}
		F(t,x,u,u_x,u_{xx}):=2G \left(u_x g(t,x) + \frac{1}{2} u_{xx} \left(h(t,x)\right)^2\right) + f(t,x)u_x - x u.
	\end{align*}
	Then $F^{\epsilon} \to F$ locally uniformly for $\epsilon \to 0$. 
	\item $\varphi^{\epsilon} \to \varphi$ locally uniformly for $\epsilon \to 0$.
\end{enumerate}
\end{lemma}
\begin{proof}
1. Let $[t_1,t_2] \times [x_1,x_2] \times [a_1,a_2] \times [b_1,b_2] \times [c_1,c_2] \subset  [0,T] \times \mathbb{R}^+ \times \mathbb{R}^3$ {for $t_1<t_2,x_1<x_2,a_1<a_2,b_1<b_2,c_1<c_2$. Choose $\hat{\epsilon}$ small enough in $(0,1)$ such that $[x_1,x_2] \subseteq (\hat{\epsilon}, \hat{\epsilon}^{-1})$.} Let $\delta >0$.
	By the definition of  $g^{\epsilon}, f^{\epsilon},h^{\epsilon}$ in equations \eqref{eq:DefinititonCutoff1}-\eqref{eq:DefinititonCutoff3}, we have that $\phi^{\epsilon}(t,x)=\phi(t,x)$ for $(t,x) \in [0,T] \times [x_1,x_2]$ and $\epsilon < \hat{\epsilon}$, for $\phi=g, f, h$. Moreover, $\vartheta^{\epsilon}(x)=x$ for $x \in [x_1,x_2]$ and $\epsilon < \hat{\epsilon}$. Therefore, for $\epsilon < \min(\hat{\epsilon},\delta/\vert a_2\vert)$ 	
	we have
	$$
	|F^{\epsilon}(t,x,u,u_x,u_{xx})-F(t,x,u,u_x,u_{xx})| = \epsilon |u| <  \epsilon |a_2|<\delta
	$$
	for any $(t,x,u,u_x,u_{xx}) \in [t_1,t_2] \times [x_1,x_2] \times [a_1,a_2] \times [b_1,b_2] \times [c_1,c_2].$ The result then follows by Remark \ref{remark:LocallyUniformly}.  \\
2. Let $[t_1,t_2] \times [x_1,x_2] \subset [0,T] \times (0, \infty)$ {for $t_1 < t_2, x_1<x_2$}. Choose $\hat{\epsilon}$ small enough in $(0,1)$ such that $[x_1,x_2] \subseteq (0,2 \hat{\epsilon}^{-1})$. Let $\delta>0$. By Definition \ref{def:ApproximationPayoff} it follows that for any $\epsilon < \min(\hat{\epsilon},\delta)$
\begin{equation*}
	\vert \varphi^{\epsilon}(t,x) - \varphi(t,x) \vert < \delta.
\end{equation*}
for $(t,x) \in [t_1,t_2]\times [x_1,x_2]$.
\end{proof}
\begin{lemma}\label{lemma:x1andx2}
For fixed  $r \in [0, T]$, $\epsilon \in (0,1)$ and $x \in (\epsilon, \epsilon^{-1})$, define the stopping times 
\begin{align} \label{eq:AdditionalStoppingTimes}
	\tau^{1,r,x}_{\epsilon}:=\inf \lbrace t \in [r,T ]: X_t^{r,x} \geq \epsilon^{-1} \rbrace \wedge T, \quad \tau^{2,r,x}_{\epsilon}:=\inf \lbrace t \in [r,T]: X_t^{r,x} \leq \epsilon \rbrace \wedge T.
\end{align}
Then for any $r \in [0, T]$ and $ \epsilon < y_1 < y_2 <  \epsilon^{-1}$, we have
\begin{align} 
	\hat{\mathbb{E}}_r \left[ \textbf{1}_{\lbrace\tau_{\epsilon}^{1, r,y_1}<T\rbrace}\right] \leq \hat{\mathbb{E}}_r \left[ \textbf{1}_{\lbrace\tau_{\epsilon}^{1,r,y_2}<T\rbrace}\right], \label{eq:ConvergenceStability4}
\end{align}
and
\begin{align} 
	\hat{\mathbb{E}}_r \left[ \textbf{1}_{\lbrace\tau_{\epsilon}^{2, r,y_2}<T\rbrace}\right] \leq \hat{\mathbb{E}}_r \left[ \textbf{1}_{\lbrace\tau_{\epsilon}^{2,r,y_1}<T\rbrace}\right]. \label{eq:ConvergenceStability4bis}
\end{align}
\end{lemma}
\begin{proof}
First we note that $\tau_{\epsilon}^{1,r,x}$ and $\tau_{\epsilon}^{2,r,x}$ are quasi-continuos by similar arguments as in Lemma \ref{QuasiContinuity}. For $ \epsilon < y_1 < y_2 <  \epsilon^{-1}$ define the exit time 
\begin{equation*}
	\tilde{\tau}^{r,y_1,y_2}:= \inf \lbrace t \in [r,T]: X_t^{r,y_2}-X_t^{r,y_1} \leq 0 \rbrace \wedge T.
\end{equation*}
Again by similar arguments as in Lemma \ref{QuasiContinuity} it follows that $\tilde{\tau}^{x,y_1,y_2}$ is quasi-continuous, as the process $(X_t^{r,y_2}-X_t^{r,y_1})_{t \in [r,T]}$ is quasi-continuous. We first consider the case $\tilde{\tau}^{r,y_1,y_2}<T$. 
Since $X^{r,y_2}$ and $X^{r,y_1}$ are both quasi-continuous processes, it follows that $X^{r,y_2}_{\tilde{\tau}^{x,y_1,y_2}}=X^{r,y_1}_{\tilde{\tau}^{x,y_1,y_2}}$ quasi-surely. Thus, for each $t > \tilde{\tau}^{x,y_1,y_2}$ we have
\begin{align*}
	X_t^{r,y_1}&=X_{\tilde{\tau}^{x,y_1, y_2}}^{r,y_1}+ \int_{\tilde{\tau}^{x,y_1,y_2}}^t f(s,X_s^{r,y_1})ds + \int_{\tilde{\tau}^{x,y_1,y_2}}^t g(s,X_s^{r,y_1})d\langle B \rangle_s + \int_{\tilde{\tau}^{x,y_1,y_2}}^t h(s,X_s^{r,y_1})dB_s, \\
	X_t^{r,y_2}&=X_{\tilde{\tau}^{x,y_1, y_2}}^{r,y_2}+ \int_{\tilde{\tau}^{x,y_1,y_2}}^t f(s,X_s^{r,y_2})ds + \int_{\tilde{\tau}^{x,y_1,y_2}}^t g(s,X_s^{r,y_2})d\langle B \rangle_s + \int_{\tilde{\tau}^{x,y_1,y_2}}^t h(s,X_s^{r,y_2})dB_s, \\
\end{align*}
which implies that $X_t^{r,y_1}=X_t^{r,\blue{y_2}}$ for $t >\tilde{\tau}^{r,y_1,y_2}$. Therefore, we get \eqref{eq:ConvergenceStability4} by definition of $\tau^{1,r,y_1}$ and $\tau^{1,r,y_2}$. If $\tilde{\tau}^{r,y_1,y_2}=T$ it holds that $X_t^{y_2}> X_t^{y_1}$ for all $t \in [r,T]$ and thus \eqref{eq:ConvergenceStability4} holds. 
The proof of \eqref{eq:ConvergenceStability4bis} is fully analogous. 
\end{proof}
}

\blue{In order to prove the Feynman-Kac formula in Theorem \ref{theorem:StabilityBounded}, we need to introduce a further hypothesis.
\begin{asum} \label{asumPayoffNoTimeDependence}
	We assume that
	\begin{equation*}
		\sup_{r \in [r_1,r_2]} \hat{\mathbb{E}}_r \left[ \textbf{1}_{\lbrace \tau_{\epsilon}^{1,r,x_2} <T \rbrace}\right] \to 0 
	\end{equation*}
	for any $0<x_2, 0<r_1<r_2$, and
	\begin{equation*}
		\sup_{r \in [r_1,r_2]} \hat{\mathbb{E}}_r \left[ \textbf{1}_{\lbrace \tau_{\epsilon}^{2,r,x_1}<T \rbrace}\right] \to 0 
	\end{equation*}
	for any $0<x_1, 0<r_1<r_2$, where $\tau_{\epsilon}^{1,r,x_2},\tau_{\epsilon}^{2,r,x_1}$ are defined in \eqref{eq:AdditionalStoppingTimes}.
\end{asum}
\begin{remark}
Note that Assumption \ref{asumPayoffNoTimeDependence} is very weak, and is for example satisfied if for any $x>0$ and any $0\le r_1<r_2\le T$ there exist $\bar r \in [r_1,r_2]$ and $\eta>0$ such that 
	\begin{equation} \label{eq:AssumptionStoppingTimes}
		\hat{\mathbb{E}}_r \left[ \textbf{1}_{\lbrace \tau_{\epsilon}^{r,x}<T\rbrace} \right] \le  \hat{\mathbb{E}}_{\bar r} \left[ \textbf{1}_{\lbrace \tau_{\epsilon}^{\bar{r},x}<T\rbrace} \right]
	\end{equation}
	for any $r \in [r_1,r_2]$ and any $\epsilon < \eta$. For most common models the latter condition holds (in most cases taking $\bar r = r_1$). Moreover, Assumption \ref{asumPayoffNoTimeDependence} holds for every solution of a $G$-SDE with time independent coefficients. 
\end{remark}}

\blue{
\begin{theorem} \label{theorem:StabilityBounded}
\blue{Let $\varphi \in C([0,T] \times \mathbb{R}^+)$ satisfy Assumption \ref{asum:phi}. Given $(r,x) \in [0,T] \times (0,\infty)$, } assume that the process $X^{\blue{r,x}}=(X_t^{\blue{r,x}})_{t \in [\blue{r},T]}$ in $(\ref{G-SDE_New})$ is quasi-surely strictly positive. Let Assumptions \ref{asumC1,2General} and \ref{asumPayoffNoTimeDependence} be satisfied.
Define the function $u:[0,T] \times (0,\infty) \to \mathbb{R}$ by 
\begin{equation} \label{eq:DefinitionLimit}
	u(r,x):=\hat{\mathbb{E}}_r \left[ \varphi (T,X_T^{r,x})e^{-\int_r^{T} X_s^{r,x}ds}\right].
\end{equation}
Then $u$ is a viscosity solution of the PDE 
\begin{align}
	u_t + 2G \left(u_x g(t,x) + \frac{1}{2} u_{xx} \left(h(t,x)\right)^2\right)	 + f(t,x)u_x - xu =0, \quad &\text{ for } (t,x) \in Q  \label{eq:PDENoTimeDependenceCovergence1}\\
		u=\varphi \quad &\text{ for } (t,x) \in \partial'Q,\label{eq:PDENoTimeDependenceConvergence2}
\end{align}
for $Q=(0,T) \times (0,\infty)$.
\end{theorem}}

\begin{proof}
\blue{In order to get the result, we prove that {$(u^{\epsilon})_{\epsilon \in (0,1)}$} converges locally uniformly to $u^{\epsilon}$ for $\epsilon \to 0$. By Remark \ref{remark:LocallyUniformly}, this boils down to prove that for every fixed $0\le r_1<r_2 \le T$ and $0<x_1<x_2$, for all $\delta >0$ there exists a constant $\eta>0$ such that for every $\epsilon < \eta$ it holds
\begin{equation*}
	\sup_{(r,x) \in [r_1,r_2] \times [x_1,x_2]} \vert u^{\epsilon}(r,x) -u(r,x)  \vert < \delta.
\end{equation*}
We have
\begin{align}
	& \vert u^{\epsilon}(r,x) - u(r,x) \vert \nonumber \\ 
	 &\quad = \bigg \vert \hat{\mathbb{E}}_r \left[ u^{\epsilon} (T\wedge \tau_{\epsilon}^{r,x},X_{T\wedge \tau_{\epsilon}^{r,x}}^{r,x})e^{-\int_r^{T \wedge \tau_{\epsilon}^{r,x}} (X_s^{r,x}+\epsilon)ds}\right] - \hat{\mathbb{E}}_r   \left[ \varphi (T,X_T^{r,x})e^{-\int_r^{T} X_s^{r,x}ds}\right] \bigg \vert \nonumber \\
	& \quad \leq   \hat{\mathbb{E}}_r \left[ \big \vert u^{\epsilon} (T\wedge \tau_{\epsilon}^{r,x},X_{T\wedge \tau_{\epsilon}^{r,x}}^{r,x})e^{-\int_r^{T \wedge \tau_{\epsilon}^{r,x}} (X_s^{r,x}+\epsilon)ds}- \varphi (T,X_T^{r,x})e^{-\int_r^{T} X_s^{r,x}ds} \big \vert \right] \nonumber \\
	& \quad \leq   \hat{\mathbb{E}}_r \left[ \big \vert u^{\epsilon} (T\wedge \tau_{\epsilon}^{r,x},X_{T\wedge \tau_{\epsilon}^{r,x}}^{r,x})e^{-\int_r^{T \wedge \tau_{\epsilon}^{r,x}} (X_s^{r,x}+\epsilon)ds}- \varphi (T,X_T^{r,x})e^{-\int_r^{T} X_s^{r,x}ds} \big \vert \textbf{1}_{\lbrace \tau_{\epsilon}^{r,x}<T\rbrace}\right]\nonumber \\
	& \quad \quad +  \hat{\mathbb{E}}_r \left[ \big \vert u^{\epsilon} (T\wedge \tau_{\epsilon}^{r,x},X_{T\wedge \tau_{\epsilon}^{r,x}}^{r,x})e^{-\int_r^{T \wedge \tau_{\epsilon}^{r,x}} (X_s^{r,x}+\epsilon)ds} - \varphi (T,X_T^{r,x})e^{-\int_r^{T} X_s^{r,x}ds} \big \vert \textbf{1}_{\lbrace \tau_{\epsilon}^{r,x}\geq T\rbrace}\right].\label{eq:ConvergenceStability1}
\end{align}}

\blue{
Let $\delta \in (0,1)$. Then for any $\epsilon < \bar{\epsilon}$, for an appropriate choice of $\bar{\epsilon}$, it holds
\begin{align}
	&\hat{\mathbb{E}}_r \left[ \big \vert u^{\epsilon} (T\wedge \tau_{\epsilon}^{r,x},X_{T\wedge \tau_{\epsilon}^{r,x}}^{r,x})e^{-\int_r^{T \wedge \tau_{\epsilon}^{r,x}}(X_s^{r,x}+\epsilon)ds} - \varphi (T,X_T^{r,x})e^{-\int_r^{T} X_s^{r,x}ds} \big \vert \textbf{1}_{\lbrace \tau_{\epsilon}^{r,x}\geq T\rbrace}\right] \nonumber \\
	&=\hat{\mathbb{E}}_r \left[ \big \vert \tilde \varphi^{\epsilon} (T,X_T^{r,x})e^{-\int_r^T (X_s^{r,x}+\epsilon)ds} - \varphi (T,X_T^{r,x})e^{-\int_r^{T} X_s^{r,x}ds} \big \vert \textbf{1}_{\lbrace \tau_{\epsilon}^{r,x}\geq T\rbrace}\right] \label{eq:phitilda}  \\
	&\le\hat{\mathbb{E}}_r \left[ \big \vert \tilde \varphi^{\epsilon} (T,X_T^{r,x}) \big \vert  \big \vert e^{-\int_r^T (X_s^{r,x}+\epsilon)ds} - e^{-\int_r^{T} X_s^{r,x}ds} \big \vert \right] \notag \\
	&\quad +\hat{\mathbb{E}}_r \left[e^{-\int_r^{T} X_s^{r,x}ds} \big \vert \tilde \varphi^{\epsilon} (T,X_T^{r,x}) - \varphi (T,X_T^{r,x})\big \vert \right] \label{eq:termboundedbyepsilon}  \\
 & \le (M_0+\epsilon) \left(1-e^{-(T-r)\epsilon} \right) + \epsilon < \frac{\delta}{2}.\label{eq:ConvergenceStability2}
	\end{align}
	Note that both \eqref{eq:phitilda} and the fact that the term \eqref{eq:termboundedbyepsilon} is bounded by $\epsilon$ follow by Definition \ref{def:ApproximationPayoff}.}

	\blue{ 
Moreover, it holds
\begin{align}
	&\hat{\mathbb{E}}_r \left[ \big \vert u^{\epsilon} (T\wedge \tau_{\epsilon}^{r,x},X_{T\wedge \tau_{\epsilon}^{r,x}}^{r,x})e^{-\int_r^{T \wedge \tau_{\epsilon}^{r,x}} (X_s^{r,x}+\epsilon)ds}- \varphi (T,X_T^{r,x})e^{-\int_r^{T} X_s^{r,x}ds} \big \vert \textbf{1}_{\lbrace \tau_{\epsilon}^{r,x}<T\rbrace}\right] \nonumber \\
	&{\leq \hat{\mathbb{E}}_r \left[ \big \vert u^{\epsilon} (T\wedge \tau_{\epsilon}^{r,x},X_{T\wedge \tau_{\epsilon}^{r,x}}^{r,x}) \big \vert \textbf{1}_{\lbrace \tau_{\epsilon}^{r,x}<T\rbrace} \right] + \hat{\mathbb{E}}_r \left[ \big \vert \varphi(T,X_T^{r,x}) \big \vert \textbf{1}_{\lbrace \tau_{\epsilon}^{r,x}<T\rbrace} \right] }\\
	& \leq 2{(M_0+1)} \hat{\mathbb{E}}_r \left[ \textbf{1}_{\lbrace \tau_{\epsilon}^{r,x}<T\rbrace} \right] \nonumber \\
	& \leq 2{(M_0+1)} \left(\hat{\mathbb{E}}_r \left[ \textbf{1}_{\lbrace \tau_{\epsilon}^{1,r,x}<T\rbrace} \right] + \hat{\mathbb{E}}_r \left[ \textbf{1}_{\lbrace \tau_{\epsilon}^{2,r,x}<T\rbrace} \right] \right)\nonumber \\
	& \leq 2{(M_0+1)}\sup_{r \in [r_1,r_2]}\hat{\mathbb{E}}_r \left[ \textbf{1}_{\lbrace \tau_{\epsilon}^{1,r,x}<T\rbrace} \right] + 2{(M_0+1)} \sup_{r \in [r_1,r_2]}\hat{\mathbb{E}}_r \left[ \textbf{1}_{\lbrace \tau_{\epsilon}^{2,r,x}<T\rbrace} \right] \nonumber \\
	& \leq 2{(M_0+1)}\sup_{r \in [r_1,r_2]}\hat{\mathbb{E}}_r \left[ \textbf{1}_{\lbrace \tau_{\epsilon}^{1,r,x_2}<T\rbrace} \right] + 2{(M_0+1)} \sup_{r \in [r_1,r_2]}\hat{\mathbb{E}}_r \left[ \textbf{1}_{\lbrace \tau_{\epsilon}^{2,r,x_1}<T\rbrace} \right]. \label{eq:ConvergenceStability3}
	\end{align}
Note now that \eqref{eq:ConvergenceStability3} converges to zero by Assumption \ref{asumPayoffNoTimeDependence} for any $0<x_1<x_2$.
For this reason, there exists $\hat{\epsilon}>0$ such that 
\begin{equation} \label{eq:ConvergenceStability7}
	  2{(M_0+1)}\sup_{r \in [r_1,r_2]}\hat{\mathbb{E}}_r \left[ \textbf{1}_{\lbrace \tau_{\epsilon}^{1,r,x_2}<T\rbrace} \right] + 2{(M_0+1)} \sup_{r \in [r_1,r_2]}\hat{\mathbb{E}}_r \left[ \textbf{1}_{\lbrace \tau_{\epsilon}^{2,r,x_1}<T\rbrace} \right]< \frac{\delta}{2}
\end{equation}
 for any $\epsilon < \hat{\epsilon}$. Putting together equations  \eqref{eq:ConvergenceStability1}, \eqref{eq:ConvergenceStability2} and \eqref{eq:ConvergenceStability7}, 
  we get 
  \begin{align*}
	&\bigg \vert u^{\epsilon}(r,x) - \hat{\mathbb{E}}_r   \left[ \varphi (T,X_T^{r,x})e^{-\int_r^{T} X_s^{r,x}ds}\right] \bigg \vert  < \frac{\delta}{2}+ \frac{\delta}{2} = \delta
	\end{align*}
	for any $\epsilon < \min(\bar{\epsilon}, \eta, \hat{\epsilon})$.}
	
\blue{By using Proposition 4.3 in \cite{users_guide_viscosity} and similar arguments as in the proof of Theorem 2, Chapter 3 in \cite{katzourakis_2015} it follows that $u$ is a viscosity solution of the PDE \eqref{eq:PDENoTimeDependenceCovergence1}-\eqref{eq:PDENoTimeDependenceConvergence2}, as $u^{\epsilon} \to u$ locally uniformly, $F^{\epsilon} \to F$ locally uniformly and $\varphi^{\epsilon} \to \varphi$ locally uniformly by Lemma \ref{lemma:ConvergencePDEEquation}.}
\end{proof}

\begin{remark}
	 In Theorem \ref{theorem:PointwiseConvergence} we prove a weaker convergence result under weaker conditions than in Theorem \ref{theorem:StabilityBounded}. In particular, Assumption \ref{asumPayoffNoTimeDependence} is not necessary for Theorem \ref{theorem:PointwiseConvergence}, but crucial in the proof of \ref{theorem:StabilityBounded}.
\end{remark}

Note that the presence of the linear term $xu$ in \eqref{eq:PDENoTimeDependenceCovergence1} does not allow to apply the comparison principle in Appendix C in \violet{\cite{peng_nonlinearExpectation_book}}, which guarantees uniqueness of a viscosity solution. \violet{However, Proposition \ref{prop:estimate} provides an estimate for the error in approximating the $G$-expectation with the unique viscosity solution $u^{\epsilon}$ to the PDE $(\ref{PDEExtendedDomain1})-(\ref{PDEExtendedDomain2})$ for a given $\epsilon>0$.}

\subsection{Feynman-Kac Formula for unbounded payoffs} \label{sectionUnbounded}
\blue{In the following we extend the results of Section \ref{subsec:boundedpayoff} to the case of a continuous payoff function $\varphi$ with polynomial growth.}
\blue{\begin{asum}\label{asum:PhiUnbounded}
	Let $L \geq 2$ be a constant. The function $\varphi \in C([0,T] \times \mathbb R^+)$ satisfies a polynomial growth condition of order less or equal $L$ {in} the second variable {uniformly in $t$}, i.e., there exists a constant $\tilde C(T)$ only depending on the final time $T$ such that it holds
\begin{equation*}
	|\varphi(t,x)| \leq \tilde C(T) (1+x^L)
\end{equation*}
for any $(t,x) \in [0,T] \times \mathbb{R}^+$.
\end{asum}}
\blue{Assumption \ref{asum:PhiUnbounded} is now {general} enough to {be satisfied by all the common derivatives on the market.} \blue{We now prove a Feynman-Kac formula for payoffs $\varphi$ satisfying the less restrictive Assumption \ref{asum:PhiUnbounded} by applying Theorem \ref{theorem:StabilityBounded} for the} (bounded) payoff functions $\varphi^{\epsilon}$ approximating $\varphi$ as in Definition \ref{def:ApproximationPayoff}, and then \blue{passing} to the limit for $\epsilon \to 0$. The next lemma states that the functions $\varphi^{\epsilon}$ are indeed bounded, and also provides an explicit form for such a bound.}
\blue{
\begin{lemma} \label{lemma:ConvergencePhiUnbounded}
\blue{Let $\varphi$ satisfy Assumption \ref{asum:PhiUnbounded} and $(\varphi^{\epsilon})_{\epsilon \in (0,1)}$ be the family of approximating functions for $\varphi$ constructed in Definition \ref{def:ApproximationPayoff}.}
For any $\epsilon \in (0,1)$ there exists a constant $M^{\epsilon}>0$ such that
\begin{equation} \label{eq:UnboundedPayoffEstimate}
	\vert \varphi^{\epsilon}(t,x)\vert \leq M^{\epsilon}
\end{equation}
 for all $(t,x) \in [0,T] \times \mathbb{R}^+$.  In particular,
 \begin{equation} \label{eq:UnboundedPayoffEstimate1}
 M^{\epsilon}:=K(T,L)\epsilon^{-L}>0,
\end{equation}
 where $K(T,L)>0$ is a constant depending only on $T$ and $L$.
 \end{lemma}
\begin{proof} 
The existence of a constant bound \blue{$M^{\epsilon}$} for $\varphi^{\epsilon}$ for $\epsilon\in (0,1)$ follows by \blue{Definition \ref{def:ApproximationPayoff}.} We now fix $\epsilon\in (0,1)$ and derive the explicit form of $M^{\epsilon}$. {Fix $t \in [0,T]$.  {By \eqref{eq:DefinititonCutOffPayoffFunction} it is enough to prove \eqref{eq:UnboundedPayoffEstimate1} only for $x \leq 2{\epsilon}^{-1}$.} By the definition of $\varphi^{\epsilon}$ {and Assumption \ref{asum:PhiUnbounded}} we have 
\begin{align}
		\vert \varphi^{\epsilon}(t,x)\vert &= \vert \tilde{\varphi}^{\epsilon}(t,x)\vert \nonumber \\
		& \leq \vert {\varphi}(t,x) \vert + \epsilon 	\notag %\label{eq:ExplanationBounds0}  
		\\ & \leq \tilde{C}(T)(1+x^L) + 1	
		\label{eq:ExplanationBounds1} \\ 
		& \leq \tilde{C}(T)(1+2^L \epsilon^{-L}) + 1 
		\label{eq:ExplanationBounds2} \\ 
		& =K(T,L) \epsilon^{-L} \nonumber.
 \end{align}
 \blue{In} \eqref{eq:ExplanationBounds2} \blue{we use that} $x \leq 2\epsilon^{-1}$.} 
\end{proof}
The specific form of the bound given in \eqref{eq:UnboundedPayoffEstimate1} will be crucial to prove Theorem \ref{thm:firstresultforunboundedpayoff} and Theorem \ref{thm:secondresultforunboundedpayoff}. The latter constitutes the main result of the paper.}
\blue{
We start by \blue{extending} the result of Lemma \ref{lem:convergenceapproxvarphi} \blue{to} the case when $\varphi$ is not bounded.
\begin{lemma}\label{lem:convergenceapproxvarphinotbounded}
Let $\varphi$ satisfy Assumption \ref{asum:PhiUnbounded} and $(\varphi^{\epsilon})_{\blue{\epsilon \in (0,1)}}$ be the \blue{family} of functions constructed in Definition \ref{def:ApproximationPayoff}. \blue{Given $(r,x) \in [0,T] \times (0,\infty)$, let $X^{\blue{r,x}}=(X_t^{\blue{r,x}})_{t \in [\blue{r},T]}$ the solution to the $G$-SDE $(\ref{G-SDE_New})$ and assume that $X^{r,x}$ is quasi-surely strictly positive. If Assumption \ref{asumC1,2General} is satisfied,} then it holds
$$
\lim_{\epsilon \to \blue{0}} \blue{\mathbb{\hat{E}}}_{\blue{r}}\left[|\varphi(T,X_T^{\blue{r,x}})-\varphi^{\epsilon}(T,X_T^{\blue{r,x}})|\right] = 0.
$$
\end{lemma}
\begin{proof}
We have 
\begin{align}
& \blue{\hat{\mathbb{E}}}_{\blue{r}}\left[|\varphi(T,X_T^{\blue{r,x}})-\varphi^{\epsilon}(T,X_T^{\blue{r,x}})|\right] \notag \\
& \quad \le \blue{\hat{\mathbb{E}}}_{\blue{r}}\left[|\varphi(T,X_T^{\blue{r,x}})-\varphi^{\epsilon}(T,X_T^{\blue{r,x}})|\mathbf{1}_{\{X_T^{\blue{r,x}} \le 2 \epsilon^{-1}\}}\right] +  \blue{\hat{\mathbb{E}}}_{\blue{r}}\left[|\varphi(T,X_T^{\blue{r,x}})-\varphi^{\epsilon}(T,X_T^{\blue{r,x}})|\mathbf{1}_{\{X_T^{\blue{r,x}} > 2 \epsilon^{-1}\}}\right].\label{eq:splittingtheindicator}
\end{align}
In particular, for any $\epsilon \in (0,1)$ it holds
\blue{\begin{align}
&\blue{\hat{\mathbb{E}}}_{\blue{r}}\left[|\varphi(T,X_T^{\blue{r,x}})-\varphi^{\epsilon}(T,X_T^{\blue{r,x}})|\mathbf{1}_{\{X_T^{\blue{r,x}} > 2 \epsilon^{-1}\}}\right] \nonumber \\
&\le \blue{\hat{\mathbb{E}}}_{\blue{r}}\left[|\varphi(T,X_T^{\blue{r,x}})|\mathbf{1}_{\{X_T^{\blue{r,x}} > 2 \epsilon^{-1}\}} \right] + \hat{\mathbb{E}}_{\blue{r}} \left[|\varphi^\epsilon(T,X_T^{\blue{r,x}})|\mathbf{1}_{\{X_T^{\blue{r,x}} > 2 \epsilon^{-1}\}}\right]\notag \\
&\le \blue{\hat{\mathbb{E}}}_{\blue{r}}\left[\tilde{C}(T)(1+(X_T^{r,x})^L)\mathbf{1}_{\{X_T^{\blue{r,x}} > 2 \epsilon^{-1}\}} \right] + \hat{\mathbb{E}}_{\blue{r}} \left[K(T,L)\epsilon^{-L}\mathbf{1}_{\{X_T^{\blue{r,x}} > 2 \epsilon^{-1}\}}\right]\label{eq:ApproximationDetails1} \\
&\le 2\tilde{C}(T)\blue{\hat{\mathbb{E}}}_{\blue{r}}\left[(X_T^{r,x})^L\mathbf{1}_{\{X_T^{\blue{r,x}} > 2 \epsilon^{-1}\}} \right] + 2^{-L} K(T,L)\hat{\mathbb{E}}_{\blue{r}} \left[(2\epsilon^{-1})^{L}\mathbf{1}_{\{X_T^{\blue{r,x}} > 2 \epsilon^{-1}\}}\right]\label{eq:ApproximationDetails1_2}\\
&\le 2\tilde{C}(T)\blue{\hat{\mathbb{E}}}_{\blue{r}}\left[(X_T^{r,x})^L\mathbf{1}_{\{X_T^{\blue{r,x}} > 2 \epsilon^{-1}\}} \right] + 2^{-L} K(T,L)\hat{\mathbb{E}}_{\blue{r}} \left[(X_T^{r,x})^L\mathbf{1}_{\{X_T^{\blue{r,x}} > 2 \epsilon^{-1}\}}\right]\label{eq:ApproximationDetails3} \\
&  =\tilde{K}(T,L)\blue{\hat{\mathbb{E}}}_{\blue{r}}\left[|X_T^{\blue{r,x}}|^L\mathbf{1}_{\{X_T^{\blue{r,x}} > 2 \epsilon^{-1}\}}\right],\label{eq:firstinequalitysecondtermapprox}
 \end{align}
 where we use Assumption \ref{asum:PhiUnbounded} and \blue{\eqref{eq:UnboundedPayoffEstimate1}} in \eqref{eq:ApproximationDetails1}. Moreover, we have that \blue{\eqref{eq:ApproximationDetails1_2} and }\eqref{eq:ApproximationDetails3} follow as $\blue{1<}(2\epsilon^{-1})^{L} <(X_T^{r,x})^L$ on the event $\{X_T^{\blue{r,x}} > 2 \epsilon^{-1}\}$. In \eqref{eq:firstinequalitysecondtermapprox} we set $\tilde{K}(T,L):=2\tilde{C}(T)+2^{-L}K(T,L)$.}
Since $\blue{\hat{\mathbb{E}}}_{\blue{r}}\left[|X_T^{\blue{r,x}}|^L\right]<\infty$ by \eqref{eq:Estimate1}, we can use \violet{Proposition 6.3.2 in \cite{peng_nonlinearExpectation_book}} and Theorem 25 of \cite{denis2011function} to get
 \begin{align}
 \lim_{\epsilon \to 0} \blue{\hat{\mathbb{E}}}_{\blue{r}}\left[|X_T^{\blue{r,x}}|^L\mathbf{1}_{\{X_T^{\blue{r,x}} > 2 \epsilon^{-1}\}}\right] =  \lim_{N \to \infty} \blue{\hat{\mathbb{E}}}_{\blue{r}}\left[|X_T^{\blue{r,x}}|^L\mathbf{1}_{\{X_T^{\blue{r,x}} > N\}}\right] = 0.\label{eq:IntegrabilityOfX}
 \end{align}
 Together with \eqref{eq:firstinequalitysecondtermapprox}, this implies that for any $\delta>0$ there exists $\epsilon_{\delta}\blue{>0}$ such that
  \begin{equation} \label{eq:firstinequalitysecondtermapprox1}
 \blue{\hat{\mathbb{E}}}_{\blue{r}}\left[|\varphi(T,X_T^{\blue{r,x}})-\varphi^{\epsilon}(T,X_T^{\blue{r,x}})|\mathbf{1}_{\{X_T^{\blue{r,x}} > 2 \epsilon^{-1}\}}\right] \blue{<} \frac{\delta}{2}
 \end{equation}
 for any $\epsilon <  \epsilon_{\delta}$.  Moreover, by \eqref{eq:approxphiepsbelow} we have 
 \begin{equation} \label{eq:firstinequalitysecondtermapprox2}
\blue{\hat{\mathbb{E}}}_{\blue{r}}\left[|\varphi(T,X_T^{\blue{r,x}})-\varphi^{\epsilon}(T,X_T^{\blue{r,x}})|\mathbf{1}_{\{X_T^{\blue{r,x}} \le 2 \epsilon^{-1}\}}\right] \blue{<} \epsilon.
\end{equation}
 for any $\epsilon \in (0,1)$. \blue{For any} $\delta>0$, by \eqref{eq:splittingtheindicator}, \blue{\eqref{eq:firstinequalitysecondtermapprox1} and \eqref{eq:firstinequalitysecondtermapprox2}}  we get
$$
\blue{\hat{\mathbb{E}}}_{\blue{r}}\left[|\varphi(T,X_T^{\blue{r,x}})-\varphi^{\epsilon}(T,X_T^{\blue{r,x}})|\right] \blue{<} \delta 
$$
for any $\epsilon < \bar \epsilon$ \blue{with $\bar \epsilon = \min(\epsilon_{\delta}, \frac{\delta}{2})$.}
\end{proof}
We now get the following result.
\begin{theorem}\label{thm:firstresultforunboundedpayoff}
	Let $\varphi$ satisfy Assumption \ref{asum:PhiUnbounded} and $(\varphi^{\epsilon})_{\epsilon \in (0,1)}$ be the \blue{family} of functions constructed in Definition \ref{def:ApproximationPayoff}. \blue{Given $(r,x) \in [0,T] \times (0,\infty)$,} assume that the process $X^{\blue{r,x}}=(X_t^{\blue{r,x}})_{t \in [\blue{r},T]}$ in $(\ref{G-SDE_New})$ is quasi-surely strictly positive. Moreover, let Assumptions \ref{asumC1,2General} and \ref{asumPayoffNoTimeDependence} be satisfied. Then there exists a family $(u^{\epsilon})_{\epsilon \in (0,1)}$ of viscosity solutions for the PDEs
	\begin{align}
	u_t^{\epsilon} + 2G \left(u_x^{\epsilon} g(t,x) + \frac{1}{2} u_{xx}^{\epsilon} \left(h(t,x)\right)^2\right)	 + f(t,x)u_x^{\epsilon} - xu^{\epsilon} =0, \quad &\text{ for } (t,x) \in Q  \label{eq:PDEUnbounded1}\\
		u^{\epsilon}=\varphi^{\epsilon} \quad &\text{ for } (t,x) \in \partial'Q \label{eq:PDEUnbounded2}
\end{align}
for $Q=(0,T) \times (0,\infty)$, with the property that
	\begin{equation*}
		\lim_{\epsilon \to 0} \left \vert u^{\epsilon}(r,x)- \hat{\mathbb{E}}_r\left[ \varphi(T,X_T^{r,x})e^{-\int_r^T X_s^{r,x}ds}\right] \right \vert = 0
	\end{equation*}
	for every $(r,x) \in [0,T] \times (0,\infty)$.
\end{theorem}
\begin{proof}
	By Theorem \ref{theorem:StabilityBounded} it follows that for every fixed $\epsilon > 0$ there exists a viscosity solution $u^{\epsilon}$ to the PDE \eqref{eq:PDEUnbounded1} - \eqref{eq:PDEUnbounded2} with the representation
\begin{equation} 
	u^{\epsilon}(r,x):=\hat{\mathbb{E}}_r\left [ \varphi^{\epsilon}(T,X_T^{r,x}) e^{-\int_r^T X_s^{r,x} ds}\right]
\end{equation}
for $(r,x) \in [0,T] \times (0,\infty)$.
Thus, by Lemma \blue{\ref{lem:convergenceapproxvarphinotbounded}} and since $e^{-\int_r^T X_s^{r,x} ds} \le 1$ we have
	\begin{align*}
		u^{\epsilon}(r,x)\to \hat{\mathbb{E}}_r\left [ \varphi(T,X_T^{r,x}) e^{-\int_r^T X_s^{r,x} ds}\right], \quad \text{for } \epsilon \to 0.
	\end{align*}
\end{proof}
\begin{theorem}\label{thm:secondresultforunboundedpayoff}
	Let $\varphi$ satisfy Assumption \ref{asum:PhiUnbounded}. \blue{Given $(r,x) \in [0,T] \times (0,\infty)$,} assume that the process $X^{\blue{r,x}}=(X_t^{\blue{r,x}})_{t \in [\blue{r},T]}$ in $(\ref{G-SDE_New})$ is quasi-surely strictly positive. Moreover, let Assumptions \ref{asumC1,2General}, \ref{asumPayoffNoTimeDependence} be satisfied. Define now the function $u:[0,T] \times (0,\infty) \to \mathbb{R}$ by
	\begin{equation} \label{eq:DefinitionValueFunctionUnbounded}
			u(r,x):=\hat{\mathbb{E}}_r\left [ \varphi(T,X_T^{r,x}) e^{-\int_r^T X_s^{r,x} ds}\right].
	\end{equation}
	Then $u$ is a viscosity solution of the following PDE
	\begin{align}
	u_t + 2G \left(u_x g(t,x) + \frac{1}{2} u_{xx}\left(h(t,x)\right)^2\right)	 + f(t,x)u_x - xu =0, \quad &\text{ for } (t,x) \in Q  \label{eq:PDEUnboundedStability1}\\
		u=\varphi \quad &\text{ for } (t,x) \in \partial'Q \label{eq:StabilityPDEUnbounded2}
\end{align}
for $Q=(0,T) \times (0,\infty)$.
\end{theorem}
\begin{proof}
	We apply a stability result similar as in the proof of Theorem \ref{theorem:StabilityBounded}. In order to do that, we need to prove that the family $(u^{\epsilon})_{\epsilon \in (0,1)}$ introduced in Theorem \ref{thm:firstresultforunboundedpayoff} converges locally uniformly to $u$ for $\epsilon \to 0$. This means that for every fixed $0 \leq r_1 \leq r_2 \leq T$ and $0 <x_1 <x_2$, for all $\delta>0$ there exists a constant $\eta>0$ such that for every $\epsilon < \eta$ it holds
	\begin{equation*}
		\sup_{(r,x) \in [r_1,r_2] \times [x_1,x_2]}\vert u^{\epsilon}(r,x)-u(r,x) \vert < \delta.
	\end{equation*}
      Let $(\varphi^{\epsilon})_{\epsilon \in (0,1)}$ be the family of approximating functions for $\varphi$ from Definition \ref{def:ApproximationPayoff}. For $(r,x) \in [r_1,r_2] \times [x_1,x_2]$ we have
	\begin{align}
		&\vert u^{\epsilon}(r,x)-u(r,x) \vert  \nonumber \\
		&= \left| \hat{\mathbb{E}}_r \left[ \varphi^{\epsilon} \left( T,X_T^{r,x} \right) e^{-\int_r^T X_s^{r,x}ds}\right]-\hat{\mathbb{E}}_r \left[ \varphi \left( T,X_T^{r,x} \right) e^{-\int_r^T X_s^{r,x}ds}\right] \right| \nonumber \\
		&\leq  \hat{\mathbb{E}}_r \left[ \left| \varphi^{\epsilon} \left( T,X_T^{r,x} \right) e^{-\int_r^T X_s^{r,x}ds}- \varphi \left( T,X_T^{r,x} \right) e^{-\int_r^T X_s^{r,x}ds}\right|\right]  \nonumber \\
		&\leq  \hat{\mathbb{E}}_r \left[ \left| \varphi^{\epsilon} \left( T,X_T^{r,x} \right) - \varphi \left( T,X_T^{r,x} \right) \right|\right] \nonumber \\
		&\leq  \hat{\mathbb{E}}_r \left[ \left| \varphi^{\epsilon} \left( T,X_T^{r,x} \right) - \varphi \left( T,X_T^{r,x} \right) \right| \textbf{1}_{\lbrace X_T^{r,x} > 2 \epsilon^{-1}\rbrace}\right]+\hat{\mathbb{E}}_r \left[ \left| \varphi^{\epsilon} \left( T,X_T^{r,x} \right) - \varphi \left( T,X_T^{r,x} \right) \right| \textbf{1}_{\lbrace X_T^{r,x} \leq 2 \epsilon^{-1}\rbrace}\right] \label{eq:UnboundedPayoffStability1}.
	\end{align}
	We now focus on the first term in \eqref{eq:UnboundedPayoffStability1}. We prove here that for all $\delta>0$ there exist $\eta>0$ such that for $\epsilon < \eta $ it holds
	\begin{equation} \label{eq:ConvergenceSupremum}
	\sup_{r \in [r_1,r_2]} \hat{\mathbb{E}}_r \left[ \left| \varphi^{\epsilon} \left( T,X_T^{r,x} \right) - \varphi \left( T,X_T^{r,x} \right) \right| \textbf{1}_{\lbrace X_T^{r,x} > 2 \epsilon^{-1}\rbrace}\right] <\frac{\delta}{2}
	\end{equation}
	for all $x \in [x_1,x_2]$.
	We have
	\allowdisplaybreaks{
	\begin{align}
	&\sup_{r \in [r_1,r_2]} \hat{\mathbb{E}}_r \left[ \left| \varphi^{\epsilon} \left( T,X_T^{r,x} \right) - \varphi \left( T,X_T^{r,x} \right) \right| \textbf{1}_{\lbrace X_T^{r,x} > 2 \epsilon^{-1}\rbrace}\right] \nonumber \\
		&\leq \sup_{r \in [r_1,r_2]} \hat{\mathbb{E}}_r \left[ \left| \varphi^{\epsilon} \left( T,X_T^{r,x} \right) \right| \textbf{1}_{\lbrace X_T^{r,x} > 2 \epsilon^{-1}\rbrace}\right] + \sup_{r \in [r_1,r_2]} \hat{\mathbb{E}}_r \left[ \left|  \varphi \left( T,X_T^{r,x} \right) \right| \textbf{1}_{\lbrace X_T^{r,x} > 2 \epsilon^{-1}\rbrace}\right] \nonumber \\
		&\leq K(T,L)2^{-L}\sup_{r \in [r_1,r_2]} \hat{\mathbb{E}}_r \left[  (X_T^{r,x})^L \textbf{1}_{\lbrace X_T^{r,x} > 2 \epsilon^{-1}\rbrace}\right] + 2 \tilde{C}(T) \sup_{r \in [r_1,r_2]} \hat{\mathbb{E}}_r \left[  (X_T^{r,x})^L \textbf{1}_{\lbrace X_T^{r,x} > 2 \epsilon^{-1}\rbrace}\right] \label{eq:PolynomialBound1} \\
		&=\tilde{K}(T,L) \sup_{r \in [r_1,r_2]} \hat{\mathbb{E}}_r \left[  (X_T^{r,x})^L \textbf{1}_{\lbrace X_T^{r,x} > 2 \epsilon^{-1}\rbrace}\right] \label{eq:PolynomialBound2}  \\
		& \leq \tilde{K}(T,L) \sup_{r \in [r_1,r_2]} \hat{\mathbb{E}}_r \left[  (X_T^{r,x_2})^L \textbf{1}_{\lbrace X_T^{r,x_2} > 2 \epsilon^{-1}\rbrace}\right] \label{eq:PolynomialBound2New}  \\
		&\leq \sup_{r \in [r_1,r_2]} \left(\left(\hat{\mathbb{E}}_r \left[  (X_T^{r,x_2})^{2L}\right]\right)^{\frac{1}{2}} \left( \hat{\mathbb{E}}_r \left[ \textbf{1}_{\lbrace X_T^{r,x_2} > 2 \epsilon^{-1}\rbrace}\right]\right)^{\frac{1}{2}}\right) \label{eq:Hoelder}\\
		&\leq \sup_{r \in [r_1,r_2]} \left(\hat{\mathbb{E}}_r \left[  (X_T^{r,x_2})^{2L}\right]\right)^{\frac{1}{2}} \sup_{r \in [r_1,r_2]} \left( \hat{\mathbb{E}}_r \left[ \textbf{1}_{\lbrace X_T^{r,x_2} > 2 \epsilon^{-1}\rbrace}\right]\right)^{\frac{1}{2}}. \label{eq:UnboundedPayoffStability1New}
	\end{align}}
	We used Assumption \ref{asum:PhiUnbounded} and \eqref{eq:UnboundedPayoffEstimate1} in \eqref{eq:PolynomialBound1} and set $\tilde{K}(T,L):=2\tilde{C}(T)+K(T,L)2^{-L}$ in \eqref{eq:PolynomialBound2}. Furthermore, \eqref{eq:PolynomialBound2New} holds by similar arguments as in Lemma \ref{lemma:x1andx2}. Moreover, \eqref{eq:Hoelder} follows by Hoelder's inequality.
	For the first term in \eqref{eq:UnboundedPayoffStability1New} we have
	\begin{align}
		\sup_{r \in [r_1,r_2]} \left(\hat{\mathbb{E}}_r \left[  (X_T^{r,x_2})^{2L}\right]\right)^{\frac{1}{2}} &\leq \sup_{r \in [r_1,r_2]} \left(\hat{\mathbb{E}}_r \left[  \sup_{r \in [r_1,r_2]}  (X_T^{r,x_2})^{2L}\right]\right)^{\frac{1}{2}} \nonumber \\
		&= \sup_{r \in [r_1,r_2]} \left(\hat{\mathbb{E}}_r \left[  \sup_{r \in [r_1,r_2]}  (X_T^{r,x_2}-x_2+x_2)^{2L}\right]\right)^{\frac{1}{2}} \nonumber \\
		&\leq \sup_{r \in [r_1,r_2]} \left (\overline{K}(x_2,L)  \left(\hat{\mathbb{E}}_r \left[  \sup_{r \in [r_1,r_2]}  (X_T^{r,x_2}-x_2)^{2L}\right]\right)^{\frac{1}{2}} \right)\label{eq:UnboundedPayoffStability2}\\
		&\leq \overline{K}(x_2,L) C (1+\vert x_2 \vert^{2L})^{\frac{1}{2}} (r_2-r_1)^{\frac{L}{2}}, \label{eq:UnboundedPayoffStability3}
	\end{align}
	where we use \eqref{eq:Estimate2} in \eqref{eq:UnboundedPayoffStability2} and $\overline{K}(x_2,L)>0$ denotes a suitable constant.
	Putting together \eqref{eq:UnboundedPayoffStability1New}, \eqref{eq:UnboundedPayoffStability3}, it follows 
	\begin{align}
		&\sup_{r \in [r_1,r_2]} \hat{\mathbb{E}}_r \left[ \left| \varphi^{\epsilon} \left( T,X_T^{r,x} \right) - \varphi \left( T,X_T^{r,x} \right) \right| \textbf{1}_{\lbrace X_T^{r,x} > 2 \epsilon^{-1}\rbrace}\right] \nonumber \\
		&\leq \overline{K}(x_2,L) C (1+\vert x_2 \vert^{2L})^{\frac{1}{2}} (r_2-r_1)^{\frac{L}{2}}\sup_{r \in [r_1,r_2]} \left( \hat{\mathbb{E}}_r \left[ \textbf{1}_{\lbrace X_T^{r,x_2} > 2 \epsilon^{-1}\rbrace}\right]\right)^{\frac{1}{2}} \nonumber \\
		&\leq \overline{K}(x_2,L) C (1+\vert x_2 \vert^{2L})^{\frac{1}{2}} (r_2-r_1)^{\frac{L}{2}}\sup_{r \in [r_1,r_2]} \left( \hat{\mathbb{E}}_r \left[ \textbf{1}_{\lbrace \tau_{\epsilon}^{1,r,x_2}<T\rbrace}\right]\right)^{\frac{1}{2}} \label{eq:UnboundedPayoffStability4} \\
		&\leq \overline{K}(x_2,L) C (1+\vert x_2 \vert^{2L})^{\frac{1}{2}} (r_2-r_1)^{\frac{L}{2}}\left( \sup_{r \in [r_1,r_2]}  \hat{\mathbb{E}}_r \left[ \textbf{1}_{\lbrace \tau_{\epsilon}^{1,r,x_2}<T\rbrace}\right]\right)^{\frac{1}{2}}  \to 0, \quad \text{ for } \epsilon \to 0 \label{eq:UnboundedPayoffStability5} 
	\end{align}
	for $x \in [x_1,x_2]$. Here, \eqref{eq:UnboundedPayoffStability4} follows as 
	$$
	\left \lbrace \omega \in \Omega: X_T^{r,x}(\omega) \geq 2 \epsilon^{-1} \right \rbrace \subseteq \left \lbrace \omega \in \Omega: \sup_{t \in [r,T]} X_t^{r,x}(\omega) \geq \epsilon^{-1} \right \rbrace.
	$$  
	Moreover, \eqref{eq:UnboundedPayoffStability5} holds as $x \mapsto \sqrt{x}$ is increasing and continuous and by Assumption \ref{asumPayoffNoTimeDependence}.
	Thus \eqref{eq:ConvergenceSupremum} holds. %and it implies that for all $\delta >0$ there exists $\bar{\epsilon}>0 $ such that for all $\epsilon < \bar{\epsilon}$ we have
%	\begin{equation}
%		\hat{\mathbb{E}}_r \left[ \left| \varphi^{\epsilon} \left( T,X_T^{r,x} \right) - \varphi \left( T,X_T^{r,x} \right) \right| \textbf{1}_{\lbrace X_T^{r,x} > 2 \epsilon^{-1}\rbrace}\right]< \frac{\delta}{2}.
%	\end{equation}
	Moreover, for the second term in \eqref{eq:UnboundedPayoffStability1} for any $\delta>0$ and any $\epsilon < \delta/2$ it holds
	\begin{align*}
		\hat{\mathbb{E}}_r \left[ \left| \varphi^{\epsilon} \left( T,X_T^{r,x} \right) - \varphi \left( T,X_T^{r,x} \right) \right| \textbf{1}_{\lbrace X_T^{r,x} \leq 2 \epsilon^{-1}\rbrace}\right] < \frac{\delta}{2}.
	\end{align*} 
	Thus, we conclude that $(u^{\epsilon})_{\epsilon \in (0,1)}$ converges locally uniformly to $u$ for $\epsilon \to 0$. Moreover, by the second point of Lemma \ref{lemma:ConvergencePDEEquation} it follows by the same stability arguments as in the proof of Theorem \ref{theorem:StabilityBounded} that $u$ is a viscosity solution of the PDE \eqref{eq:PDEUnboundedStability1}-\eqref{eq:StabilityPDEUnbounded2}.
	\end{proof}
}
\violet{
Next, we prove a generalization of Proposition \ref{prop:ExtendedMartingaleProperty}, i.e. we show that the process  $M^{r,x}=(M_t^{r,x})_{t \in [r,T]}$ for $(r,x) \in [0,T] \times (0,\infty)$ defined by
 	\begin{equation} \label{eq:ProcessForDynamicProgramming}
 	M_t^{r,x}:=u(t,X_t^{r,x})e^{-\int_{r}^t X_{s}^{r,x}ds},	 \quad t \in [r,T]
 	\end{equation}
 	is a $G$-martingale. Here, the function $u:[0,T] \times (0,\infty) \to \mathbb{R}$ is given in \eqref{eq:DefinitionValueFunctionUnbounded}.
\begin{prop}\label{prop:dynpro}
	Let $\varphi$ satisfy Assumption \ref{asum:PhiUnbounded}. {Given $(r,x) \in [0,T] \times (0,\infty)$,} assume that the process $X^{{r,x}}=(X_t^{{r,x}})_{t \in [{r},T]}$ in $(\ref{G-SDE_New})$ is quasi-surely strictly positive. Moreover, let Assumptions \ref{asumC1,2General} and \ref{asumPayoffNoTimeDependence} be satisfied.  Then for any $(r,x) \in [0,T] \times (0,\infty)$ the process $M^{r,x}$ in \eqref{eq:ProcessForDynamicProgramming} is a $G$-martingale.
\end{prop}
\begin{proof}
\red{First, we prove that $M_t^{r,x} \in L_G^1(\Omega_t)$ for any $t \in [r,T]$. Since $\hat{\mathbb{E}}_t[\cdot]: L_G^1(\Omega_T) \to L_G^{1}(\Omega_{t})$ and
 	\begin{equation} \label{eq:RepresentationWithConditionalExpectation}
 	M_t^{r,x}=	u(t,X_t^{r,x})e^{-\int_{r}^t X_{s}^{r,x}ds} = \hat{\mathbb{E}}_t \left[ \varphi\left(T,X_T^{t,X_t^{r,x}}\right) e^{-\int_{r}^T X_{s}^{r,x}ds}\right], \quad t \in [r,T],
 	\end{equation}
 it is enough to show that 
 $$
 M_T^{r,x}=\varphi\left(T,X_T^{t,X_t^{r,x}}\right) e^{-\int_{r}^T X_{s}^{r,x}ds}\in L_G^{1}(\Omega_T).
 $$
 By the characterization of $L_G^1(\Omega_T)$ in Exercise 6.5.2 in \cite{peng_nonlinearExpectation_book}, we need to show that $M_T^{r,x}$ is a quasi-continuous random variable and that 
 \begin{equation} \label{eq:IntegrabilityCondition}
 \lim_{n \to \infty} \hat{\mathbb{E}} \left[  \vert M_T^{r,x} \vert \textbf{1}_{\lbrace \vert M_T^{r,x}\vert >n \rbrace}\right]=0.
 \end{equation}
 As $X^{r,x}$ is a quasi-continuous process, it follows that $\int_{r}^T X_s^{r,x}ds$ is a quasi-continuous random variable. Thus, $M_T^{r,x}$ is quasi-continuous by the continuity of $u(T, \cdot)$ and $\exp(-\cdot)$. Moreover, by Assumption \ref{asum:PhiUnbounded} and by \eqref{eq:IntegrabilityOfX}, we have
\begin{align*}
&\lim_{n \to \infty} \hat{\mathbb{E}} \left[ \left \vert \varphi\left(T,X_T^{t,X_t^{r,x}}\right) e^{-\int_{r}^T X_{s}^{r,x}ds} \right \vert \textbf{1}_{\left \lbrace \left \vert  \varphi\left(T,X_T^{t,X_t^{r,x}}\right) e^{-\int_{r}^T X_{s}^{r,x}ds}\right \vert >n \right \rbrace}\right] \\
&\leq \lim_{n \to \infty} \hat{\mathbb{E}} \left[ \tilde{C}(T)\left(1+ \left \vert X_T^{t,X_t^{r,x}} \right \vert^L\right) \textbf{1}_{\left \lbrace \tilde{C}(T)\left(1+ \left \vert X_T^{t,X_t^{r,x}} \right \vert^L\right) >n\right \rbrace}\right] =0.
\end{align*}
We have then proved that $M_t^{r,x} \in L_G^1(\Omega_t)$. Fix now $0 \leq r \leq t <z \leq T$. By \eqref{eq:RepresentationWithConditionalExpectation}} we have
 	\begin{align}
 		\hat{\mathbb{E}}_t \left[ M_z^{r,x}\right]
 		 	&=\hat{\mathbb{E}}_t \left[ \hat{\mathbb{E}}_z \left[ \varphi\left(T,X_T^{z,X_z^{r,x}}\right) e^{-\int_{r}^T X_{s}^{r,x}ds}\right] \right]\nonumber \\
 		&=\hat{\mathbb{E}}_t \left[  \varphi\left(T,X_T^{z,X_z^{r,x}}\right) e^{-\int_{r}^T X_{s}^{r,x}ds} \right]\nonumber \\
 		&=\hat{\mathbb{E}}_t \left[  \varphi\left(T,X_T^{t,X_t^{r,x}}\right) e^{-\int_{r}^T X_{s}^{r,x}ds} \right]\nonumber \\
 		&=M_t^{r,x},  \label{eq:MartingaleProperty}
 	\end{align}
 	which shows that $M^{r,x}$ \red{satisfies the} $G$-martingale \red{property}. 
\end{proof}
 In particular, Proposition \ref{prop:dynpro} provides the dynamic programming principle
	\begin{align*}
		u(t,X_t^{r,x})=\hat{\mathbb{E}}_t \left[ u(z,X_z^{r,x})e^{-\int_z^T X_{s}^{r,x}ds}\right] e^{\int_r^T X_s^{r,x}ds},	 \quad t \in [r,T]
	\end{align*}
	where $u$ is given in \eqref{eq:DefinitionValueFunctionUnbounded}. 
}
\begin{remark}
	In financial applications \blue{the} results derived in Section \ref{sec:results} are especially of interest when $X$ models a short rate or default/mortality intensity. In this case, the result\blue{s} in Theorem\blue{s} \blue{\ref{theorem:StabilityBounded} and \ref{thm:secondresultforunboundedpayoff}} allow to \blue{represent} the price of a zero-coupon bond with maturity $T$ by a solution of a nonlinear PDE. However, compared to the classical case the $G$-SDE in $(\ref{G-SDE_New})$ contains an additional term, namely the integral with respect to the quadratic variation of the $G$-Brownian motion. This extra term can be regarded as an additional uncertain drift term, as the process $\langle B\rangle$ represents mean-uncertainty in the $G$-setting.  
\end{remark}

\begin{remark}
	 Most of the previous results only hold for a quasi-surely strictly positive process $X$ in $(\ref{G-SDE_New})$. This assumption is for example satisfied for all process $X$ of the form $e^Y$, where $Y=(Y_t)_{t \in [0,T]}$ is a process such that $-\infty < Y$ quasi-surely. In Section \ref{sec:numerical} we give an example of such a process, namely a generalized version of the exponential Vasicek model. 
\end{remark}
\begin{remark} 
\red{The aim of this paper is to provide a Feynman-Kac formula in the $G$-setting which allows to numerically compute the following $G$-conditional expectation $$ 
	\hat{\mathbb{E}}_r\left [ \varphi(T,X_T^{r,x}) e^{-\int_r^T X_s^{r,x} ds}\right].
	$$ 
	To do so, we start from the $G$-SDE in \eqref{G-SDE_New} under the standard assumption of Lipschitz continuous coefficients in the $G$-setting. To the best of our knowledge there are no general results without the Lipschitz assumption which guarantee existence and uniqueness of the solution for a $G$-SDE. The only contributions dealing with $G$-SDEs with non-Lipschitz coefficients are \cite{bai_lin_2014}, \cite{lin_2013}. \\In order to model interest rates under model uncertainty there exist two different approaches in the literature. First, \cite{fns_2019} introduces non-linear affine processes by considering a family of continuous semimartingale laws as a starting points, such that the differential characteristics are bounded from below and above from affine functions. However, this approach is not applicable in our setting as it does not consider $G$-SDEs. Second, in \cite{hoelzermann_2019}, \cite{hoelzermann_2020}, \cite{hoelzermann_quian_2020} a HJM framework in the $G$-setting is considered. More specifically, the forward rate dynamics is modeled by a $G$-SDE. However, also here only Lipschitz continuous coefficients in the corresponding $G$-SDEs are considered. }
\end{remark}

\section{Numerical applications}\label{sec:numerical}
\blue{In this section we use an explicit Euler scheme to numerically approximate the solution $ u$ to the PDE \eqref{eq:PDEUnboundedStability1}-\eqref{eq:StabilityPDEUnbounded2} for some choices of its coefficients, of the payoff function and of the volatility uncertainty interval $[\underline{\sigma}, \overline{\sigma}]$. By Theorem \ref{thm:secondresultforunboundedpayoff}, such solution represents the $G$-expectation of the discounted payoff with underlying given by \eqref{G-SDE_New}.}

%In this way we can  simulate $\hat u^{\epsilon}(t , X_{t}),$ $0 \le t \le T$, that according to Theorem \ref{theorem:PointwiseConvergence} may give a good approximation of
%\begin{equation}\label{eq:conexpectationtoapproximate}
%\hat{\mathbb{E}}_t \left[ \varphi (T,X_T) e^{-\int_t^T X_s ds}\right], \quad 0 \le t \le T.
%\end{equation}

\begin{remark}
\blue{We choose to numerically solve} \blue{ \eqref{eq:PDEUnboundedStability1}-\eqref{eq:StabilityPDEUnbounded2}} with an explicit Euler scheme in order to avoid artificial boundary conditions, since the domain is unbounded and it is hard to guess a possible behaviour of the solution for large values of the spatial coordinate. 

Using  the explicit Euler scheme we are of course left with the problem of the stability of the numerical solution. Computing stability conditions in our case appears to be a difficult task, since the PDE is fully nonlinear and with nonconstant coefficients. This goes beyond the scope of this paper and can be object of further investigation. However, for the time discretization step we choose for our numerical simulations, we observe a behavior in line with the theory. \blue{For} example, \blue{the values we find for the classical case $\underline{\sigma} = \overline{\sigma}$ for a call option with an underlying following log-normal dynamics agree with the well known analytical formulas}.
\end{remark}
\blue{We focus on two examples.} 

%
%In order to simulate the process $X$ in \eqref{G-SDE_New}, we numerically generate the trajectories of both the $G$-Brownian motion $B$ and its quadratic variation by following the approach of \cite{NumericalSimulationG-Expectation}. In particular, given a time discretization 
%$$
%0=t_0<t_1<\dots<t_n=T,
%$$
%we simulate the $G$-Brownian increments $\Delta B_{t_k}$, $k=0,\dots,n-1$, solving numerically for some values of $a\in (-\infty,\infty)$ the PDE
%\begin{align}
%u_t &= G\left(u_{xx}\right), \quad 0< t \le 1,\notag \\
%u(0,x)&=\mathbf{1}_{\{x \le a\}},
%\end{align}
%whose solution gives the distribution $\hat{\mathbb{E}}[\mathbf{1}_{\{\Delta B_{t_k} \le a\}}]$, $k=0,\dots,n-1$. We get the missing values of the distribution by a linear interpolation and obtain the increments by inversion sampling. We then recover the dynamics of the $G$-Brownian motion by adding the increments. The quadratic variation is obtained by using its definition as the limit of squared increments. 
%
%We now focus on two examples.
\subsection{Dothan model with quadratic variation term}\label{Example:DothanWith}
 We here \blue{consider the stochastic process $X^{r,x}=(X^{r,x}_t)_{t \in [r,T]}$ which follows the $G$-SDE
\begin{align}\label{eq:dothan}
		dX^{r,x}_t&=\beta X^{r,x}_t dt + \gamma X^{r,x}_t d \langle B \rangle_t+ \alpha X^{r,x}_t dB_t, \quad r \le t \le T,\\
		X^{r,x}_r&=x,
	\end{align}
	with $\alpha >0$ and $(r,x) \in (0,T) \times (0,\infty)$.
	\\
		As $X^{r,x}$ is quasi-surely strictly positive, we can apply Theorem \ref{thm:secondresultforunboundedpayoff} to approximate the $G$-expectation
	$$
	u(r,x) := \hat{\mathbb{E}}_r\left [ \varphi(T,X^{r,x}_T) e^{-\int_r^T X^{r,x}_s ds}\right], \quad (r,x) \in (0,T) \times  (0, \infty),
	 $$ 
	 for a given payoff function $\varphi$ satisfying Assumption \ref{asum:PhiUnbounded}, as a numerical solution to the PDE  \eqref{eq:PDEUnboundedStability1}-\eqref{eq:StabilityPDEUnbounded2} associated to \eqref{eq:dothan}, i.e., with coefficients
$$
f(t,x):=\beta x, \qquad g(t,x):=\gamma x, \qquad h(t,x):=\alpha x.
$$
We also set }
$$
\varphi(T,x)=(x-K)^+.
$$ 
In financial applications, $\varphi(T,x)=(x-K)^+$ represents the payoff of a call option on \blue{$X$} (or of a caplet on \blue{$X$}, if \blue{$X$} models an interest rate). 

 \blue{We take parameters $T=1$, $K=3$, $\beta=0.3$, $\alpha=0.4$, $\overline{\sigma}=1.0$, and let $\underline{\sigma}$ vary in $\lbrace 0.2, 0.5, 0.8, 1 \rbrace$,  $\gamma$ in $\lbrace -0.1, - 0.3 \rbrace$. The right boundary of our approximated domain is $(0,1) \times \{20\}$. 
 \begin{remark}
 	\blue{Having to consider a bounded domain represents a further approximation of our solution with respect to the analytic one. As commented above, we avoid to fix artificial boundary conditions since we apply an Explicit Euler scheme. In our plots we represent the approximated solution up to $(0,1) \times \{4\}$, which is far away from the right boundary of the approximated domain.}
 \end{remark}
Figures \ref{fig:solutionsg-01} and \ref{fig:solutionsg-03} show the values of the approximated solutions $\hat u(r,x)$ for $(r,x) \in (0,1) \times (0,4)$. We note that the values of the approximated expectations are increasing with the uncertainty on the volatility. Moreover,  the expectation is higher for $\gamma = -0.1$ and smaller for $\gamma = -0.3$. This is not surprising at least for the classical case (i.e., when $\underline{\sigma} = 1$), where $\gamma$ gets just summed to the drift term since $\langle B \rangle_t=t$.} 

 \begin{figure}[]
\centering 
\subfloat[$\underline{\sigma}=0.2$]{\includegraphics[scale = 0.4]{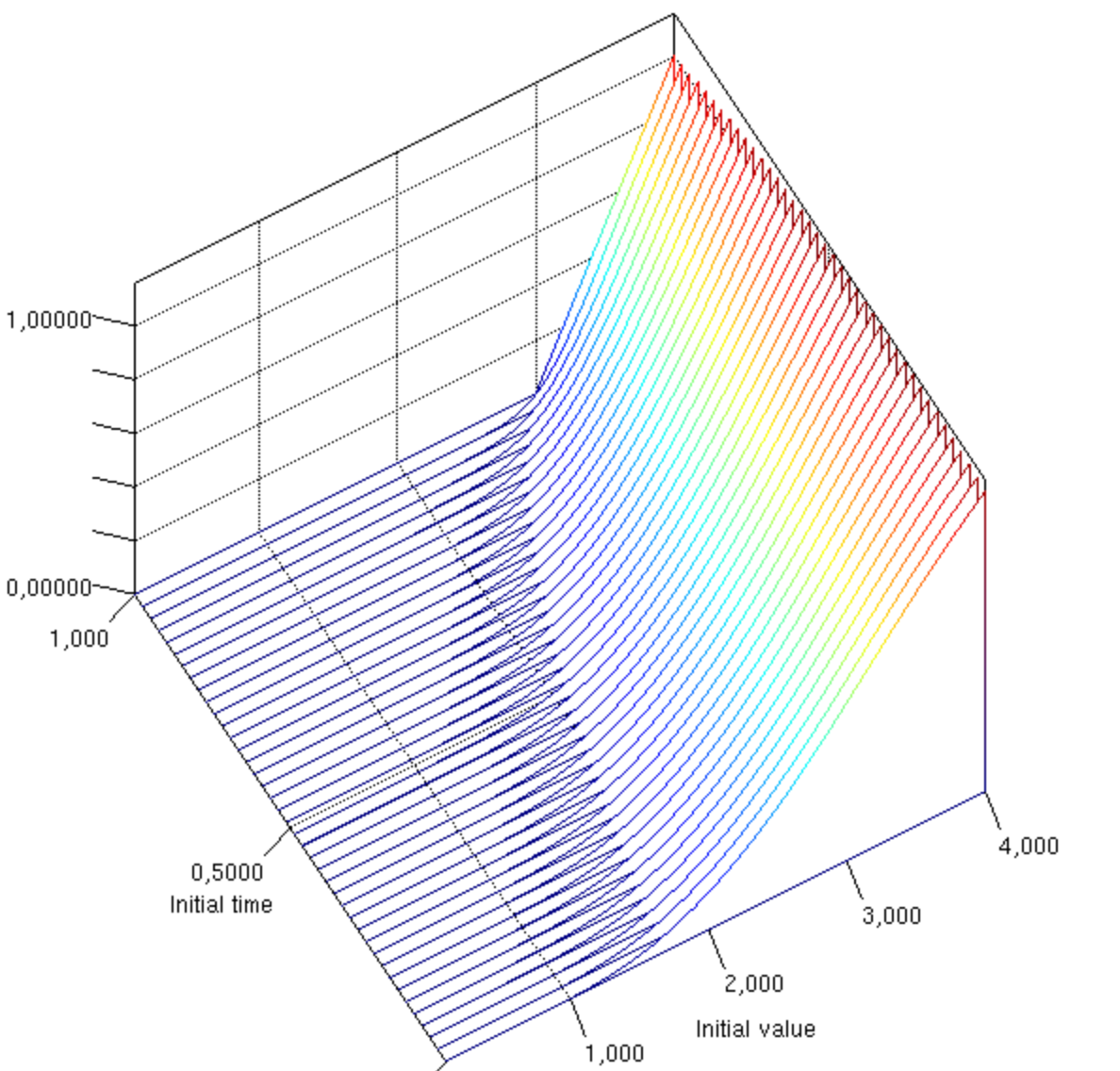}}
\subfloat[$\underline{\sigma}=0.5$]{\includegraphics[scale = 0.4]{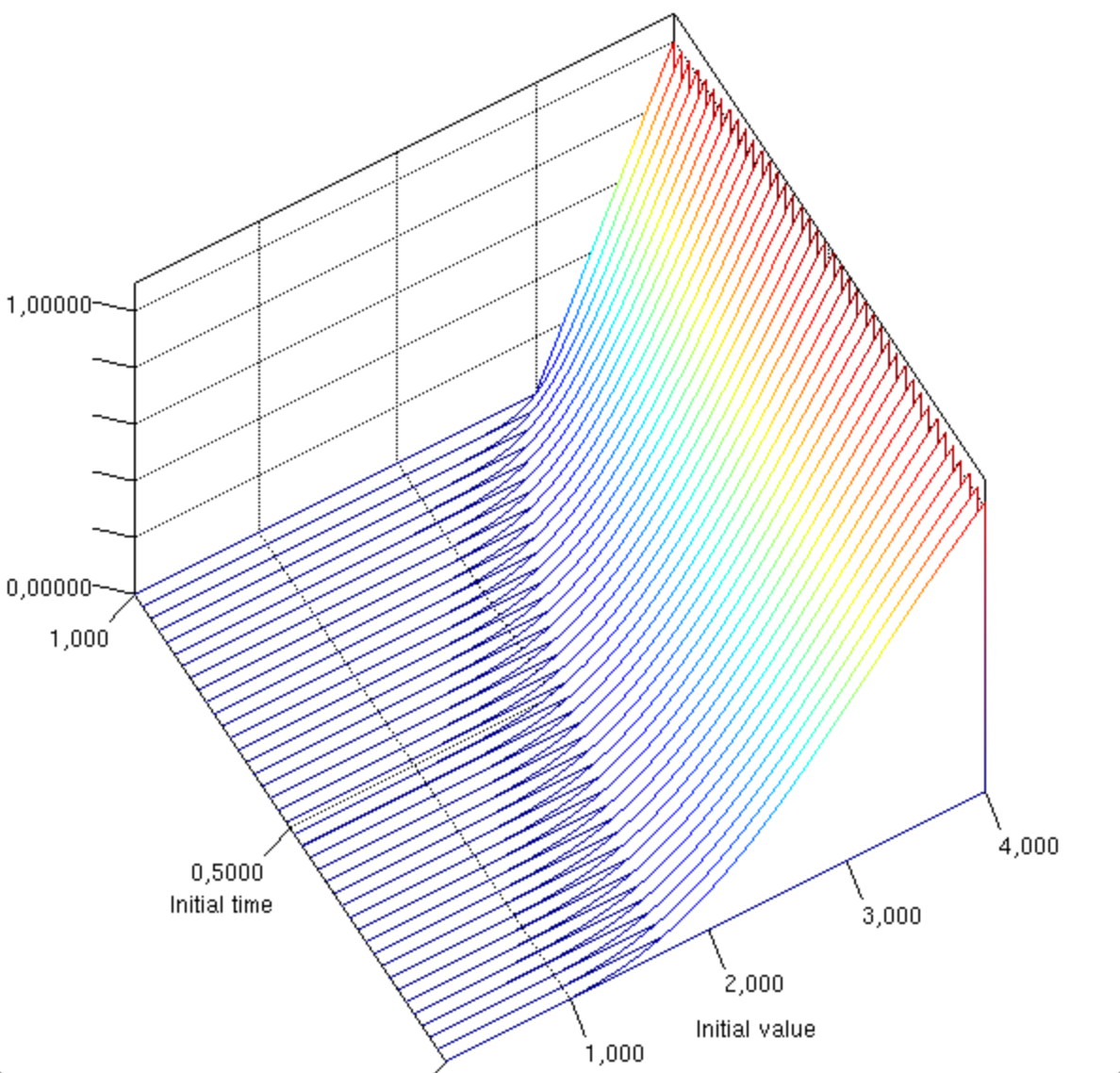}}\\
\subfloat[$\underline{\sigma}=0.8$]{\includegraphics[scale = 0.4]{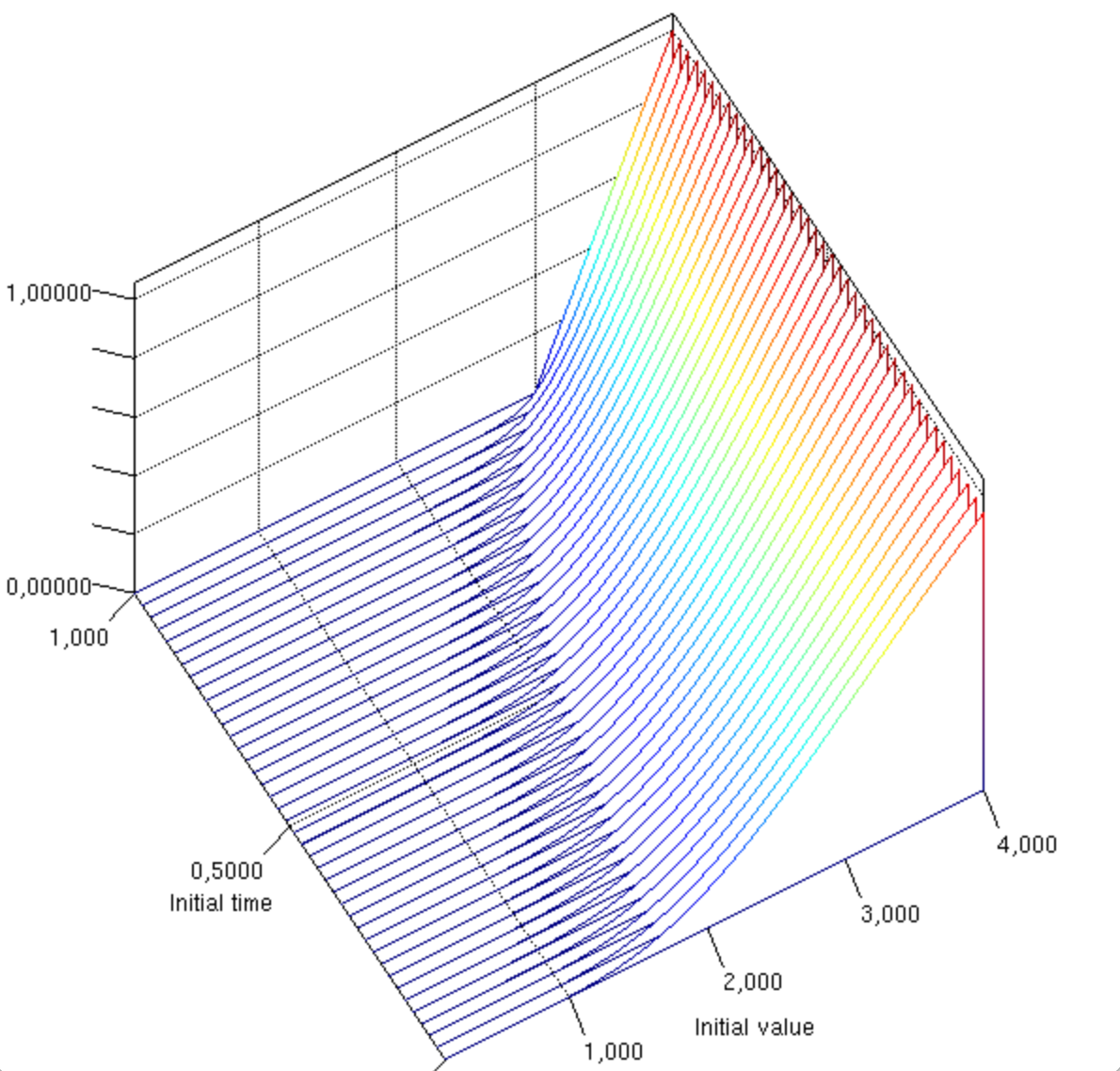}}
\subfloat[$\underline{\sigma}=1$]{\includegraphics[scale = 0.4]{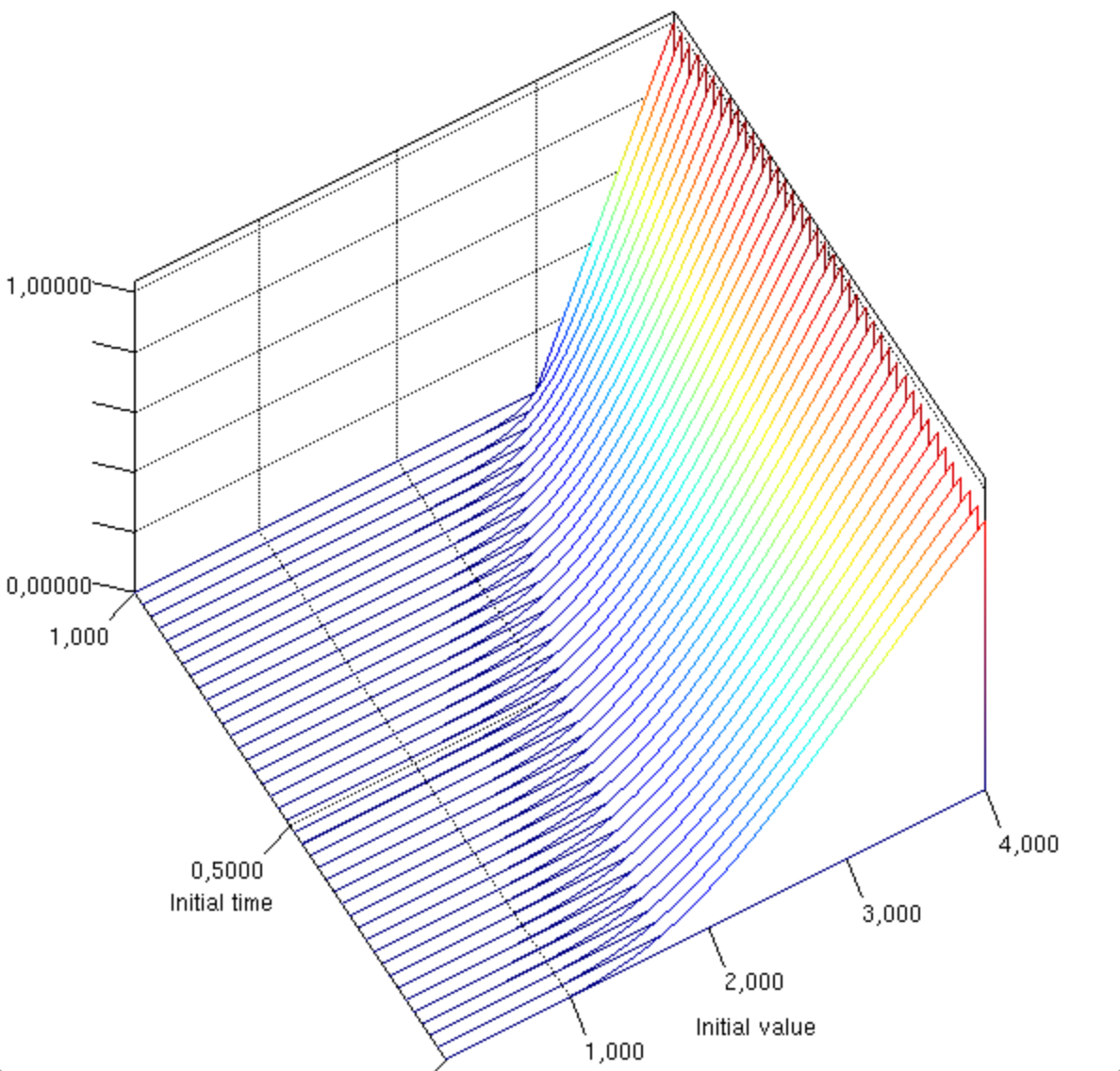}} \caption{\blue{Approximation of $ \hat{\mathbb{E}}_r\left [(X_1^{r,x}-3)^+ e^{-\int_r^1 X^{r,x}_s ds}\right]$ for $(r,x) \in (0,1) \times (0,4)$, when $X^{r,x}$ follows the Dothan model $(\ref{eq:dothan})$ with parameters $\beta=0.3$, $\alpha=0.4$, $\gamma = -0.1$.}}
\label{fig:solutionsg-01}
 \end{figure}

 \begin{figure}[]
\centering 
\subfloat[$\underline{\sigma}=0.2$]{\includegraphics[scale = 0.4]{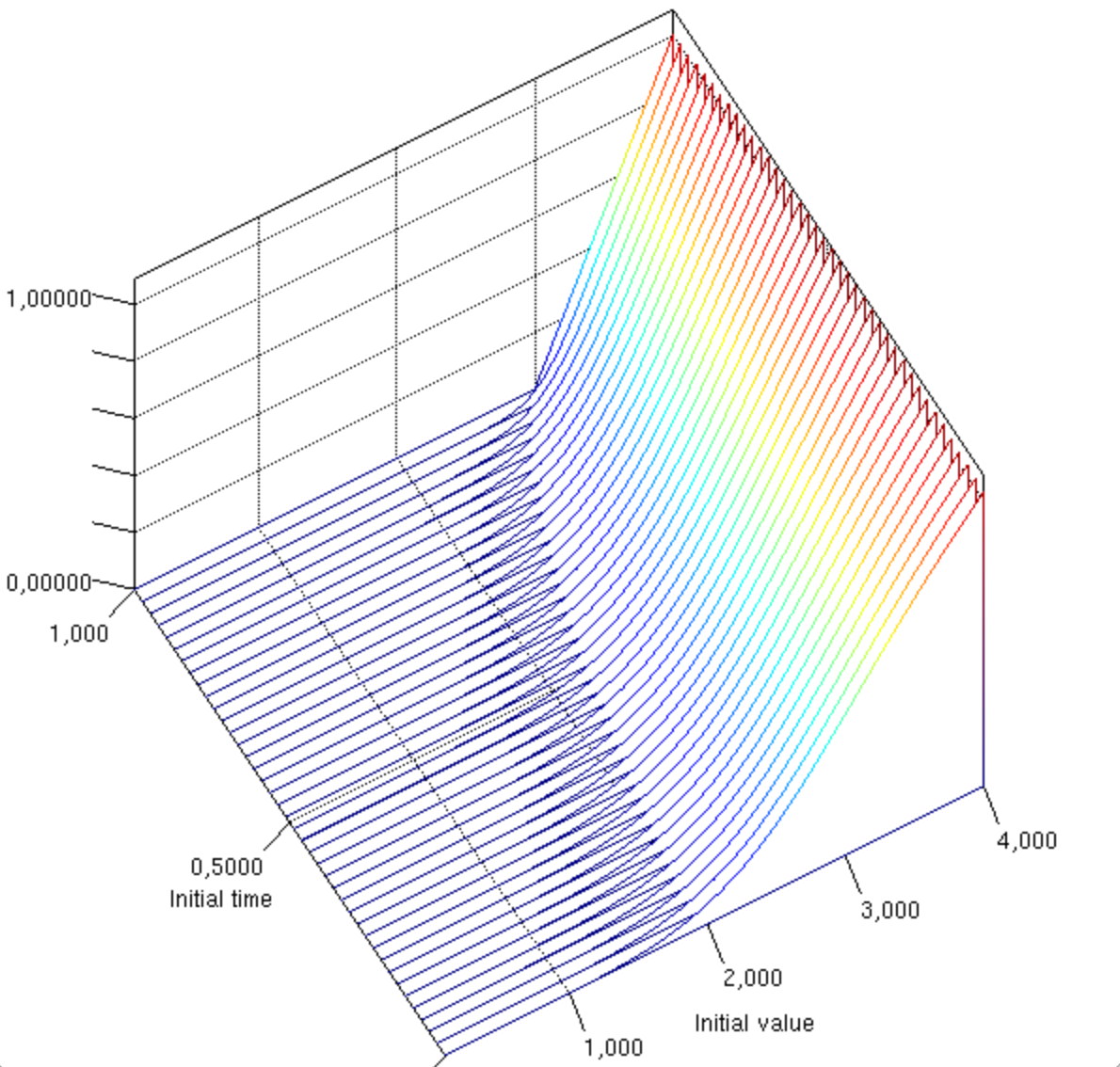}}
\subfloat[$\underline{\sigma}=0.5$]{\includegraphics[scale = 0.4]{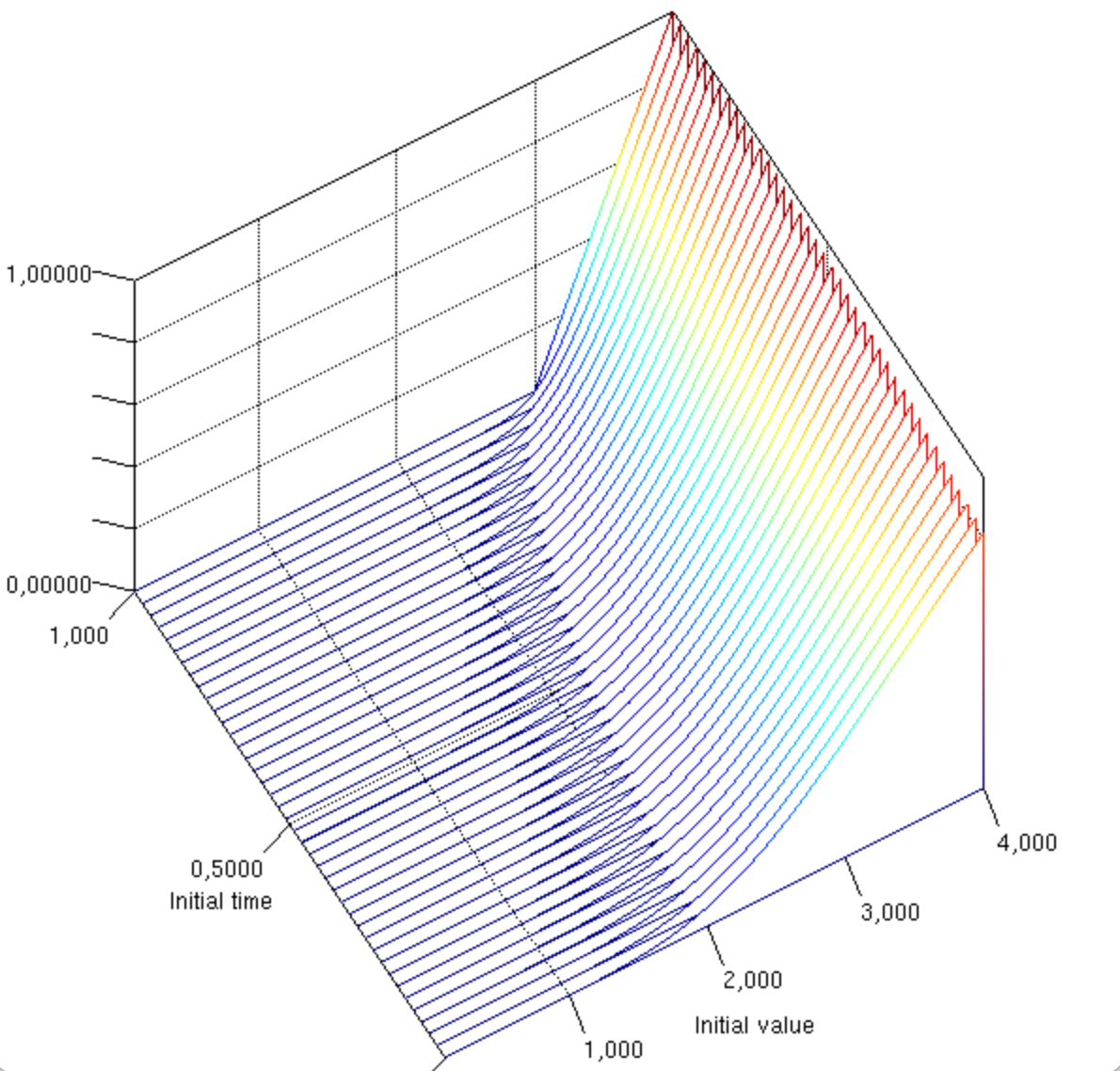}}\\
\subfloat[$\underline{\sigma}=0.8$]{\includegraphics[scale = 0.4]{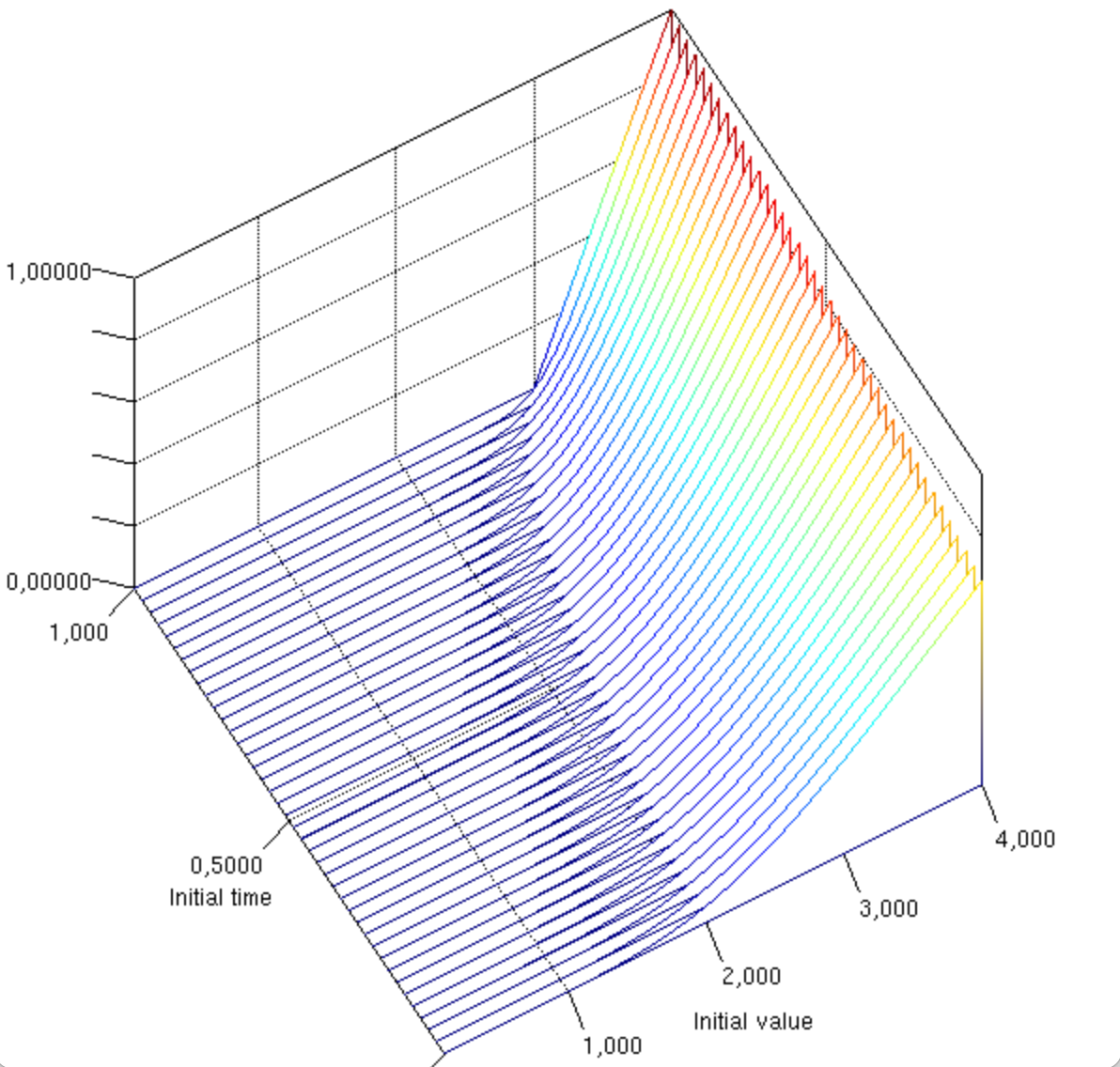}}
\subfloat[$\underline{\sigma}=1$]{\includegraphics[scale = 0.4]{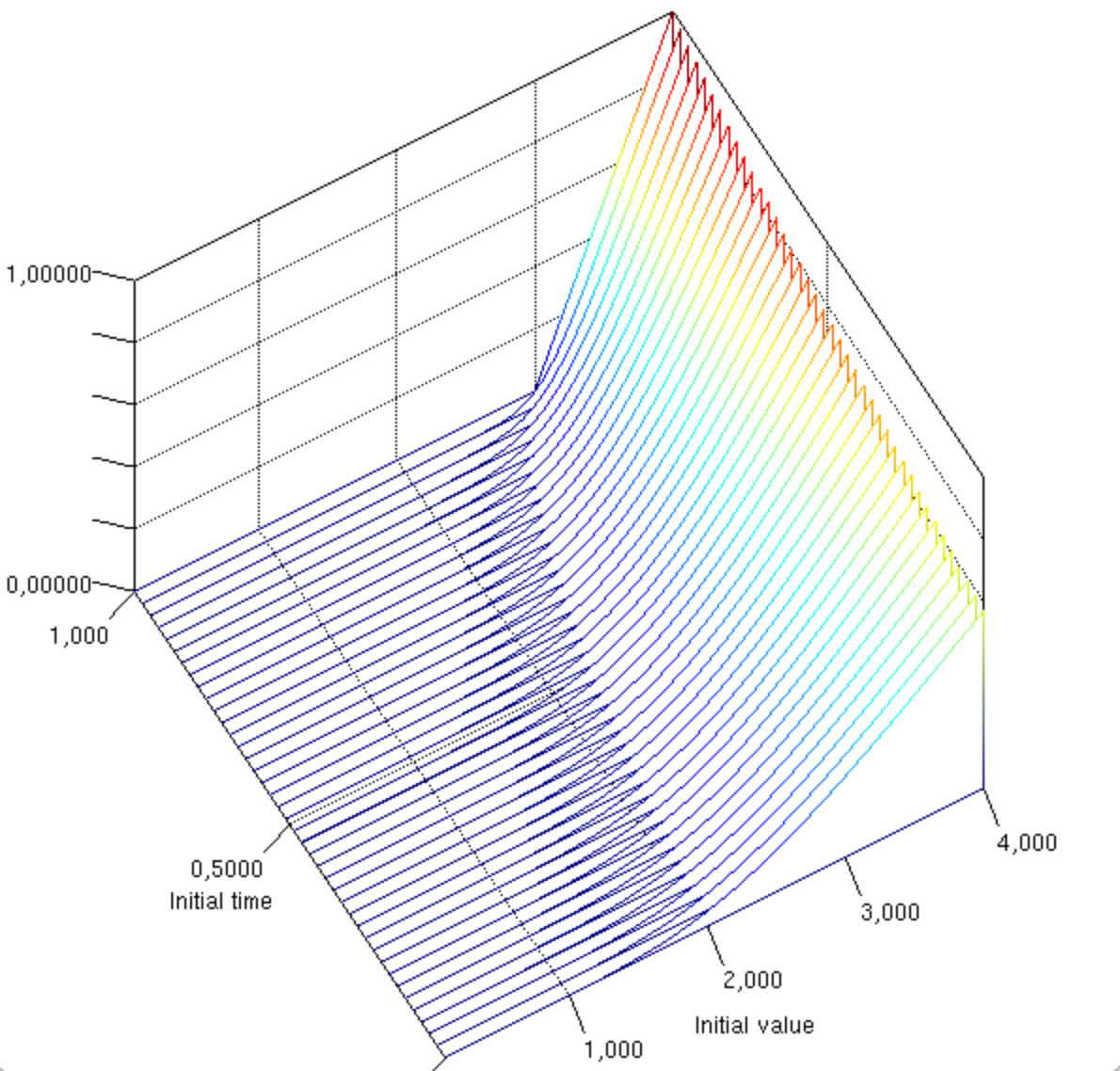}} \caption{\blue{Approximation of $ \hat{\mathbb{E}}_r\left [(X_1^{r,x}-3)^+ e^{-\int_r^1 X^{r,x}_s ds}\right]$ for $(r,x) \in (0,1) \times (0,4)$, when $X^{r,x}$ follows the Dothan model $(\ref{eq:dothan})$ with parameters $\beta=0.3$, $\alpha=0.4$, $\gamma = -0.3$.}}
\label{fig:solutionsg-03}
 \end{figure}

\subsection{Exponential Vasicek model with quadratic variation term}\label{ex:expvasicek} 
Consider now the process $Y^{r,y}=(Y^{r,y}_t)_{t \in [r,T]}$ given by the $G$-SDE
\begin{align}
	dY_t^{r,y}&=k(\theta - Y^{r,y}_t)dt+ \tilde{k}(\tilde{\theta}-Y^{r,y}_t)d\langle B \rangle_t+ \alpha dB_t, \quad r \le t \le T, \notag\\
	Y_r&=y, \notag
\end{align}
for $(r,y) \in (0,T) \times \mathbb{R}$, which has a unique solution by Theorem \violet{5.1.3 in \cite{peng_nonlinearExpectation_book}}, \blue{and set $$X^{r,x}_t:=e^{Y^{r,y}_t}$$ for $t \in [r,T]$ and $x:=e^y$}. Since $Y^{r,y}>-\infty$ quasi-surely \blue{for any $(r,y) \in (0,T) \times \mathbb{R}$  because it is quasi-continuous, $X^{r,x}$ is quasi-surely strictly positive for any $(r,x) \in (0,T) \times (0, \infty)$ }as required in Theorem \ref{theorem:PointwiseConvergence}. 

\blue{We apply the $G$-It\^{o}'s formula to $X^{r,x}$, in order to compute the coefficients of the PDE. Since the function $y \to e^y$ is not bounded on the whole real line, we choose a big enough constant $M>0$ such that $y \in (-\infty,\log(M))$ and stop the process $Y^{r,y}$ at stopping times $\bar{\tau}_{M}^{r,x}$ given by
$$
\bar{\tau}_{M}^{r,x} := \inf \lbrace t \in [0,T]: X^{r,x}_t \notin (0, M] \rbrace \wedge T = \inf \lbrace t \in [0,T]: Y^{r,y}_t \notin (-\infty,\log(M)] \rbrace \wedge T,
$$
so that using similar arguments as in Lemma \ref{QuasiContinuity} it follows that $\bar{\tau}_{M}^{r,x}$ is a quasi-continuous stopping time for every $M > 0$ and any $(r,x) \in (0,T) \times (0, M)$}. \blue{Then we have that} \blue{\begin{align} 
	dX_{t \wedge \bar{\tau}_{M}^{r,x}}^{r,x}&=X_{t \wedge \bar{\tau}_{M}^{r,x}}^{r,x} k (\theta-\log(X_{t \wedge \bar{\tau}_{M}^{r,x}}^{r,x}))dt +  X_{t \wedge \bar{\tau}_{M}^{r,x}}^{r,x} \left(\tilde{k}(\tilde{\theta}-\log(X_{t \wedge \bar{\tau}_{M}^{r,x}}^{r,x}))+\frac{1}{2}\alpha^2\right)d\langle B \rangle_t \notag \\ & \quad +  \alpha X_{t \wedge \bar{\tau}_{M}^{r,x}}^{r,x} dB_t, \quad r \le t \le T. \label{eq:XSDE}
\end{align}}
Note that \blue{the} SDE \blue{\eqref{eq:XSDE} has} non Lipschitz coefficients. However, the existence and uniqueness of \blue{$X^{r,x}$} \blue{follows as} \blue{$X^{r,x}=e^{Y^{r,y}}$}.
We then numerically solve the PDE \blue{\eqref{eq:PDEUnboundedStability1}-\eqref{eq:StabilityPDEUnbounded2}} with 
$$
f(t,x):=x k (\theta-\log(x)), \qquad g(t,x):=x \left(\tilde{k}(\tilde{\theta}-\log(x))+\frac{1}{2}\alpha^2\right), \qquad h(t,x):=\alpha x.
$$
\blue{We consider the sum of a call and of a put option with same strike, i.e., we fix the final condition
$$
\varphi(T,x)=(x-K)^++(K-x)^+.
$$
We take parameters  $T=1$, $K=2$, $\overline{\sigma}=1.0$, $\theta = 1$, $\tilde{\theta}=1$, $k = \tilde{k}=0.7$, $\alpha=0.6$. As above, we let $\underline{\sigma}$ vary in $\{0.2, 0.5, 0.8, 1\}$. The right boundary of our approximated domain is $(0,1) \times \{20\}$. Figure \ref{fig:expvasicek} shows the values of the approximated solutions $\hat u(r,x)$ for $(r,x) \in (0,1) \times (0,4)$. Again, we note that the values of the approximated expectations is increasing with the uncertainty on the volatility.} 

 \begin{figure}[]
\centering 
\subfloat[$\underline{\sigma}=0.2$]{\includegraphics[scale = 0.4]{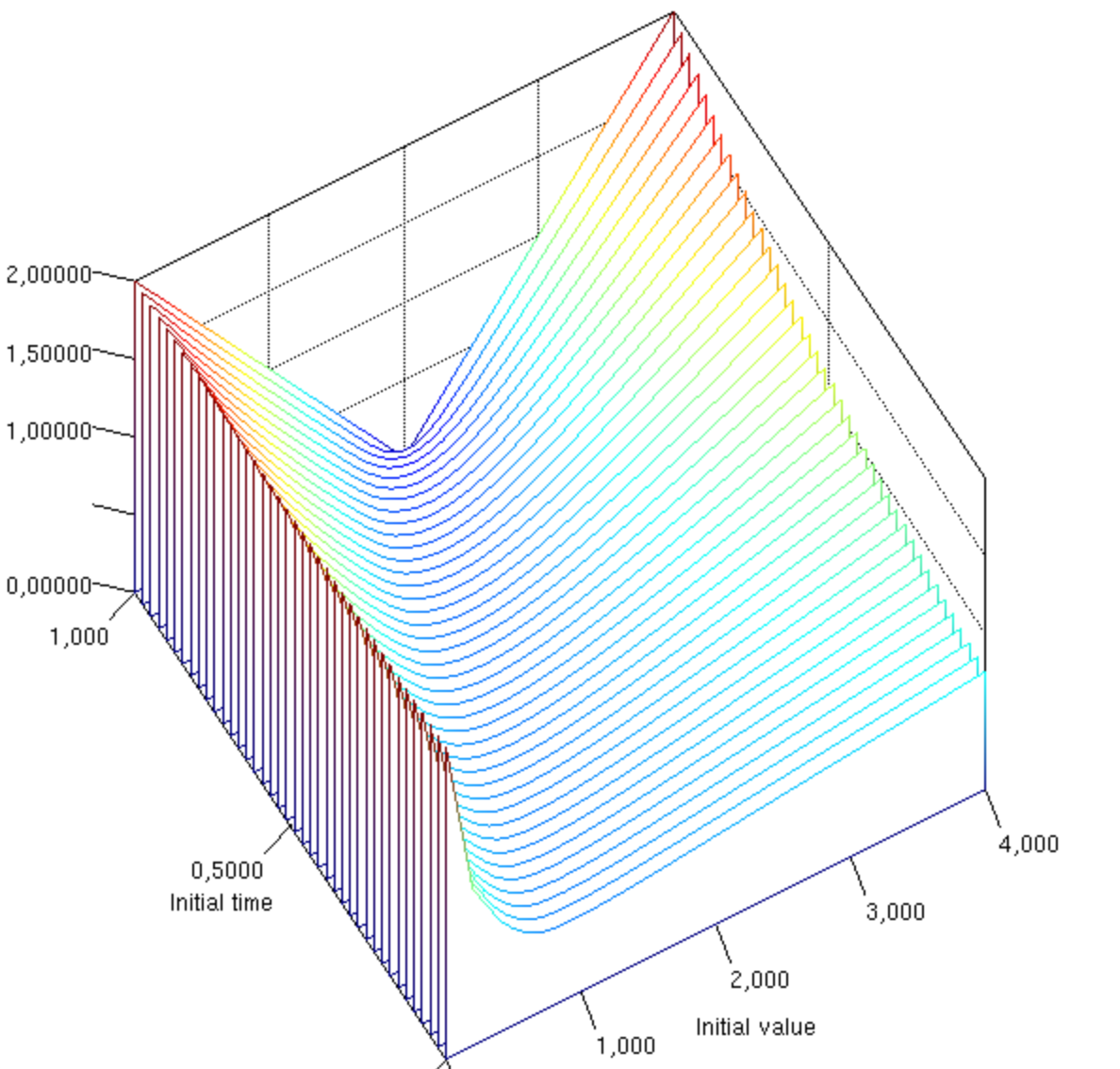}}
\subfloat[$\underline{\sigma}=0.5$]{\includegraphics[scale = 0.4]{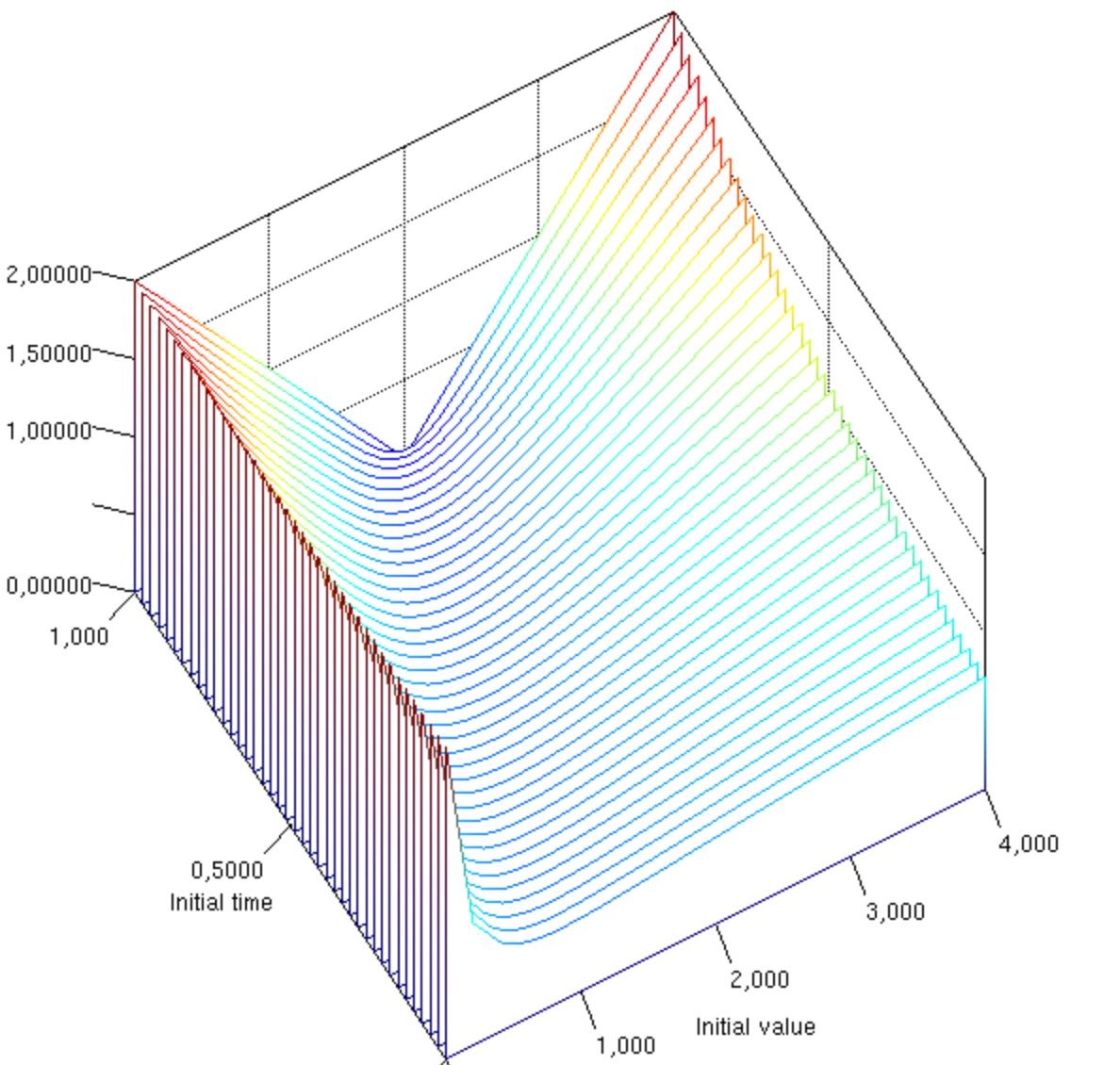}}\\
\subfloat[$\underline{\sigma}=0.8$]{\includegraphics[scale = 0.4]{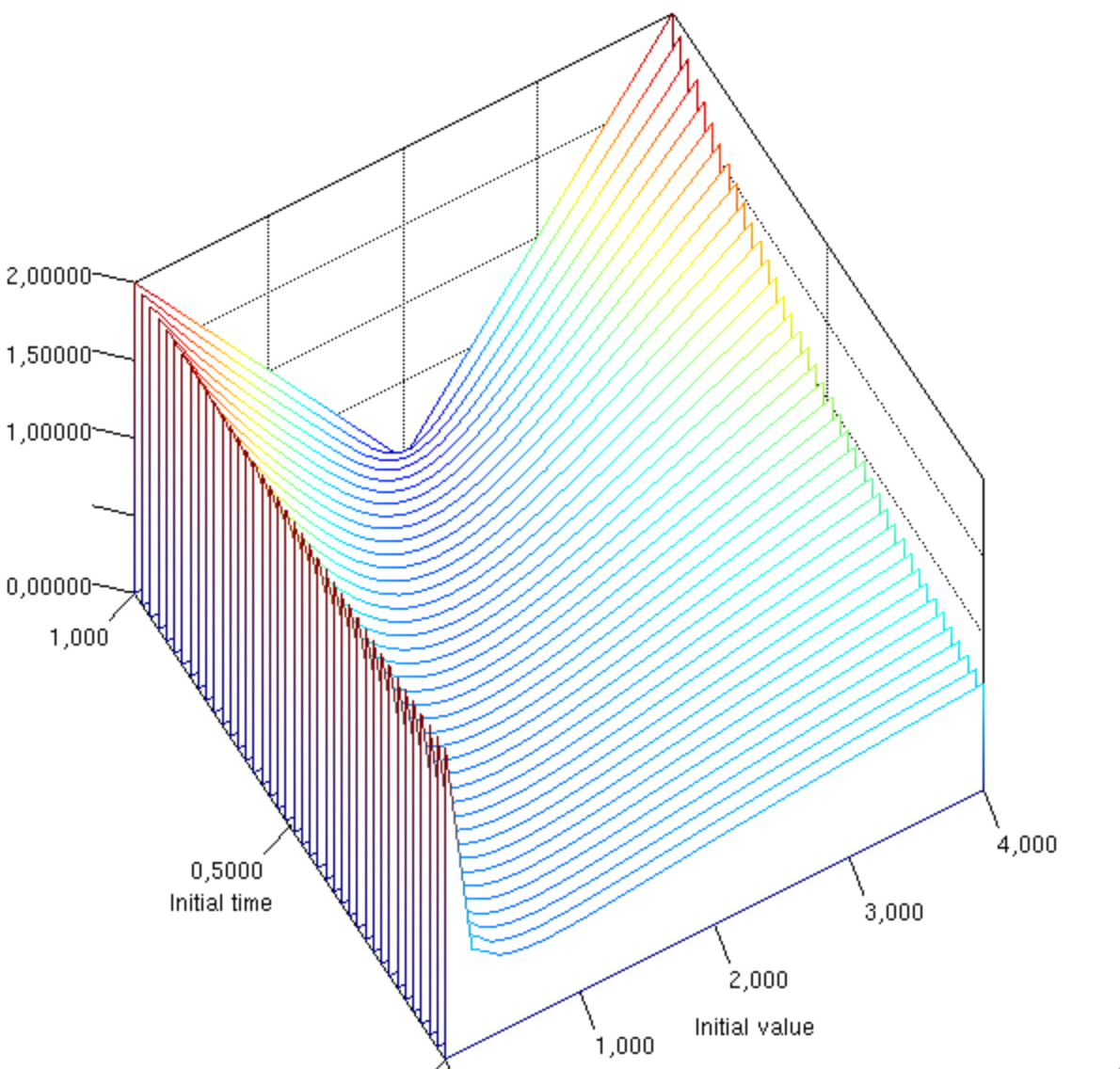}}
\subfloat[$\underline{\sigma}=1$]{\includegraphics[scale = 0.4]{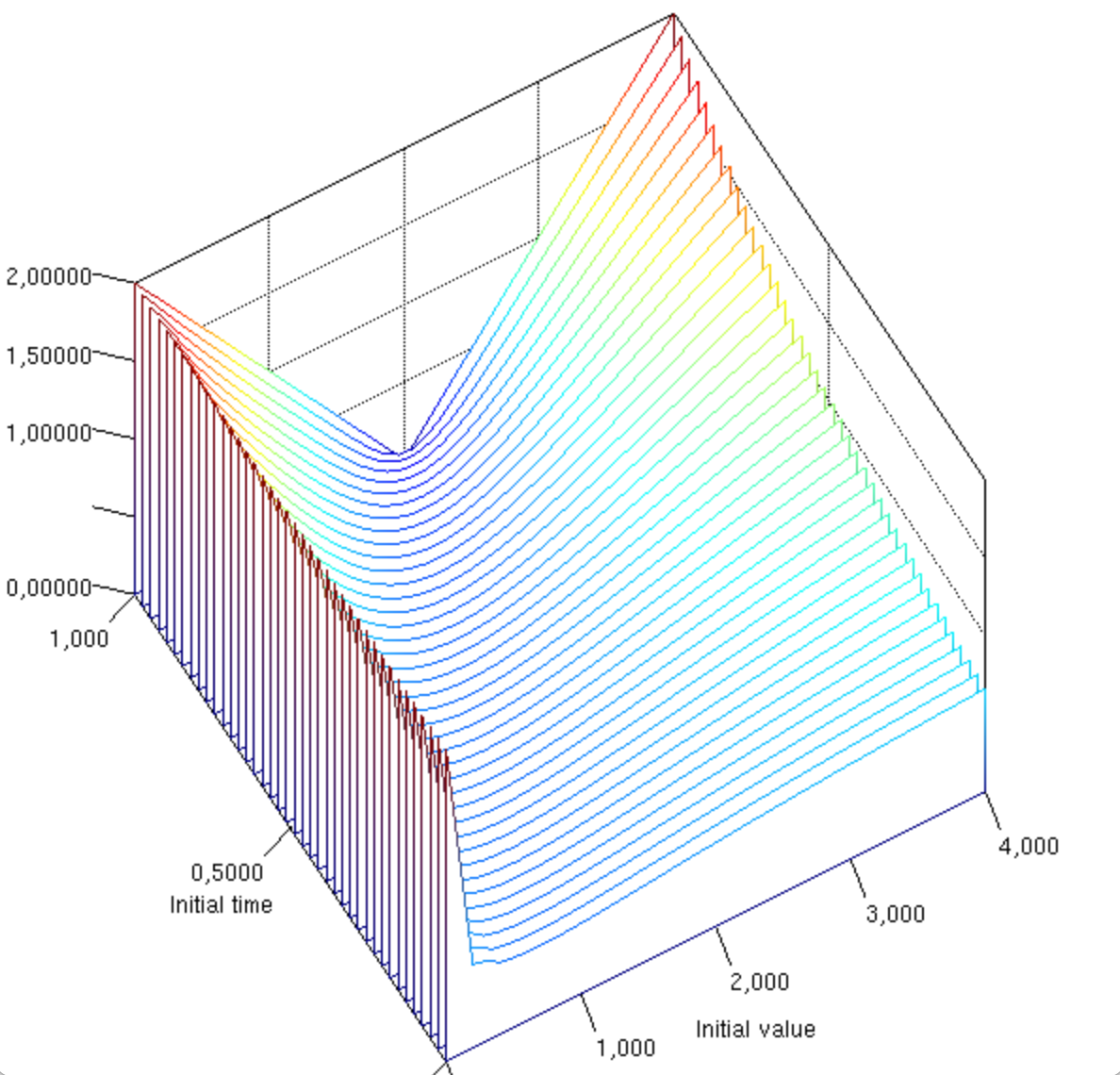}}
\caption{\blue{Approximation of $ \hat{\mathbb{E}}_r\left [\left((X_1^{r,x}-2)^++(2-X_1^{r,x})^+\right) e^{-\int_r^1 X^{r,x}_s ds}\right]$ for $(r,x) \in (0,1) \times (0,4)$, when $X^{r,x}$ follows the exponential Vasicek model in \eqref{eq:XSDE}, with parameters $\theta = \tilde{\theta}=1$, $k = \tilde{k}=0.3$, $\alpha=0.8$,  $\underline{\sigma}=0.2, 0.5, 0.8, 1$ and $\overline{\sigma}=1.0$.}}
\label{fig:expvasicek}
 \end{figure}

\appendix

\section{Computations of the partial derivatives for $u^{\epsilon}$ in Proposition \ref{prop:ExistenceExtendedSolution}}
\black{Let $a^{\epsilon},b^{\epsilon}$ be defined as in $(\ref{eq:aEpsilon}),(\ref{eq:bEpsilon})$ in Proposition \ref{prop:ExistenceExtendedSolution}. Then by straightforward computations we obtain}
\begin{align*}
	a^{\epsilon}_{\blue{x}}(\sigma,t,x)&=2 \sigma h^{\epsilon}(t,x)  h^{\epsilon}_{\blue{x}}(t,x), \\
	 a^{\epsilon}_{\blue{xx}}(\sigma,t,x)&= 2 \sigma \big[ h^{\epsilon}_{\blue{xx}}(t,x) h^{\epsilon}(t,x) + (h^{\epsilon}_{\blue{x}}(t,x))^2 \big], \\
	 a^{\epsilon}_{\blue{t}}(\sigma,t,x)&= 2 \sigma h^{\epsilon}(t,x)  h^{\epsilon}_{\blue{t}}(t,x)
\end{align*} 
and 
\begin{align*}
	 b^{\epsilon}_{\blue{x}}(\sigma,t,x,y,z)&=z \big[  g^{\epsilon}_{\blue{x}}(t,x) \sigma +  f^{\epsilon}_{\blue{x}}(t,x) \big ] - \vartheta^{\epsilon}_{\blue{x}}(x)y, \\
	b^{\epsilon}_{\blue{xx}}(\sigma,t,x,y,z)&=z \big [ g^{\epsilon}_{\blue{xx}}(t,x) \sigma + f^{\epsilon}_{\blue{xx}}(t,x) \big ] -\vartheta^{\epsilon}_{\blue{xx}}(x)y, \\
	b_{\blue{xy}}^{\blue{\epsilon}}(\sigma,t,x,y,z)&=b_{\blue{yx}}^{\blue{\epsilon}}(\sigma,t,x,y,z)=-\vartheta^{\epsilon}_{\blue{x}}(x)\\
 	b^{\epsilon}_{\blue{xz}}(\sigma,t,x,y,z)&=  b^{\epsilon}_{\blue{zx}}(\sigma,t,x,y,z)=g^{\epsilon}_{\blue{x}}(t,x) \sigma + f^{\epsilon}_{\blue{x}}(t,x), \\
	 b^{\epsilon}_{\blue{y}}(\sigma,t,x,y,z)&=-\vartheta^{\epsilon}(x) \blue{+\epsilon},\\
	b^{\epsilon}_{\blue{z}}(\sigma,t,x,y,z)&= g^{\epsilon}(t,x) \sigma + f^{\epsilon}(t,x),\\
	 b^{\epsilon}_{\blue{yy}}(\sigma,t,x,y,z)&=  b^{\epsilon}_{\blue{zz}}(\sigma,t,x,y,z) =  b^{\epsilon}_{\blue{yz}}(\sigma,t,x,y,z)=   b^{\epsilon}_{\blue{zy}}(\sigma,t,x,y,z)= 0, \\
	b^{\epsilon}_{\blue{t}}(\sigma,t,x,y,z)&= z \big[ g^{\epsilon}_{\blue{t}}(t,x) \sigma + f^{\epsilon}_{\blue{t}}(t,x) \big ]. 
\end{align*}

\section{Risky asset and short rate driven by a $d$-dimensional $G$-Brownian motion}
\blue{
The techniques developed in Section \ref{sec:results} can also be applied to derive a Feynman-Kac formula when the underlying is a $G$-It\^{o} process $X^{r,x}$ driven by a $d$-dimensional $G$-Brownian motion. This allows to apply our results for more complex applications such as market models with multiple risk factors or stochastic discounting factors. To illustrate how to generalize our techniques we consider $G$-It\^{o} processes $X^{\bar{t},x}$ and $r^{\bar{t},r}$ representing a risky asset and a stochastic interest rate, respectively.\\
For fixed $ \bar{t}\in [0,T]$ and $x \in \mathbb{R}, r >0$ consider the $G$-It\^o processes $X^{\bar{t},x}=(X_t^{\bar{t},x})_{t \in [\bar{t},T]}, r^{\bar{t},r}=(r_t^{\bar{t},r})_{t \in [\bar{t},T]}$, given by \begin{equation}
		X_t^{\bar{t},x}=x + \int_{\bar{t}}^t f(s,X_s^{\bar{t},x})d{s} + \sum_{i,j=1}^d\int_{\bar{t}}^t g^{ij}(s,X_s^{\bar{t},x}) d\langle B^i,B^j \rangle_s + \sum_{i=1}^d \int_{\bar{t}}^t h^i(s,X_s^{\bar{t},x})dB_s^i, \quad \bar{t} \leq t \leq T, \label{G-SDE_NewAssetDDimesnisons}
	\end{equation}
	and
	\begin{equation}
		r_t^{\bar{t},r}=r+\int_{\bar{t}}^t a(s,r_s^{\bar{t},r})d{s} + \sum_{i,j=1}^d\int_{\bar{t}}^t b^{ij}(s,r_s^{\bar{t},r}) d\langle B^i,B^j \rangle_s + \sum_{i=1}^d \int_{\bar{t}}^t c^i(s,r_s^{\bar{t},r})dB_s^i, \quad \bar{t} \leq t \leq T, \label{G-SDE_NewShortRateDDimesnisons}
	\end{equation}
where $B=(B_t)_{t \in [0,T]}$ is a $d$-dimensional $G$-Brownian motion and $f,g^{ij},h^{i},a,b^{ij},c^{i}:[0,T] \times \mathbb{R} \to \mathbb{R}$ are deterministic functions such that $f(\cdot,x), g^{ij}(\cdot,x), h^{i}(\cdot,x),a(\cdot,x), b^{ij}(\cdot,x), c^{i}(\cdot,x)$ are continuous in $t$ for every $x \in \mathbb{R}$ and $f(t,\cdot), g^{ij}(t, \cdot), h^i(t,\cdot), a(t,\cdot), b^{ij}(t, \cdot), c^i(t,\cdot)$ are Lipschitz-continuous functions for every $t \in [0,T]$, $i,j=1,...,d$. 
\begin{asum}\label{asum:PayoffDDimensionalRiskyAssetShortRate}
For all $i,j=1,...,d,$ the functions $a,b^{ij},c^i$ belong to the space $C^{1,2}([0,T] \times (0,\infty))$ and $f,g^{ij},h^i$ belong to the space $C^{1,2}([0,T] \times \mathbb{R})$. Moreover, for all $i=1,...,d$ $h^i, c^i$ are bounded away from zero on every subset $\lbrace (t,y) \in [0,T] \times \mathbb{R}: y \geq a \rbrace, a >0,$ and $c^i(t,x)>0, c^i_{xx}(t,x) \geq 0$ for all $(t,x) \in [0,T] \times \mathbb{R}^+$. Furthermore, $h^i(t,x) > 0,h^i_{xx}(t,x) \geq 0 $ for all $(t,x) \in [0,T] \times \mathbb{R}$.
\end{asum}
 \begin{asum} \label{asum:PayoffFunctionGeneralized}
 	Let ${\varphi} \in C([0,T] \times \mathbb{R} \times \mathbb{R}^+)$ be bounded by a constant $M_0$ or with polynomial growth of order less or equal $L$ for $L\geq 2$, i.e., there exists a constant $\tilde{C}(T)$ depending only on the the final time $T$ such that it holds
 	$$
 	\vert \varphi(t,x,r) \vert \leq {C}(T)\left( 1 + \vert x \vert^L + r^L \right)
 	$$
 	for any $(t,x,r) \in [0,T] \times \mathbb{R} \times \mathbb{R}^+$.
 	\end{asum}
 We are now interested in evaluating the $G$-conditional expectation
 \begin{equation*}
\hat{\mathbb{E}}_{\bar{t}} \left[ \varphi \left(T,X_T^{\bar{t},x},r_T^{\bar{t},r} \right)e^{-\int_{\bar{t}}^{T}r_s^{\bar{t},r} ds}\right].
\end{equation*} 
Define the function ${F}:[0,T ] \times \mathbb{R} \times \mathbb{R}^+ \times \mathbb{R}^7 \to \mathbb{R}$ by
 \begin{align}
&{F}(t,x,r,u,u_x,u_r,u_{xx},u_{rr},u_{xr}, u_{rx}):=   f(t,x)u_x+ a^{\epsilon}(t,r)u_r -u\violet{r}\nonumber \\
&\quad+2 G\bigg(u_x g(t,x) +u_r b(t,x) + \frac{1}{2}\big[ u_{xx} \violet{h(t,x)}(h(t,x))^T + u_{rr}c(t,r)(c(t,r))^T +( u_{rx}+u_{xr})h(t,x)(c(t,r))^T\big]\bigg),  \label{eq:DefinitionNonlinearityDDimension}
 \end{align} 
where $g(t,x):=(g^{ij}(t,x))_{i,j=1,...d} \in \mathbb{R}^{d \times d}, b(t,r):=(b^{ij}(t,r))_{i,j=1,...d} \in \mathbb{R}^{d \times d}$ and $h(t,x):=(h^{i}(t,x))_{i=1,...d} \in \mathbb{R}^d, c(t,r):=(c^{i}(t,r))_{i=1,...d} \in \mathbb{R}^d$. Here, $G: \mathbb{S}(d) \to \mathbb{R}$
\begin{equation} \label{eq:DefinitionGeneralG}
	G(A):=\frac{1}{2} \sup_{\gamma \in \Theta} \text{tr}[A\gamma],
\end{equation}
where $\Theta \subset \mathbb{S}(d)$ is bounded, closed and convex such that $G$ is non-degenerate.
\begin{asum} \label{asum:ExpectationStoppingTimeDDimensionalBounded}
For fixed $\bar{t} \in [0,T]$, $\epsilon \in (0,1)$ and $({x},{r}) \in (-\epsilon^{-1},\epsilon^{-1}) \times (\epsilon, \epsilon^{-1})$, define the stopping times 
	\begin{align}
		&\tau_{\epsilon}^{1,\bar{t},{x}}:= \inf \left \lbrace t \in [\bar{t},T]: X_t^{\bar{t},{x}} \geq \epsilon^{-1}\right \rbrace \wedge T, \quad \tau_{\epsilon}^{2,\bar{t},{x}}:= \inf \left \lbrace t \in [\bar{t},T]: X_t^{\bar{t},{x}} \leq - \epsilon^{-1}\right \rbrace \wedge T, \label{eq:DefinitionRandomTimesAssumption1}\\
		& \tau_{\epsilon}^{3,\bar{t},{r}}:= \inf \left \lbrace t \in [\bar{t},T]: r_t^{\bar{t},{r}} \geq \epsilon^{-1}\right \rbrace \wedge T, \quad \tau_{\epsilon}^{4,\bar{t},{r}}:= \inf \left \lbrace t \in [\bar{t},T]: r_t^{\bar{t},{r}} \leq \epsilon \right \rbrace \wedge T. \label{eq:DefinitionRandomTimesAssumption2}
	\end{align}
We assume that \begin{align*}
	\sup_{\bar{t} \in \left[ \bar{t}_1,\bar{t}_2\right]}\hat{\mathbb{E}}_{\bar{t}} \left[ \textbf{1}_{\lbrace \tau_{\epsilon}^{1,\bar{t},{x}_2} < T \rbrace}  \right] \to 0 \quad \text{ and } \quad  \sup_{\bar{t} \in \left[ \bar{t}_1,\bar{t}_2\right]} \hat{\mathbb{E}}_{\bar{t}} \left[ \textbf{1}_{\lbrace \tau_{\epsilon}^{2,\bar{t},{x}_1} < T \rbrace}  \right]  \to 0
\end{align*}
for any ${x}_1<{x}_2, 0< \bar{t}_1 < \bar{t}_2 < T$, and
\begin{align*}
	\sup_{\bar{t} \in \left[ \bar{t}_1,\bar{t}_2\right]}\hat{\mathbb{E}}_{\bar{t}} \left[ \textbf{1}_{\lbrace \tau_{\epsilon}^{3,\bar{t},{r}_2} < T \rbrace}  \right] \to 0 \quad \text{ and } \quad \sup_{\bar{t} \in \left[ \bar{t}_1,\bar{t}_2\right]} \hat{\mathbb{E}}_{\bar{t}} \left[ \textbf{1}_{\lbrace \tau_{\epsilon}^{4,\bar{t},{r}_1} < T \rbrace}  \right] \to 0
\end{align*}
for any $0<{r}_1<{r}_2, 0< \bar{t}_1 < \bar{t}_2 < T$.
\end{asum}
\begin{theorem} \label{theorem:StabilityHigherBounded}
Let ${\varphi} \in C([0,T] \times \mathbb{R} \times \mathbb R^+)$ satisfy Assumption \ref{asum:PayoffFunctionGeneralized}. 
	Given $(\bar{t}, {x},{r}) \in [0,T] \times \mathbb{R} \times (0,\infty)$, assume that the processes  $r^{\bar{t},{r}}=(r_t^{\bar{t},{r}})_{t \in [\bar{t},T]}$ in \eqref{G-SDE_NewShortRateDDimesnisons} is quasi-surely strictly positive and $X^{\bar{t},{x}}=(X_t^{\bar{t},{x}})_{t \in [\bar{t},T]}$ is given in \eqref{G-SDE_NewAssetDDimesnisons}. Let Assumptions \ref{asum:PayoffDDimensionalRiskyAssetShortRate}, \ref{asum:ExpectationStoppingTimeDDimensionalBounded} be satisfied.
Define the function $u:[0,T] \times \mathbb{R} \times (0,\infty) \to \mathbb{R}$ by
	\begin{equation} \label{eq:DefinitionViscositySolutionDDimensionalBounded}
	u(\bar{t},{x},{r}):= \hat{\mathbb{E}}_{\bar{t}} \left[ {\varphi}\left(T,X_{T}^{\bar{t},{x}},r_T^{\bar{t},{r}}\right)e^{-\int_{\bar{t}}^T r_s^{\bar{t},{r}} ds }\right].	
	\end{equation}
Then u is a viscosity solution of the PDE
\begin{align}
	u_t+ F(t,x,r,u,u_x,u_r,u_{xx},u_{rr},u_{xr},u_{rx})=0 \ \quad & \text{for } (t,x,r) \in Q \label{eq:LimitPDEDDimensional1} \\
	u={\varphi} \quad & \text{for } (t,x,r) \in \partial' Q,\label{eq:LimitPDEDDimensional2}
\end{align}
\end{theorem}
\begin{remark}
	For the technical details and the proofs, we refer to \cite{oberpriller_phd}.
	\end{remark}}

\bibliography{Generalized_Feynman_Kac_formula_under_volatility_uncertainty.bib}
\bibliographystyle{plain}

\end{document}